\documentclass[a4paper,reqno,11pt]{article}

\usepackage[margin=1in]{geometry}
\usepackage{amssymb}
\usepackage{amsmath}
\usepackage{amsthm}
\usepackage{authblk}
\usepackage{marvosym}

\usepackage{algorithmic}

\usepackage{algorithm}

\usepackage[normalem]{ulem}

\usepackage[unicode=true]{hyperref}
\hypersetup{
    colorlinks,
    linkcolor={red!50!black},
    citecolor={blue!50!black},
    urlcolor={blue!80!black}
}

\usepackage{times}
\usepackage{todonotes}
\usepackage{bbm}
\usepackage{paralist}
\usepackage[capitalise]{cleveref}

\usepackage{thmtools, thm-restate}
\declaretheorem[name=Theorem]{theorem}
\declaretheorem[name=Lemma,sibling=theorem]{lemma}

\declaretheorem[name=Proposition,sibling=theorem]{proposition}
\declaretheorem[name=Definition,sibling=theorem,style=definition]{definition}
\declaretheorem[name=Corollary,sibling=theorem]{corollary}
\declaretheorem[name=Observation,sibling=theorem]{observation}
\declaretheorem[name=Claim,sibling=theorem]{claim}

\newcommand{\conv}{\operatorname{conv}}

\newcommand{\poly}{\operatorname{poly}}
\newcommand{\R}{\mathbb{R}}
\newcommand{\Z}{\mathbb{Z}}
\newcommand{\Q}{\mathbb{Q}}
\newcommand{\N}{\mathbb{Z}_{\geq 0}}
\newcommand{\supp}{\operatorname{supp}}
\newcommand{\ocp}{\operatorname{ocp}}
\newcommand{\oct}{\operatorname{oct}}

\newcommand{\zero}{\mathbf{0}}
\newcommand{\one}{\mathbf{1}}

\newcommand{\EP}{Erd\H{o}s-P\'osa}
\newcommand{\surf}{\mathbb{S}}

\newcommand{\eps}{\varepsilon}

\newcommand{\eg}{\operatorname{eg}}
\newcommand{\holes}{\operatorname{h}}
\newcommand{\Del}{\operatorname{Del}}
\newcommand{\del}{\operatorname{del}}
\newcommand{\full}{\operatorname{f}}
\newcommand{\val}{\operatorname{val}}

\newcommand{\bd}{\operatorname{bd}}
\newcommand{\inter}{\operatorname{int}}

\newcommand{\apexbd}{g} 
\newcommand{\vortexbd}{g}
\newcommand{\adhesionbd}{g}
\newcommand{\genusbd}{g}
\newcommand{\resiliencebd}{\rho}

\let\le\leqslant
\let\ge\geqslant
\let\leq\leqslant
\let\geq\geqslant

\usetikzlibrary{intersections,calc,arrows}

\title{Integer programs with bounded subdeterminants and two nonzeros per row\thanks{A preliminary version of this paper appeared in \emph{Proceedings of the 2021 IEEE 62nd Annual Symposium on Foundations of Computer Science (FOCS)}, pages 13--24. Denver, CO, USA, 2022.}}

\author[1]{Samuel Fiorini}
\author[1]{Gwena\"el Joret}
\author[2]{Stefan Weltge}
\author[1]{Yelena Yuditsky}
\affil[1]{\itshape\small Universit\'e libre de Bruxelles, Brussels, Belgium}
\affil[2]{\itshape\small Technische Universit\"at M\"unchen, Munich, Germany}
\affil[ ]{}
\affil[1]{{\color{gray} \footnotesize \texttt{\{samuel.fiorini,gwenael.joret,yelena.yuditsky\}@ulb.be}}}
\affil[2]{{\color{gray} \footnotesize \texttt{weltge@tum.de}}}
\begin{document}

\maketitle

\begin{abstract}
We give a strongly polynomial-time algorithm for integer linear programs defined by integer coefficient matrices whose subdeterminants are bounded by a constant and that contain at most two nonzero entries in each row.
The core of our approach is the first polynomial-time algorithm for the weighted stable set problem on graphs that do not contain more than $k$ vertex-disjoint odd cycles, where $k$ is any constant.
Previously, polynomial-time algorithms were only known for $k=0$ (bipartite graphs) and for $k=1$.

We observe that integer linear programs defined by coefficient matrices with bounded subdeterminants and two nonzeros per column can be also solved in strongly polynomial-time, using a reduction to $b$-matching.
\end{abstract}

\section{Introduction}\label{sec:intro}

Many discrete optimization problems can be naturally formulated as integer (linear) programs of the form
\begin{equation} \label{eq:IP} \tag{IP}
\max \, \{ w^\intercal x : Ax \le b,\ x \in \Z^n \},
\end{equation}
where $A \in \Z^{m \times n}$, $b \in \Z^m$, $w \in \Z^n$.
While general integer programs cover NP-hard problems, polynomial-time algorithms have been developed for various interesting classes.
Prominent examples include Papadimitriou's algorithm~\cite{Papadimitriou} (recently improved by Eisenbrand and Weismantel~\cite{EisenbrandWeismantel}), Lenstra's algorithm~\cite{Lenstra} (improved by Kannan~\cite{Kannan} and Dadush~\cite{Dadush}), and several algorithms for block-structured integer programs, see, \emph{e.g.}~\cite{hemmecke2013n,jansen2020near,cslovjecsek2021block}.
These methods can be found at the core of many approaches for combinatorial optimization problems.
Illustrative examples are algorithms by Goemans and Rothvo\ss~\cite{GoemansRothvoss} and Jansen, Klein, Maack, and Rau~\cite{JansenKMR19} for bin packing and scheduling problems, respectively.

One of the most fundamental efficiently solvable classes consists of integer programs of the form~\eqref{eq:IP} whose coefficient matrix $A$ is totally unimodular, \emph{i.e.}, each square submatrix of $A$ has a determinant within $\{-1,0,1\}$.
In this case, an optimal solution of~\eqref{eq:IP} can be easily obtained from its linear programming relaxation $\max \{ w^\intercal x \mid Ax \leqslant b,\ x \in \R^n \}$. 
Such integer programs include maximum flows, minimum cost flows, maximum matchings in bipartite graphs, or maximum stable sets in bipartite graphs, and we refer to Schrijver's books~\cite{SchrijverLPIP,SchrijverCombOpt} for more background and examples.

Recently, Artmann, Weismantel, and Zenklusen~\cite{Artmann} showed that integer programs of the form~\eqref{eq:IP} can be still solved in (strongly) polynomial time when $A$ is \emph{totally 2-modular}.
Here, we say that an integer matrix is \emph{totally $\Delta$-modular} if each of its subdeterminants (determinants of square submatrices of any size) is within $\{-\Delta,\dots,\Delta\}$.
In view of other work on linear and integer programs with bounded subdeterminants~\cite{Tardos86,VC09,BDEHN14,DF94,EV17,Paat}, it is tempting to believe that integer programs defined by totally $\Delta$-modular matrices can be solved in polynomial time whenever $\Delta$ is a constant.
While this question is still open (even for $\Delta = 3$), we prove the following result.

\begin{theorem} \label{thmMain}
    For every integer $\Delta \ge 0$ there exists a strongly polynomial-time algorithm for solving integer programs of the form~\eqref{eq:IP} where $A$ is totally $\Delta$-modular and contains at most two nonzero entries in each row or in each column.
\end{theorem}
We remark that if we do not restrict the coefficient matrices to be totally $\Delta$-modular, then the above class covers NP-hard problems such as stable set (two nonzero entries per row, see below) or integer knapsack (two nonzero entries per column, see~\cite[\S A6]{GareyJohnson}).

Our main motivation for studying the above class of integer program stems from the fact that it captures the weighted \emph{stable set problem in graphs that do not contain $k+1$ vertex-disjoint odd cycles}, where $k\geq 0$ is a fixed constant.
In fact, given an undirected graph $G$ and vertex weights $w : V(G) \to \R$, the weighted stable set problem can be formulated as maximizing $\sum_{v \in V(G)} w(v) x_v$ subject to $x_v \in \Z$, $0 \le x_v \le 1$ for each vertex $v \in V(G)$ and $x_u + x_v \le 1$ for each edge $uv \in E(G)$.
Bringing this into the form~\eqref{eq:IP}, the largest subdeterminant of the coefficient matrix is equal to $2^{\ocp(G)}$, where $\ocp(G)$ is the largest number of vertex-disjoint odd cycles in $G$, called the \emph{odd cycle packing number} of $G$ (this is a well-known fact proved, \emph{e.g.}, in~\cite{GKS95}). Thus, the following result is a direct consequence of Theorem~\ref{thmMain}.

\begin{theorem} \label{thmStableSet}
    For every integer $k \ge 0$ there exists a strongly polynomial-time algorithm for solving the weighted stable set problem in graphs that do not contain $k+1$ vertex-disjoint odd cycles, that is, graphs $G$ with $\ocp(G) \leq k$.
\end{theorem}

Determining the complexity of the stable set problem for graphs with bounded odd cycle packing number was an open problem pioneered by Bock, Faenza, Moldenhauer, and Ruiz-Vargas~\cite{BFMR14}.
Note that $\ocp(G) = 0$ holds if and only if $G$ is bipartite, in which case the stable set problem is polynomially solvable.
The result of Artmann \emph{et al.}~\cite{Artmann} led to the first polynomial-time algorithm for the case $\ocp(G) = 1$ and it was recently shown that the corresponding stable set polytopes admit quadratic-size extended formulations~\cite{ocp1}.

Considering the family of graphs $G$ with $\ocp(G) \le k$ for some constant $k \ge 2$, a polynomial-time algorithm has been obtained recently under the additional assumption that the genus of $G$ is bounded by a constant~\cite{ocpgenus}.
Without the latter assumption, it was known that the stable set problem admits a PTAS.
In fact, the stable set problem in graphs $G$ excluding an odd $K_t$ minor admits a PTAS, as shown by Tazari~\cite{tazari12} using structural results by Demaine, Hajiaghayi and Kawarabayashi~\cite{DHK10}.
Notice that if $G$ has an odd $K_t$ minor, then in particular $\ocp(G) \geq
\lfloor t/3 \rfloor$. Moreover, Bock \emph{et al.}~\cite{BFMR14} state that 
the problem even admits a PTAS if $ \ocp(G) = O(\sqrt{|V(G)| / \log 
\log |V(G)|}) $. However, previous to our work, an exact 
polynomial-time algorithm was not even known for graphs $G$ with 
$\ocp(G) = 2$.

While the statement of Theorem~\ref{thmMain} directly implies Theorem~\ref{thmStableSet}, the proof of the latter result is actually at the core of our approach.
In fact, using proximity results by Cook, Gerards, Schrijver, and Tardos~\cite{CGST86} we show that integer programs defined by a totally $\Delta$-modular coefficient matrix with at most two nonzero entries in each \emph{row} can be reduced to the weighted stable set problem on graphs with odd cycle packing number at most $\log_2 \Delta$. The reduction is efficient provided that $\Delta$ is a constant.
Our proof of Theorem~\ref{thmStableSet} is based on new structural results about graphs with bounded odd cycle packing number, Theorems~\ref{thm:structure_bounded_OCP} and \ref{thm:structure_bounded_OCP_refined}, as well as recent connections between stable sets in surface embedded graphs $G$ and integer circulations in the dual of $G$ established in~\cite{ocpgenus}. 
Let us remark that, while our algorithm in Theorem~\ref{thmStableSet} is efficient for fixed $k$, the dependency of its running time on $k$ is huge: It is of the form $O(|V(G)|^{c})$ for some enormous constant $c=c(k)$. 
This is because it relies on the Excluded Minor Structure Theorem of Robertson and Seymour~\cite{RS-GraphMinorsXVI-JCTB03}, which involves such enormous constants.   
The dependency in $\Delta$ of the running time of the algorithm in Theorem~\ref{thmMain} is similarly huge. 
Whether fixed-parameter algorithms exist for these two problems when parameterized by $k$ and $\Delta$, respectively, is left as an open problem.

In the case matrix $A$ has at most two nonzero entries in each \emph{column}, we give an efficient reduction to the subcase where all entries of $A$ are in $\{-1,0,+1\}$, again provided that the original coefficient matrix is totally $\Delta$-modular for some constant $\Delta$.
For such matrices it is known that the corresponding integer program can be solved in strongly polynomial time~\cite[Thm.~36.1]{SchrijverCombOpt}.

\section{Overview} \label{sec:prelim}

Consider an integer program of the form~\eqref{eq:IP}, where $A$ is totally $\Delta$-modular for some constant $\Delta \ge 1$ and has at most two nonzeros per row or per column.
In both cases, we will employ a proximity result by Cook \emph{et al.}~\cite{CGST86}, see Theorem~\ref{thm:Cook_et_al}, to reduce to integer programs with a particular structure.
In the case of at most two nonzeros per row, we will obtain the integer programming formulation of a weighted stable set problem on a graph $G$ with $\ocp(G) \leq k := \lfloor \log_2 \Delta \rfloor$.
In the case of at most two nonzeros per column, we will obtain an integer program where the two nonzero entries are within $\{-1,+1\}$.
The latter case is considerably simpler and can be treated using a reduction due to Tutte and Edmonds, see Section~\ref{sec2nonzerospercolumn}.

In fact, the main body of this work is concerned with a strongly polynomial algorithm for the weighted stable set problem on graphs with bounded odd cycle packing number.
We start in Section~\ref{sec:struct} by providing first structural results on graphs $G$ with $\ocp(G) \le k$, see Theorem~\ref{thm:structure_bounded_OCP}.
Our findings are based on various deep results on graph minors.

An {\em odd cycle transversal} of a graph $G$ is a vertex subset meeting all odd cycles of $G$. The minimum size of an odd cycle transversal is called the {\em odd cycle transversal number} of $G$, and is denoted $\oct(G)$. Trivially, $\ocp(G) \leq \oct(G)$ for every graph $G$. Hence, a small odd cycle transversal number always guarantees a small odd cycle packing number. However, as was originally observed by Lov\'asz and Schrijver (see~\cite{seymour95}), there is no counterpart to this fact: a small odd cycle packing number \emph{does not} guarantee a small odd cycle transversal number. In fact, there are $n$-vertex graphs with $\ocp(G) = 1$ and $\oct(G)$ as big as $\Omega(\sqrt{n})$. These graphs are known as \emph{Escher walls}, see Fig.~\ref{fig:Escher_wall} (see Section~\ref{subsec:walls} for the precise definition).

\begin{figure}[h]
    \centering
    \includegraphics[width=0.6\textwidth]{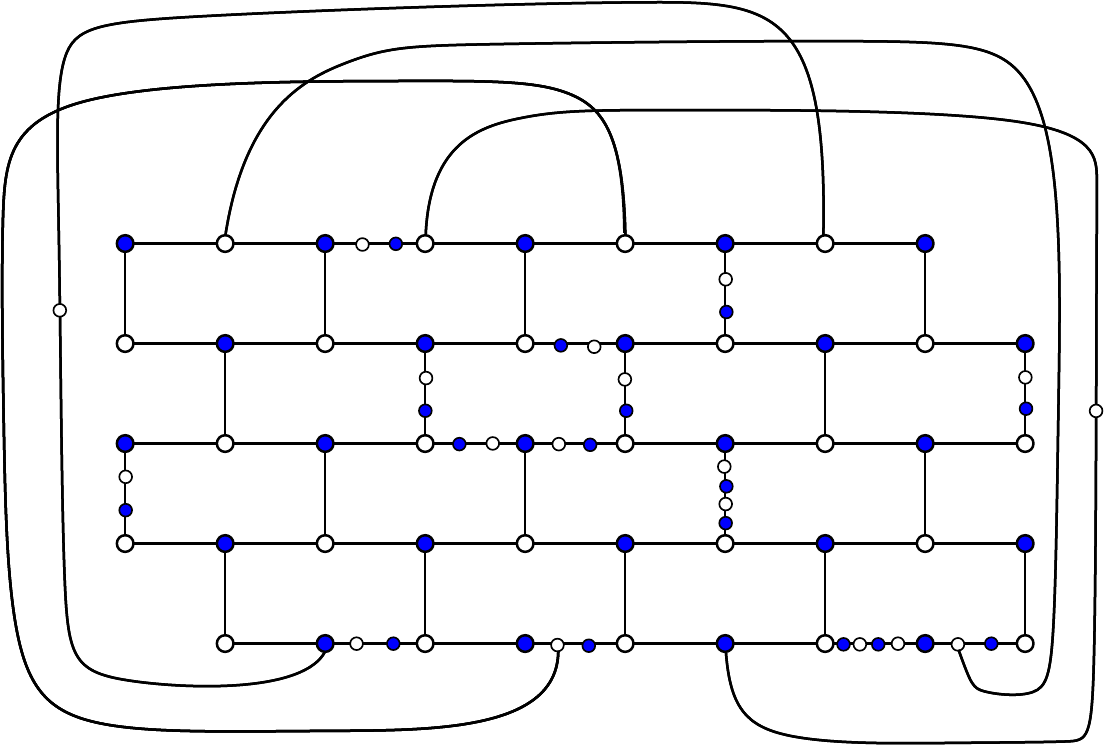}
    \caption{\label{fig:Escher_wall}An Escher wall of height $5$. 
    (The coloring highlights the bipartition of the wall.)}
\end{figure}

The starting point for our first structural result is a theorem of Reed~\cite{Reed99} stating that every graph with $\ocp(G) \leq k$ and $\oct(G) \geq f(k,h)$ contains an Escher wall of height $h$, where $f$ is a computable function. This allows us to assume that our graph $G$ contains a large Escher wall, since graphs with small odd cycle transversals are easily dealt with. Indeed, if $\oct(G)$ is bounded by a constant, then one can compute a minimum-size odd cycle transversal $X$ of $G$ efficiently using an algorithm of Kawarabayashi and Reed~\cite{KR_SODA10}, and for each stable set $S'$ in $X$, one can compute efficiently a maximum weight stable set $S''$ in the bipartite graph $G-X-N(S')$, where $N(S')$ denotes the set of vertices outside $S'$ having a neighbor in $S'$. 
Then, choosing the pair $S', S''$ whose union $S' \cup S''$ has maximum weight gives an optimal solution to the maximum weight stable set problem on $G$.

Then, using a result of Geelen, Gerards, Reed, Seymour, and Vetta~\cite{GGRSV09}, we show that $G$ cannot contain a large complete graph minor that is ``well attached'' to the Escher wall since otherwise we would obtain a packing of $k+1$ odd cycles in $G$. 
Next, we use the celebrated Excluded Minor Structure Theorem of Robertson and Seymour~\cite{RS-GraphMinorsXVI-JCTB03}, and more precisely a recent version of the theorem due to Kawarabayashi, Thomas, and Wollan~\cite{KTW20}, which comes with explicit bounds and an efficient algorithm. 
This allows us to conclude that $G$ has a ``near embedding'' in a surface $\surf$ of bounded genus: after removing a bounded number of vertices, we can embed a subgraph $G_0 \subseteq G$ in $\surf$ in such a way that $G_0$ essentially contains a large Escher wall, and that the rest of $G$ sits in a bounded number of ``large vortices'' and in a (possibly large) number of ``small vortices'', where the number of small vortices is unbounded in terms of $k$ (but bounded in the size of the graph). A vortex is a pair of a (sub)graph and a linearly ordered subset of some vertices of the graph. Each vortex (large or small) can be drawn, with edge crossings, in a disk contained in $\surf$ (each vortex gets its own disk). The interiors of these disks are mutually disjoint, and disjoint from $G_0$. Each large vortex has bounded ``depth'', and each small vortex has at most $3$ vertices in common with $G_0$. See Fig.~\ref{fig:struct_thm} for an illustration.

\begin{figure}[ht]
\centering
\includegraphics[width=10cm,clip=true,trim=3cm 3.5cm 6cm 5cm]{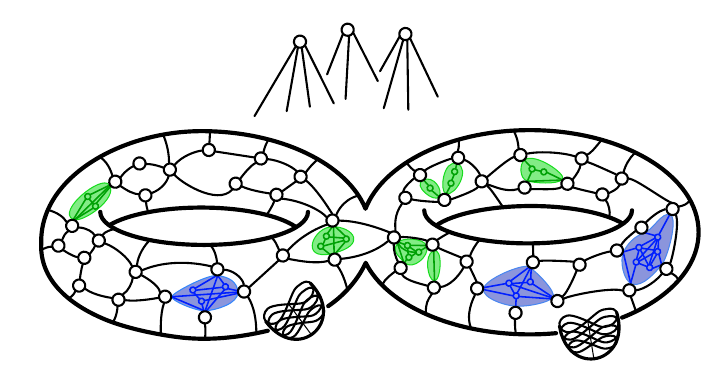}
\caption{Theorem~\ref{thm:structure_bounded_OCP} illustrated. The surface part $G_0$ is shown in black. Small vortices are depicted in green, and large vortices in blue.}
\label{fig:struct_thm}
\end{figure}

Gearing toward the stable set problem, we prove a more precise structural result in Section~\ref{sec:clean_near_embeddings}. We elaborate on the differences between our first structural result, Theorem~\ref{thm:structure_bounded_OCP}, and the refined result, Theorem~\ref{thm:structure_bounded_OCP_refined}. Instead of assuming that $\oct(G)$ is large, we assume the more restrictive condition that $G$ has high ``resilience'', namely, that there is no small vertex subset $X$ such that $\ocp(H) < \ocp(G)$ holds for every connected component $H$ of $G - X$. 
This implies that, for every small vertex subset $X$, one connected component $H$ of $G - X$ satisfies $\ocp(H) = \ocp(G)$, while all other connected components of $G - X$ are bipartite. 

A very useful consequence of having high resilience is that all the small vortices are bipartite. (This is because each small vortex can be separated from $G_0$ by removing a small vertex set $X$, and the connected component of $G-X$ containing most of the Escher wall that ``lived'' in $G_0$ is not bipartite, implying that all other connected components of $G-X$ are bipartite.) 
Among the further properties gained in the refined result, we can make the distance (in the vertex-face incidence graph) between any two large vortices, as well as the ``face-width'' (also known as ``representativity'') of $G_0$, arbitrarily large. This is based on results of Diestel, Kawarabayashi, M\"uller and Wollan \cite[Section 5]{DiestelKMW12}. Moreover, we can make sure that every odd cycle in $G_0$ defines a M\"obius strip in $\surf$. This implies in particular that all the faces of $G_0$ are even, and relies on previous work~\cite[Theorem 10.9]{ocpgenus}. Finally, we prove that each large vortex can be made bipartite in a strong sense: we can assume that every such vortex is bipartite, even when it is augmented with the boundary of corresponding face of $G_0$, see Theorem~\ref{thm:bipartite_enlarged}. All these properties are achieved at the cost of deleting a bounded number of extra vertices from $G$. 

We point out that the near embedding guaranteed by each one of  Theorems~\ref{thm:structure_bounded_OCP} and \ref{thm:structure_bounded_OCP_refined} can be constructed in polynomial time.

In Section~\ref{sec:preprocessing}, we explain the final preprocessing steps that we perform on the given instance of the maximum weight stable set problem. The first step is to ensure high resilience. If the resilience is small, we can find a suitable vertex subset $X$ in polynomial time by combining a brute-force search and the FPT algorithm of Kawarabayashi and Reed~\cite{KR_STOC10} to compute the odd cycle packing number of each component of $G - X$. Once $X$ is found, one can solve the given weighted stable set instance in strongly polynomial time by recursing on $k$. The second step is to translate the stable set problem to a problem on the \emph{edges} of $G$, a key idea from the recent works~\cite{ocpgenus,ocp1}. This is done through solving the linear programming (LP) relaxation. The LP relaxation can in fact be solved in strongly polynomial time, using Tardos's algorithm~\cite{Tardos86}. The third step is to collapse all small vortices to gadgets with edge costs, see Lemma~\ref{lemhd8sg7}. Each gadget (a path or a subdivided $Y$ graph) can be added to the surface part $G_0$ without creating edge crossings. This involves the solution of a small (constant) number of weighted stable set instances on each small vortex, which again can be performed in strongly polynomial time. The fourth and last step is to slightly modify $G_0$ in order to ensure that each face of $G_0$ is bounded by a cycle.

In Section~\ref{sec:DP}, we present our main algorithm for solving instances of the weighted stable set problem on graphs $G$ with $\ocp(G) \leq k$, where $k$ is any constant, in strongly polynomial time. Our algorithm exploits the structure from the three previous sections, which can be assumed without loss of generality. Let $G_1$, \ldots, $G_t$ denote the subgraphs of $G$ obtained from the large vortices by adding the boundary of the corresponding faces of $G_0$. Each such boundary is an even cycle, which equals the intersection of $G_0$ and the corresponding $G_i$. Recall that the number $t$ of large vortices is bounded.

Viewing any solution to the problem as an edge subset, we can decompose the solution into several parts, one global part contained in $G_0$ and several local parts, each contained in some $G_i$ for $i \in [t] := \{1,\ldots,t\}$. From~\cite[Proposition 3.1]{ocpgenus}, we know that the part of the solution in $G_0$ is a 0/1 circulation in a well-chosen orientation of the dual graph $G_0^*$. This circulation satisfies $g$ extra constraints, where $g := \eg(\surf)$ is the Euler genus of $\surf$. Topologically, these $g$ extra constraints mean that the circulation from our solution is homologous to the all-one circulation. 

The main result of \cite{ocpgenus} (see also Morell, Seidel and Weltge~\cite{MSW21}) tells us how to solve the problem on the surface part $G_0$ alone. If we knew that the feasible solution avoids the boundary cycle of each $G_i$, $i \in [t]$, then we would be done after performing a single homologous circulation computation on $G_0^*$, and taking the empty (zero) solution within each $G_i$, $i \in [t]$. 
(This is because $G_i$ is bipartite; in terms of stable sets, this amounts to selecting the side of the bipartition of $G_i$ that has maximum weight.) 
Unfortunately, this is in general not the case, and we cannot simply ignore the subgraphs $G_1$, \ldots, $G_t$. The global solution may interact nontrivially with some local solutions through the boundary cycle of the corresponding $G_i$.

In addition to being bipartite, the subgraphs $G_i$, $i \in [t]$ have a particularly nice structure in terms of their cutsets, which is known as ``linear decomposition'' of bounded ``adhesion''. This means that for each vertex $v$ on the boundary of $G_i$, we have a ``bag'' $X(v) \subseteq V(G_i)$ containing $v$. For each $z \in V(G_i)$, the set of boundary vertices $v \in V(G_i)$ with $z \in X(v)$ form a path contained in the boundary of $G_i$. For $z, z' \in V(G_i)$, there can be an edge between $z$ and $z'$ only if the corresponding paths intersect. Moreover, if two boundary vertices $v, v' \in V(G_i)$ are consecutive on the boundary then $X(v) \cap X(v')$ has a bounded number of vertices. Using the bags $X(v)$ and their properties, we can define a bounded size ``cutset'' $Y(e)$ for each boundary edge $e=vv' \in E(G_i)$.

The strategy we employ is to consider an optimal solution and, using an uncrossing step, derive from the corresponding global solution a directed graph on $t$ vertices that is embedded in $\surf$. We call this a ``sketch'' of the optimal solution. This sketch describes how the global solution connects the different subgraphs $G_i$, $i \in [t]$. Despite the fact that the sketch has a bounded number of vertices, it may have an unbounded number of directed edges (parallel or anti-parallel edges as well as loops). Hence the algorithm cannot possibly guess the whole sketch. Instead, the algorithm guesses the edges of the sketch one by one. Moreover, it focuses on a single face of the subgraph of the sketch (formed by the edges added so far) at a time. Each time an edge of the sketch is guessed, we also guess the intersection of the optimal stable set with the two corresponding cutsets $Y(e)$ and $Y(e')$, where $e$ is a boundary edge of some $G_i$, $i \in [t]$ and $e'$ is a boundary edge of some $G_{i'}$, $i' \in [t]$ (possibly $i = i'$). The final technical hurdle that we need to overcome is that the sketch might have some face that has an unbounded number of directed edges (loops). The cure we propose to that is to only remember the relevant part of the face boundary. Each time a part of the boundary is erased from our records, we have to solve some maximum weight stable set problem ``between'' two cutsets inside the same $G_i$, $i \in [t]$. We obtain a dynamic program that can be solved in strongly polynomial time, see Theorem~\ref{thm:DP}.

To the best of our knowledge, our Theorems~\ref{thm:structure_bounded_OCP} and \ref{thm:structure_bounded_OCP_refined} have just the right ingredients to permit an efficient algorithm for solving the stable set problem on graphs with bounded odd cycle packing number.
Let us mention that some of the ideas underlying these theorems are already present in the works of Kawarabayashi and Reed~\cite{KR_STOC10, KR_SODA10} giving FPT algorithms for odd cycle packing and odd cycle transversal.

\section{Integer programs with two nonzero entries per row} \label{sec:IP2_row}

Let us consider the integer program~\eqref{eq:IP} under the assumption that $A$ is totally $\Delta$-modular for some integer constant $\Delta \ge 1$, and that every row of $A$ has at most two nonzero entries.

We mention a related result due to Hochbaum, Megiddo, Naor, and Tamir~\cite{HMNT93} who gave a 2-approximation for the class of integer programs with at most two nonzero entries per row that minimize a nonnegative objective over nonnegative variables, with running time polynomial in the size of the coefficient matrix and an upper bound on each variable.
However, their techniques do not seem to be applicable to obtain an exact algorithm for our case.

In fact, we follow a different approach:
In this section, we will describe a strongly polynomial reduction to the weighted stable set problem on a graph $G$ with $\ocp(G) \le k = \lfloor \log_2 \Delta \rfloor$, consisting of four steps.

As a first step, we will employ the following proximity result by Cook \emph{et al.}~\cite{CGST86} to efficiently eliminate variables that appear with a coefficient of absolute value greater than 1.

\begin{theorem}[Cook \emph{et al.}~\cite{CGST86}] \label{thm:Cook_et_al}
Let $A$ be a totally $\Delta$-modular $m \times n$ matrix and let $b$ and $w$ be vectors such that $Ax \leqslant b$ has an integral solution and $\max \{w^\intercal x : Ax \leqslant b\}$ exists. Then for each optimal solution $\bar{x}$ to $\max \{w^\intercal x : Ax \leqslant b\}$, there exists an optimal solution $z^*$ to $\max \{w^\intercal x : Ax \leqslant b,\ x \in \Z^n\}$ with $||\bar{x} - z^*||_{\infty} \leqslant n \Delta$.
\end{theorem}

In the second step, we introduce auxiliary integer variables to obtain an equivalent integer program in which all constraints are of the form $x_i + x_j \le b_{ij}$, $x_i + x_j = b_{ij}$, or describe variable bounds.

By modifying the objective function, we are able to eliminate constraints of the second type in the third step.

Finally, we show that for integer programs of the resulting type, Theorem~\ref{thm:Cook_et_al} can be significantly strengthened, allowing us to restrict all variables attain values in $\{0,1\}$.
With some final simple processing we end up with a weighted stable set problem over some graph whose odd cycle packing number is indeed bounded by $\log_2 \Delta$.

\subsection{Reduction to coefficients in $\{-1,0,+1\}$}\label{sec:reduction_to_bidirected}

Let us start by identifying a small set of variables whose elimination yields an integer program defined by a submatrix of $A$ that has all its entries in $\{-1,0,+1\}$.
To this end, consider the following lemma.

\begin{lemma} \label{lem:IP2_reduction}
Let $A \in \Z^{m \times n}$ be totally $\Delta$-modular.
There exists a subset $J \subseteq [n]$ with $|J| \le \log_2 \Delta$ such that for each $i \in [m]$, we have $A_{i}^\intercal \in \{-1,0,+1\}^n$ or there is some $j \in J$ with $A_{ij} \ne 0$. Moreover, the set $J$ can be computed in polynomial time. 
\end{lemma}
\begin{proof}
We define $J \subseteq [n]$ iteratively, initializing $J$ with the empty set. For later use, we also define a set $I \subseteq [m]$, which is initially empty.
Suppose that there exists a row index $i \in [m]$ such that none of the two conditions in the above statement is satisfied.
Then there is a column index $j \in [n] \setminus J$ with $|A_{ij}| \ge 2$.
We add $i$ to $I$ and $j$ to $J$ and repeat the process until every row index satisfies one of the two conditions above.

It is easy to see that throughout all steps, and in particular at the end of the process, the submatrix $A'$ of $A$ induced by $I$ and $J$ is a (square) triangular matrix whose diagonal entries are integers with absolute value at least two.
We conclude $\Delta \ge |\det(A')| \ge 2^{|J|}$ and hence $|J| \le \log_2 \Delta$.
\end{proof}

Now, consider an integer program of the form~\eqref{eq:IP}, where the coefficient matrix $A$ is totally $\Delta$-modular for some constant $\Delta$ and such that $A$ has at most two nonzero entries in each row.
Let us first solve the LP relaxation $\max \{w^\intercal x : Ax \leqslant b\}$.
This linear program can be solved in strongly polynomial time using the algorithm of Tardos~\cite{Tardos86} since the absolute values of all entries in $A$ are bounded by $\Delta$.
If the LP is infeasible, then~\eqref{eq:IP} is infeasible as well and we are done.
If the LP is unbounded, then~\eqref{eq:IP} is either infeasible or unbounded.
In this case we may stop or repeat the process with $w = \zero$ to distinguish between the latter two cases.

Thus, we may assume that the LP relaxation has an optimal solution, say $\bar{x}$.
Let $J$ be as in Lemma~\ref{lem:IP2_reduction}.
By Lemma~\ref{lem:IP2_reduction} such a set $J$ can be efficiently computed.
By Theorem~\ref{thm:Cook_et_al}, we know that there exists an optimal integer solution $\bar{z}$ (if the IP is bounded) that satisfies
\[
    \bar{z}_j \in \{ z \in \Z : |\bar{x}_j - z| \le n \Delta \} =: S_j \quad \text{ for all } j \in J.
\]
Thus, for each $j \in J$ we may guess a value $z_j \in S_j$, and solve a subproblem of~\eqref{eq:IP} in which we fix $x_j = z_j$ for $j \in J$.
In this approach, the total number of subproblems (guesses) that we have to consider is
\[
    \prod_{j \in J} |S_j| \le (2 n \Delta + 1)^{|J|} \le (2 n \Delta + 1)^{\log_2 \Delta} = \poly(n).
\]
By Theorem~\ref{thm:Cook_et_al}, if~\eqref{eq:IP} is feasible, then at least one subproblem has an optimal solution that is also optimal for~\eqref{eq:IP}.
If all subproblems are infeasible, then so is~\eqref{eq:IP}.

It remains to observe that each subproblem arises from~\eqref{eq:IP} by deleting all variables indexed by $J$.
The resulting subproblem has the property that every constraint either involves at most one variable, or two variables with coefficients in $\{-1,+1\}$.
After replacing each constraint of the form $\alpha x_j \leqslant \beta$ by $x_j \leqslant \lfloor \beta / \alpha\rfloor$ if $\alpha > 0$ or $- x_j \leqslant \lfloor - \beta / \alpha\rfloor$ if $\alpha < 0$, we can in fact assume that the constraint matrix of each subproblem has \emph{all} its entries in $\{-1,0,+1\}$.
Moreover, the constraint matrix is still totally $\Delta$-modular.

\subsection{Reduction to edge constraints and variable bounds}

We have seen that it suffices to consider integer programs~\eqref{eq:IP} with a totally $\Delta$-modular matrix $A \in \{-1,0,+1\}^{m \times n}$ having at most two nonzero entries per row.
Next, we will construct an equivalent integer program together with a graph $G$ with $\ocp(G) \le \log_2 \Delta$ such that every constraint that is not a variable bound ($x_i \le u_i$ or $x_i \ge \ell_i$) is of the form $x_i + x_j \le b_{ij}$ or $x_i + x_j = b_{ij}$ for some edge $ij \in E(G)$.

As before, we may assume that the LP relaxation of the given IP has an optimal solution, which we can compute in strongly polynomial time.
By Theorem~\ref{thm:Cook_et_al}, we may then add variable bounds to obtain an equivalent integer program
\[
    \max \, \{ w^\intercal x : Ax \le b, \, x \in [\ell, u] \cap \Z^n \}
\]
for some $\ell, u \in \Z^n$.
At this point, we can assume that every row of $A$ has exactly two nonzero entries.
This means that every constraint in $Ax \le b$ is of one of the types
\[
    x_i + x_j \le \alpha_{ij} \qquad
    x_i - x_j \le \beta_{ij} \qquad
    -x_i - x_j \le \psi_{ij}
\]
with $i,j \in [n]$, $i \ne j$, $\alpha_{ij}, \beta_{ij}, \psi_{ij} \in \Z$.
We may assume that for every pair $i,j$ there is at most one constraint of every type.
For every constraint of the second type, let us introduce an auxiliary integer variable $y_{ij}$ and replace the constraint by two new (equivalent) constraints
\[
    x_i + y_{ij} \le \beta_{ij} + 1, \qquad
    x_j + y_{ij} = 1.
\]
Similarly, for every constraint of the third type, let us introduce two auxiliary integer variables $z_{ij}, z_{ij}'$ and replace the constraint by the three constraints
\[
    z_{ij} + z_{ij}' \le \psi_{ij} + 2, \qquad
    x_i + z_{ij} = 1, \qquad
    x_j + z_{ij}' = 1.
\]
By case analysis, it is easy to see that every replacement yields a matrix that is still totally $\Delta$-modular.
Moreover, every pair of variables appears in at most one constraint.
Thus, we may create an undirected graph $G$ and associate every variable to a vertex of $G$ such that every constraint that is not a variable bound is of the form $y_i + y_j \le b_{ij}$ or $y_i + y_j = b_{ij}$ for some edge $ij \in E(G)$, and every edge of $G$ corresponds to exactly one such constraint.
We see that the coefficient matrix of the integer program is an edge-vertex incidence matrix of $G$ and hence $\ocp(G) \le \log_2 \Delta$.

\subsection{Eliminating equations}

Our current integer program is of the form
\begin{equation} \label{eqgrapheq}
    \max \, \{ f(x) : x \in [\ell,u] \cap \Z^{V(G)} : Ax \le b, \, A_{|F}x = b_{|F} \},
\end{equation}
where $G$ is an undirected graph with $\ocp(G) \le \log_2 \Delta$, $f : \R^{V(G)} \to \R$ linear, $\ell,u \in \Z^{V(G)}$, $b \in \Z^{E(G)}$, $A$ is an edge-vertex incidence matrix of $G$, and $A_{|F}$, $b_{|F}$ are the restrictions of $A,b$ to some rows $F \subseteq E(G)$.
Our goal is to eliminate the equations $A_{|F} x = b_{|F}$ by modifying the objective function $f$.
To this end, let
\[
    \mu := \max \{ f(x) : x \in [\ell, u] \} - \min \{ f(x) : x \in [\ell, u] \} + 1
\]
and define $g : \R^{V(G)} \to \R$ via
\[
    g(x) = f(x) + \mu \one^\intercal A_{|F}x.
\]
We consider the integer program
\begin{equation} \label{eqgraphnoeq}
    \max \, \{ g(x) : x \in [\ell,u] \cap \Z^{V(G)} : Ax \le b \}.
\end{equation}
Note that if~\eqref{eqgraphnoeq} is infeasible, then so is~\eqref{eqgrapheq}.
It suffices to prove the following lemma.
\begin{lemma}
Let $x^*$ be any optimal solution to~\eqref{eqgraphnoeq}.
If $x^*$ is feasible for~\eqref{eqgrapheq}, then it is also optimal for~\eqref{eqgrapheq}.
Otherwise, \eqref{eqgrapheq} is infeasible.
\end{lemma}
\begin{proof}
    Let $\nu := \one^\intercal b_{|F}$ and note that every point $x$ that is feasible for~\eqref{eqgrapheq} satisfies $g(x) = f(x) + \mu \nu$.
    Suppose first that $x^*$ is feasible for~\eqref{eqgrapheq}.
    Clearly, every point $x$ that is feasible for~\eqref{eqgrapheq} is also feasible for~\eqref{eqgraphnoeq} and hence it satisfies
    $
        f(x) = g(x) - \mu \nu \le g(x^*) - \mu \nu = f(x^*)
    $,
    which shows that $x^*$ is indeed optimal for~\eqref{eqgrapheq}.

    Suppose now that $x^*$ is not feasible for~\eqref{eqgrapheq}.
    Then we must have $\one^\intercal A_{|F}x^* \le \one^\intercal b_{|F} - 1 = \nu - 1$ and see that
    \[
        g(x^*) \le f(x^*) + \mu (\nu - 1) \le \max \{ f(x) : x \in [\ell, u] \} + \mu (\nu - 1)
    \]
    holds.
    For the sake of contradiction, suppose that there exists a point $\bar{x}$ feasible for~\eqref{eqgrapheq}.
    Then $\bar{x}$ satisfies
    \[
        g(\bar{x}) = f(\bar{x}) + \mu \nu \ge \min \{ f(x) : x \in [\ell, u] \} + \mu \nu,
    \]
    which implies
    \begin{align*}
        g(x^*) - g(\bar{x})
        & \le \max \{ f(x) : x \in [\ell, u] \} - \min \{ f(x) : x \in [\ell, u] \} + \mu (\nu - 1) - \mu \nu \\
        & = \mu - 1 + \mu (\nu - 1) - \mu \nu = -1 < 0.
    \end{align*}
    Since $\bar{x}$ is feasible for~\eqref{eqgraphnoeq}, this is a contradiction to the optimality of $x^*$.
    Thus, \eqref{eqgrapheq} is infeasible.
\end{proof}

\subsection{Reduction to stable set}

In the previous sections we have reduced the original integer program to a series of integer programs of the form
\begin{equation}
    \label{eqhdgeIP}
    \max \, \{ w^\intercal x : x \in [\ell,u] \cap \Z^{V(G)} : Ax \le b \},
\end{equation}
where $A$ is an edge-vertex incidence matrix of an undirected graph $G$ with $\ocp(G) \le \log_2 \Delta$, and $w,\ell,u \in \Z^{V(G)}$, $b \in \Z^{E(G)}$.
In what follows, we will show that by fixing and translating some variables, we may restrict all variables to only attain values in $\{0,1\}$, in which case we end up with a stable set problem over a subgraph of $G$.
To this end, we will consider the LP relaxation
\begin{equation}
    \label{eqhdgeLP}
    \max \, \{ w^\intercal x : x \in [\ell,u] : Ax \le b \},
\end{equation}
and show that Theorem~\ref{thm:Cook_et_al} can be significantly strengthened for problems with the above structure, at least for extremal solutions.
Here, we say that a solution to a linear program is \emph{extremal} if is a vertex of the underlying polyhedron.
\begin{proposition}
    \label{propgz76du}
    Suppose that the integer program~\eqref{eqhdgeIP} is feasible.
    For every extremal optimal solution $x^*$ of the linear relaxation~\eqref{eqhdgeLP} there exists an optimal solution $\bar{x}$ to \eqref{eqhdgeIP} such that $||x^* - \bar{x}||_{\infty} \leqslant 1/2$.
\end{proposition}
In the proof of the above result we make use of the following two auxiliary facts.
\begin{lemma}[{\cite[Prop.~5.1]{ocpgenus}}] \label{lem01solprev}
Let $H$ be an undirected graph and $B$ an edge-vertex incidence matrix of $H$.
Then all vertices of the polyhedron $\conv \{ x \in \Z^{V(H)} : Bx \le \one \}$ are contained in $\{0,1\}^{V(H)}$.
\end{lemma}
\begin{lemma} \label{lem01solution}
    Let $H$ be an undirected graph, $B$ an edge-vertex incidence matrix of $H$, and $w \in \R^{V(H)}$.
    If $\max \{ w^\intercal x : Bx \le \one \}$ is attained at $\frac{1}{2} \one$, then $\max \{ w^\intercal z : Bz \le \one, \, z \in \Z^{V(H)} \}$ is attained at a point in $\{0,1\}^{V(H)}$.
\end{lemma}
\begin{proof}
    We may assume that $H$ is connected.
    Consider $P := \{ x \in \R^{V(H)} : Bx \le \one\}$ and $Q := \conv \{ x \in \Z^{V(H)} : Bx \le \one\}$.
    If $H$ is bipartite, then one side $V' \subseteq V(H)$ of the bipartition satisfies $w(V') \ge \frac{1}{2} w(V(H))$.
    Choosing $z^* \in \{0,1\}^{V(H)}$ as the characteristic vector of $V'$ yields
    \[
        w^\intercal z^*
        = w(V')
        \ge \tfrac{1}{2} w(V(H))
        = \max \{ w^\intercal x : x \in P \}
        \ge \max \{ w^\intercal z : z \in Q \},
    \]
    and hence $z^*$ is as claimed, since $z^* \in Q$.

    If $H$ is not bipartite, then $B$ has full column rank and hence $P$ does not contain a line.
    Clearly, then also $Q \subseteq P$ does not contain a line.
    This means that $\max \{ w^\intercal z : z \in Q \}$ is attained at a vertex of $Q$, which is in $\{0,1\}^{V(H)}$ by Lemma~\ref{lem01solprev}.
\end{proof}
\begin{proof}[Proof of Proposition~\ref{propgz76du}]
It is known that $x^*$ is half-integer, \emph{i.e.}, $x^* \in \frac{1}{2}\Z^{V(G)}$, see~\cite[Prop.~2.1]{NT74}. After translating the feasible region by an integer vector (which may change $\ell,u,b$), we may assume that 
\begin{alignat*}{10}
    x^*_v & = 0 & \quad & \text{ for all } v \in V_0\\
    x^*_v & = 1/2 & \quad & \text{ for all } v \in V_*,
\end{alignat*}
where $V_0 \subseteq V(G)$ and $V_* = V(G) \setminus V_0$.
Note that this implies $\ell \le \zero$ and $u \ge \chi^{V_*}$, where $\chi^{V_*}$ denotes the characteristic vector of $V_*$. A crucial observation is that if a constraint associated to an edge $vv' \in E(G)$ is `tight' at $x^*$ (\emph{i.e.}, $x^*_v + x^*_{v'} = b_{vv'}$), then we have either $v,v' \in V_0$ or $v,v' \in V_*$.

Let $z \in \Z^{V(G)}$ by any optimal solution to~\eqref{eqhdgeIP}.
We claim that
\begin{equation} \label{eqmd8bs7}
    \sum_{v \in V_0} w_v z_v \le 0.
\end{equation}
Otherwise, we may consider the point $x(\eps)$ such that $x(\eps)_v = \eps z_v$ for $v \in V_0$ and $x(\eps)_v = 1/2$ for $v \in V_*$.
Choosing $\eps > 0$ small enough, we see that $x(\eps)$ is feasible for~\eqref{eqhdgeLP} and satisfies $w^\intercal x(\eps) > w^\intercal x^*$, a contradiction to the optimality of $x^*$.

Let $H$ be the graph on vertex set $V_*$ and whose edges are the tight edges of $G$ with both ends in $V_*$, and let $B$ be an edge-vertex incidence matrix of $H$.
Observe that the restriction of $x^*$ to $V_*$, which is equal to $\frac{1}{2} \one$, is an optimal solution to $\max \{ w_{|V_*}^\intercal x : Bx \le \one \}$ and hence by Lemma~\ref{lem01solution} there exists some $y \in \{0,1\}^{V_*}$ with $By \le \one$ and $\sum_{v \in V_*} w_v y_v \ge \sum_{v \in V_*} w_v z_v$.
Let $\bar{x}$ be the vector in $\{0,1\}^{V(G)}$ that agrees with $y$ on $V_*$ and that is zero on $V_0$.
By~\eqref{eqmd8bs7} and the previous inequality, we have $w^\intercal \bar{x} \ge w^\intercal z$.

Moreover, we claim that $\bar{x}$ is feasible for~\eqref{eqhdgeIP}:
Since $\ell \le \zero$ and $u \ge \chi^{V_*}$, we see that $\bar{x} \in [\ell,u]$ holds.
Moreover, as $x^*$ is feasible for~\eqref{eqhdgeLP}, all right-hand sides of edge constraints are nonnegative, and so $\bar{x}$ satisfies all edge constraints with both ends in $V_0$.
The edge constraints with one end in $V_0$ and one end in $V_*$ have a right-hand side of at least one, and hence are also satisfied by $\bar{x}$.
It remains to consider edge constraints with both ends in $V_*$.
By construction of $y$, $\bar{x}$ satisfies all such constraints that are tight at $x^*$.
The right-hand sides of all remaining constraints are at least $2$ and are hence trivially satisfied by $\bar{x}$.

Thus, $\bar{x}$ is an optimal solution for~\eqref{eqhdgeIP}.
Moreover, since all entries of $\bar{x}$ are in $\{0,1\}$, we see that $\|x^* - \bar{x}\|_\infty \le 1/2$ holds.
\end{proof}

To finally reduce~\eqref{eqhdgeIP} to a stable set problem, we proceed as follows.
First, we compute (again in strongly polynomial time) an extremal optimal solution $x^*$ to~\eqref{eqhdgeLP}.
Recall that if no such point exists, then~\eqref{eqhdgeIP} is infeasible.
As mentioned in the above proof, $x^*$ is half-integer and hence by translating all variables by an integer vector, we may assume that $x^* \in \{0,1/2\}^{V(G)}$.
By Proposition~\ref{propgz76du}, we see that there exists an optimal solution to~\eqref{eqhdgeIP} in $\{0,1\}^{V(G)}$.
Thus, we may restrict all variables in~\eqref{eqhdgeIP} to attain values in $\{0,1\}$.
For each edge constraint consider its right-hand side $\beta$.
If $\beta < 0$, then the problem is infeasible.
If $\beta = 0$, then the constraint can be replaced by fixing the respective variables to $0$, and we remove the edge from $G$.
If $\beta > 1$, then the constraint is redundant, and so we also remove the edge from $G$.
The constraints $x \in [\ell,u] \cap \Z^{V(G)}$ may result in further fixings of variables, or, again, directly imply infeasibility.
For each variable that is fixed to $0$, we delete the corresponding vertex from $G$.
For each variable that is fixed to $1$, we delete the corresponding vertices and all its neighbors from $G$.
Note that, given an optimal solution for the resulting integer program, it is easy to recover an optimal solution for the integer program~\eqref{eqhdgeIP}.
Finally, observe that the resulting integer program is a stable set problem on a subgraph $H$ of $G$, which clearly satisfies $\ocp(H) \le \ocp(G) \le \log_2 \Delta$.

\section{Structure of graphs with bounded odd cycle packing number} \label{sec:struct}

In this section we combine a number of results from the literature on graph minors to obtain a first structural description of graphs with bounded odd cycle packing number, Theorem~\ref{thm:structure_bounded_OCP}.

\subsection{Walls}
\label{subsec:walls}

Given an integer $h\geq 2$, an {\em elementary wall of height $h$} is the graph obtained from the $2h \times h$ grid with vertex set $[2h] \times [h]$ by removing all edges with endpoints $(2i-1,2j-1)$ and $(2i-1,2j)$ for all $i\in [h]$ and $j\in [\lfloor h/2 \rfloor]$, all edges with endpoints $(2i,2j)$ and $(2i,2j+1)$ for all $i\in [h]$ and $j\in [\lfloor (h-1)/2 \rfloor]$, and removing the two vertices of degree at most $1$ in the resulting graph.
A {\em brick} of an elementary wall $W$ is a cycle of length $6$ of $W$.

\begin{figure}[hb]
\centering
\includegraphics[width=0.5\textwidth]{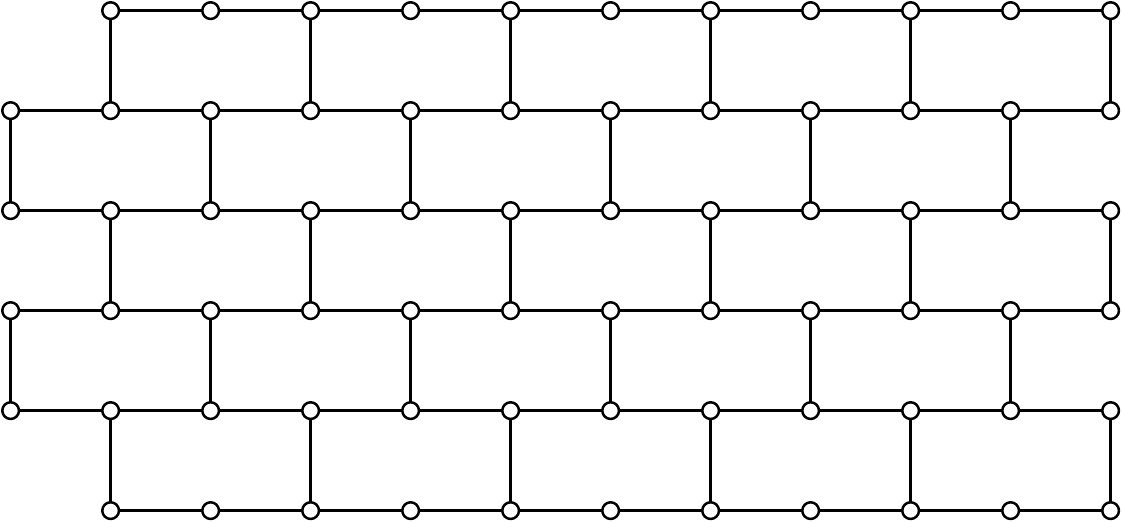}
\caption{\label{fig:wall}An elementary wall of height $6$.}
\end{figure}
 
An elementary wall of height $6$ is depicted in Fig.~\ref{fig:wall}. 
In an elementary wall $W$ of height $h$, there is a unique set of $h$ vertex-disjoint paths linking the bottom row (vertices of the form $(i,1)$) to the top row (vertices of the form $(i,h)$). 
These paths are called the {\em vertical paths} of $W$. We enumerate the vertical paths as $Q_1, \dots, Q_h$ so that the first coordinates of their vertices are increasing. There is also a unique set of $h$ vertex-disjoint paths linking $Q_1$ to $Q_h$, called the {\em horizontal paths} of $W$. 

A subdivision of an elementary wall of height $h$ is called a {\em wall of height $h$}. Bricks, vertical paths, and horizontal paths of a wall are defined as expected, as the subdivided version of their counterparts in the elementary wall.

An {\em Escher wall of height $h$} is any graph that can be obtained from a {\em bipartite} wall $W$ of height $h$ by adding $h-1$ vertex-disjoint paths $R_{1}, \dots, R_{h-1}$ in such a way that each $R_{i}$ has one endpoint in the $i$th brick of the top row of $W$ and the other in the $(h-i)$th
brick of the bottom row, $R_i$ has no other vertex in $V(W)$, and $R_{i} \cup W$ contains an odd cycle. 
It is moreover required that the endpoints of $R_{i}$ are not included in any other brick than the $i$th and the $(h-i)$th bricks of the respective rows. 
The Escher wall is said to {\em extend} the wall $W$. 
See Fig.~\ref{fig:Escher_wall} for an illustration. 

A wall $W'$ that is a subgraph of a wall $W$ is called a {\em subwall} of $W$ if every horizontal path of $W'$ is a subpath of a horizontal path of $W$, and every vertical path of $W'$ is a subpath of a vertical path of $W$.  

Notice that if $W$ is a wall of height $h$ in a graph $G$ and $X$ is a vertex subset of $G$ with $|X|<h$, then $X$ avoids at least one vertical path and one horizontal path of $W$. 
Moreover, all vertical paths and horizontal paths of $W$ avoided by $X$ are contained in the same component of $G-X$, which we call the {\em $W$-majority component} of $G-X$. 
We will need the following easy lemma about $W$-majority components and Escher walls. 

\begin{lemma}
\label{lem:odd_cycle_Escher_wall}
Let $G$ be a graph and let $W$ be a bipartite wall of height $h$ in $G$ that can be extended to an Escher wall $W'$. 
Let $X\subseteq V(G)$ with $|X| < (h-1)/4$. 
Then the $W$-majority component of $G-X$ contains an odd cycle of $W'$ intersecting every horizontal path of $W$. 
\end{lemma}
\begin{proof}
Let $P_1, \dots, P_h$ denote the horizontal paths of $W$ from top to bottom, and let $Q_1, \dots, Q_h$ denote the vertical paths of $W$ from left to right. Let $R_{1}, \dots, R_{h-1}$ denote the extra paths in $G$ that together extend the wall $W$ to the Escher wall $W'$, as in the definition above.  
Let $t:= \lceil (h-1)/2 \rceil$. 
For each $i\in [t]$, let $C_i$ be the odd cycle of $W'$ obtained as follows: Start with the endpoint of $R_i$ on the top row, walk left to the endpoint of the vertical path $Q_i$, walk down following $Q_i$ until hitting $P_{i+1}$, walk right along $P_{i+1}$ until hitting $Q_{h+1-i}$, walk down following $Q_{h+1-i}$ until the bottom row, walk left until the other endpoint of $R_i$, and finally follow $R_i$. 
Observe that each vertex of $X$ is contained in at most two of the cycles $C_1, \dots, C_t$. 
Thus $X$ avoids some cycle $C_i$, since $|X| < (h-1)/4 \leq \lceil (h-1)/2 \rceil / 2 = t/2$. 
Finally, observe that $C_i$ must be in the $W$-majority component of $G-X$ since $C_i$ intersects every horizontal path of $W$. 
\end{proof}    

Reed~\cite{Reed99} proved that large Escher walls are unavoidable in graphs with bounded odd cycle packing number but big odd cycle transversal number. 
Kawarabayashi and Reed~\cite{KR_SODA10} subsequently gave a short algorithmic proof of this fact.

\begin{theorem}[\cite{Reed99, KR_SODA10}]
\label{thm:mangoes_blueberries}
There exists a computable function $f:\mathbb{Z}_{\geq 0} \times \N\to \N$ such that, for every integers $k\geq 1$ and $h\geq 2$, every graph $G$ with $\ocp(G) \leq k$ and $\oct(G) \geq f(k,h)$ contains an Escher wall of height $h$ as a subgraph. \\
Moreover, for every fixed integer $k\geq 1$, there is an algorithm that, given an $n$-vertex graph $G$ with $\ocp(G) \leq k$, finds in time $O_k(n^2)$ either a vertex subset $X$ of size at most $f(k,h)$ such that $G-X$ is bipartite, or an Escher wall of height $h$ in $G$. 
\end{theorem}

The subscript $k$ in the notation $O_k(n^2)$ indicates that the hidden constant factor depends on $k$. 
We note that a more precise bound on the time complexity of the algorithm is given in~\cite{KR_SODA10}, it is almost linear in the number of vertices and edges of $G$.

\subsection{Odd $K_t$ models} 

Given a graph $H$, an {\em $H$ model $M$} in a graph $G$ consists of one (non-empty) tree $M(v) \subseteq G$ for each vertex $v\in V(H)$, called a {\em branch set}, and one edge $M(uv) \in E(G)$ with one endpoint in $M(u)$ and the other in $M(v)$ for each edge $uv \in E(H)$ such that all branch sets are pairwise vertex disjoint. 
For convenience, we denote by $\bigcup M$ the subgraph $\bigcup_{v\in V(H)} M(v)\cup \bigcup_{uv\in E(H)} M(uv)$ of $G$ defined by the model $M$. 
Note that $G$ has an $H$ minor if and only if $G$ has an $H$ model. 

Given a wall $W$ in $G$, and an $H$ model $M$ in $G$ for a graph $H$ with $|V(H)|=t$, we say that the wall $W$ {\em grasps} the model $M$ if for every $v\in V(H)$ there are $t$ distinct horizontal paths $P^v_1, \dots, P^v_t$ of $W$ and $t$ distinct vertical paths $Q^v_1, \dots, Q^v_t$ of $W$ such that $V(P^v_i) \cap V(Q^v_i) \subseteq M(v)$ for each $i\in [t]$. 
(The intuition behind this definition is that the model is attached to the wall in a robust way: If we remove a vertex subset $X$ of size less than $t$ from $G$, then there is still a branch set of the model that survives (i.e.\
 that avoids $X$), and a pair of an horizontal path and a vertical path that survive such that some vertex in their intersection is in the branch set.)

We will also need to consider special types of models. 
An $H$ model $M$ in $G$ is {\em odd} if, for every cycle $C$ contained in $\bigcup M$, the number of edges of $C$ that belong to branch sets is even. 
The reason for this terminology is that, if $H$ is $K_3$, then this amounts to requiring that every such cycle $C$ is odd. 
In this sense, odd $K_t$ models with $t\geq 3$ generalize odd cycles. 
In fact, an odd $K_t$ model readily gives many vertex-disjoint odd cycles in $G$, as the following easy lemma shows. 

\begin{lemma}
\label{lem:obs_oddKt_ocp}
If $G$ is a graph containing an odd $K_t$ model then $\ocp(G) \geq \lfloor t/3 \rfloor$. 
\end{lemma}

A {\em block} of a graph $G$ is an inclusion-wise maximal subgraph of $G$ which is $2$-connected, or an edge, or a vertex. 
The following easy observation about blocks and $K_t$ models will be useful. 
\begin{lemma}
Suppose $G$ is a graph, $M$ is a $K_t$ model in $G$, and $X$ is a vertex subset of $G$ with $|X| \leq t-2$. 
Then there is a block of $G-X$ that intersects all the branch sets of $M$ avoided by $X$, and this block is unique. 
This block is called the {\em main block} of $G-X$ (w.r.t.\ the model $M$). 
\end{lemma}
\begin{proof}
Let $M'$ be the model obtained from $M$ by dropping the at most $|X|\leq t-2$ branch sets of $M$ hit by $X$. 
For two edges $e, e'$ in $G-X$, we write $e \equiv e'$ if $e = e'$ or if there is a cycle of $G-X$ containing both $e$ and $e'$. This is an equivalence relation on the edges of $G-X$, whose equivalence classes are the edge sets of the blocks of $G-X$. 
If $e$ and $e'$ both link two distinct pairs of branch sets of $M'$ then $e$ and $e'$ are equivalent, as is easily checked. Thus, all the edges connecting two branch sets of $M'$ belong to the same block of $G - X$.
\end{proof}
    
We will use the following theorems about odd $K_t$ models due to Geelen \emph{et al.}~\cite{GGRSV09}.\footnote{We remark that the theorem appears with the value $c_1=32$ in~\cite[Theorem 5.2]{KR_SODA10}. However, the proof in that paper uses an inequality that is known to be valid only for large enough $\ell$, which is why we decided to leave $c_1$ unspecified, as in~\cite{GGRSV09}.}    

\begin{theorem}[\cite{GGRSV09}]
\label{thm:oddKt}
There exists a universal constant $c_1\geq 32$ such that, for every integer $\ell \geq 1$ and every graph $G$, if $G$ contains a $K_t$ model with $t = \lceil c_1 \ell \sqrt{\log_2 \ell} \rceil$ then either $G$ contains an odd $K_{\ell}$ model, or there is a subset $X$ of vertices of $G$ with $|X| < 8\ell$ such that the main block of $G-X$ is bipartite. 
\end{theorem}

\begin{lemma}
\label{lem:Kt_model_Escher_wall}
Let $k\geq 1$ be an integer and let $t:=\lceil 3c_1 k \sqrt{\log_2 (3k)} \rceil$, with $c_1$ the constant from Theorem~\ref{thm:oddKt}.  
Suppose $G$ is a graph having a $K_t$ model $M$ that is grasped by a bipartite wall $W$ of height $h\geq 96k +1$ in $G$, such that $W$ can be extended to an Escher wall $W'$ of $G$.  
Then $\ocp(G) \geq k$.  
\end{lemma}
\begin{proof}
Apply Theorem~\ref{thm:oddKt} with $\ell = 3k$. 
If the theorem yields an odd $K_{\ell}$ model, then $\ocp(G) \geq k$, as noted in Lemma~\ref{lem:obs_oddKt_ocp}, and we are done. 
Thus we may assume that Theorem~\ref{thm:oddKt} gives a vertex subset $X$ of $G$ with $|X| < 8\ell=24k$ such that the main block $U$ of $G-X$ is bipartite. 

Note that $|X| < 24k \leq (h -1)/4$, and hence by Lemma~\ref{lem:odd_cycle_Escher_wall} the $W$-majority component of $G-X$ contains an odd cycle $C$ from the Escher wall $W'$ that meets all its horizontal paths.  

Since $|X| < 24k < t-1$, we can find two distinct branch sets $B_1$ and $B_2$ of $M$ that avoids $X$. 
Since the model $M$ is grasped by $W$, we can find two horizontal paths $P_{i_1}$ and $P_{i_2}$ with $i_1\neq i_2$ such that $P_{i_j}$ avoids $X$ and meets $B_j$, for $j=1,2$. 
Now, let $u_j$ be a vertex of $U$ in $B_j$, for $j=1,2$. 
Since the cycle $C$ intersects the horizontal path $P_{i_j}$, we can find an $u_j$--$V(C)$ path $Z_j$ inside $B_j \cup P_{i_j}$, for $j=1,2$. 
We thus find two vertex disjoint paths $Z_1$ and $Z_2$ between $\{u_1,u_2\}$ and $V(C)$ in $G-X$, implying that the odd cycle $C$ is contained in the bipartite block $U$, a contradiction. 
\end{proof}
    
\subsection{Excluding a $K_t$ model grasped by a wall}

Next, we turn to a structure theorem for graphs containing some (large) wall $W$ and having no $K_t$ model that is grasped by the wall $W$. 
This theorem is at the heart of the graph minors series of Robertson and Seymour, see~\cite{RS-GraphMinorsXVI-JCTB03}. 
In this paper we will use the following recent variant of the theorem, due to Kawarabayashi \emph{et al.}~\cite{KTW20}, which comes with an efficient algorithm. 
The necessary definitions and notations will be introduced after the theorem 
(in particular, the graph $G'_0$ is a specific supergraph of $G_0$ defined at the end of this subsection).

\begin{theorem}[{\cite[Theorem 2.11]{KTW20}}]
\label{thm:KTW}
Let $t,r\geq 0$ be integers and let 
\[
h:= 49152t^{24}r + t^{10^7 t^{26}}.
\]
Let $G$ be a graph and let $W$ be a wall of height $h$ in $G$. 
Then either $G$ has a $K_t$ model grasped by $W$, or there is an $(\alpha_0,\alpha_1,\alpha_2)$-near embedding $(\sigma,G_0,A,\mathcal{V},\mathcal{W})$ of $G$ in a surface $\mathbb{S}$ with 
\begin{align*}
    &\alpha_0 := t^{10^7 t^{26}} \\
    &\alpha_1 := 2t^2 \\   
    &\alpha_2 := t^{10^7 t^{26}} \\
    &\eg(\surf) \leq t(t+1),
\end{align*} 
and such that $G'_0$ contains a flat wall $W'_0$ of height $r$ that can be lifted to a subwall $W_0$ of $W$.

Furthermore, for some function $T$, there is an algorithm with running time $T(t,r)\cdot n^{O(1)}$ that, given an $n$-vertex graph $G$ and a wall $W$ as above, finds one of the two structures guaranteed by the two outcomes of the theorem. 
\end{theorem}

We remark that in the above statement we dropped one property of the near embedding that appears in~\cite[Theorem 2.11]{KTW20}, namely `$W$-centrality', because we are not going to need it in this paper.    
This is because the existence of a large flat wall gives us a weaker version of the property that is good enough for our purposes.

Let us now introduce the necessary definitions. 
We note that the terminology and notations used here are not the original ones from~\cite{KTW20} but are taken instead from a related paper by Diestel, Kawarabayashi, M{\"{u}}ller, and Wollan~\cite{DiestelKMW12}.\footnote{While the translation between the formalism of~\cite{KTW20} and that of~\cite{DiestelKMW12} is not difficult, there are two aspects we ought to comment on: (1) Going from a `vortex of depth at most $d$' in the first paper to `a vortex having a linear decomposition of adhesion at most $d$' in the second can be done in time $O(n^{d+3})$ using~\cite[Theorem 12.2]{KTW20}. 
(2) Flat subwalls w.r.t.\ a near embedding are defined slightly differently in the two papers, which led us to introduce the notion of `lift' in the statement.}  
In the next section we will use results on near-embeddings from~\cite{DiestelKMW12}, which is why we chose their terminology. 

A {\em vortex} is a pair $V=(H,\Omega)$ where $H$ is a graph and $\Omega=\Omega(V)$ is a linearly ordered set $(u_1,u_2,\ldots,u_n)$ of some vertices in $H$. 
With a slight abuse of notation, we will denote the (unordered) set of vertices $\{u_1,u_2,\ldots,u_n\}$ by $\Omega$ as well. 
Also, we sometimes treat the vortex $V$ as the corresponding graph $H$ when convenient; for instance, we may simply write $V\subseteq G$ to mean that $H$ is a subgraph of the graph $G$. The vertices of $V - \Omega(V)$ are called {\em internal} vertices. 

Given a vortex $V=(H,\Omega)$ and a vertex subset $A$, we denote by $V-A$ the vortex obtained by deleting the vertices in $A$. 
If $\mathcal{V}$ is a set of vortices, we let $\mathcal{V}-A:=\{V-A:V\in \mathcal{V}, V-A \neq \varnothing\}$.  

A \emph{surface} $\surf$ is a non-empty compact connected Hausdorff topological space in which every point has a neighborhood that is homeomorphic to the plane~\cite{MoharThom}. The \emph{M\"obius strip} is the one-sided non-orientable surface obtained from a rectangle by identifying two of its parallel sides with one of them taken in the opposite direction to the other.
Let $\surf(h,c)$ denote the surface obtained by removing $2h+c$ open disks with disjoint closures from the sphere and gluing $h$ cylinders and $c$ M\"obius strips onto the boundaries of these disks. As stated by the classification theorem of surfaces, every surface is topologically equivalent to $\surf (h,c)$, for some $h$ and $c$. The \emph{Euler genus} $\eg(\surf)$ of a surface $\surf \cong \surf(h,c)$ is $2h+c$. Let $\surf$ be a surface and let $\Delta$ be a closed disk in the surface. We denote by $\bd(\Delta)$ and $\inter(\Delta$) the boundary and the interior of $\Delta$, respectively. 

\begin{definition}[Linear decomposition and adhesion]
Let $V=(H,\Omega)$ be a vortex with $\Omega=(u_1,u_2,\ldots,u_n)$. 
A {\em linear decomposition} of $V$ is a collection of sets $(X_1,X_2,\ldots, X_n)$ such that 
\begin{itemize}
\item for each $i\in [n]$, $X_i \subseteq V(H)$ and $u_i\in X_i$, and moreover $\cup_{i=1}^n X_i = V(H)$; 
\item for every $uv\in E(H)$, there exists $i\in [n]$ such that $\{u,v\}\subseteq X_i$, and
\item for every $x\in V(H)$, the set $\{i:x\in X_i\}$ is an interval in $[n]$. 
\end{itemize}
The {\em adhesion} of the linear decomposition is $\max(|X_i\cap X_{i+1}|: 1\le i \le n-1)$.
\end{definition}

\begin{definition}[$(\alpha_0,\alpha_1,\alpha_2)$-near embedding]
Let $\alpha_0,\alpha_1,\alpha_2\in \mathbb{N}$, and let $\mathbb{S}$ be a surface. A graph $G$ is {\em $(\alpha_0,\alpha_1,\alpha_2)$-nearly embeddable in $\mathbb{S}$} if there is a set of vertices $A\subseteq V(G)$ with $|A|\le \alpha_0$ and an integer $\alpha'\le \alpha_1$ such that $G-A$ can be written as the union of $t+1$ edge-disjoint graphs $G_0,G_1,\ldots,G_t$ with the following properties:
\begin{enumerate}[(i)]
\item For each $i\in [t]$, $V_i:=(G_i,\Omega_i)$ is a vortex where $\Omega_i=V(G_i\cap G_0)$. For all $1\le i < j \le t$, $G_i\cap G_j\subseteq G_0$.
\item The vortices $V_1,\ldots,V_{\alpha'}$ have a linear decomposition of adhesion at most $\alpha_2$. Let $\mathcal{V}$ be the collection of those vortices.
\item The vortices $V_{\alpha'+1},\ldots,V_t$ have $|\Omega_i|\le 3$ for each $\alpha'+1\le i \le t$. Let $\mathcal{W}$ be the collection of those vortices.
\item There are closed disks $\Delta_1,\ldots,\Delta_t$ in $\mathbb{S}$ with disjoint interiors and an embedding $\sigma: G_0 \hookrightarrow \mathbb{S} - \cup_{i=1}^t \inter(\Delta_i)$ such that $\sigma(G_0)\cap \bd(\Delta_i)=\sigma(\Omega_i)$ for each $i\in [t]$ and the linear ordering of $\Omega_i$ is compatible with the cyclic ordering of $\sigma(\Omega_i)$ (that is, they both agree when starting from the first vertex of $\Omega_i$). 
\end{enumerate} 
It is helpful to imagine each vortex $V_i$ ($i\in [t]$) as being drawn inside its disk $\Delta_i$ with edge crossings, as is done in \cite{KTW20}; this way each vortex is contained inside its disk (the drawing itself is irrelevant).
We call the tuple $(\sigma,G_0,A,\mathcal{V},\mathcal{W})$ an {\em $(\alpha_0,\alpha_1,\alpha_2)$-near embedding} of $G$. Below we sometimes denote $\Delta(V)$ the disk $\Delta_i$ for vortex $V = V_i$. We call vortices in $\mathcal{V}$ {\em large vortices} and vortices in $\mathcal{W}$ {\em small vortices}. 
\end{definition}

Observe that for each vortex $V_i$ ($i\in [t]$) in the above definition, there is no edge in $G$ linking a vertex in $V(G_i) - \Omega_i$ to a vertex in $V(G)-A-V(G_i)$.  
Given a near embedding $(\sigma,G_0,A,\mathcal{V},\mathcal{W})$ of a graph $G$ in a surface $\mathbb{S}$, we define a corresponding graph $G_0'$, obtained from $G_0$ by adding an edge $vw$ for every non-adjacent vertices $v,w$ in $G_0$ that are in a common small vortex $V\in \mathcal{W}$. 
These extra edges are drawn without crossings, in the disks accommodating the corresponding vortices. 
We will refer to these edges as the {\it virtual edges} of $G_0'$. 
If $H'$ and $H$ are subgraphs of $G_0'$ and $G-A$, respectively, such that $H$ is obtained from $H'$ by replacing each virtual edge $uv$ of $H'$ with an $u$--$v$ path contained in some vortex in $\mathcal{W}$, in such a way that all the paths are internally (vertex) disjoint, then we call $H$ a {\it lift} of $H'$ (w.r.t.\ the near embedding $(\sigma,G_0,A,\mathcal{V},\mathcal{W})$), and say that $H'$ {\it can be lifted to} $H$. 

A cycle $C$ in a graph $H$ embedded in a surface $\surf$ is {\em flat} if $C$ bounds a disk in $\surf$. 
A wall $W$ in $H$ is {\em flat} if the boundary cycle of $W$ (which is defined in the obvious way) bounds a closed disk $D(W)$ with the wall $W$ drawn inside it.

\subsection{Graphs with bounded odd cycle packing number}

We may now state a structure theorem for graphs with bounded odd cycle packing number, which follows from the results mentioned above. 

\begin{theorem}
\label{thm:structure_bounded_OCP}
There is a computable function $f_1(k,r)$ such that, for every integers $k\geq 1$ and $r\geq 0$, and every graph $G$ with $\ocp(G)\leq k$ and $\oct(G) \geq f_1(k,r)$, there is an $(\alpha_0,\alpha_1,\alpha_2)$-near embedding $(\sigma,G_0,A,\mathcal{V},\mathcal{W})$ of $G$ in a surface $\mathbb{S}$ with 
\begin{align*}    
    &\alpha_0 := t^{10^7 t^{26}} \\
    &\alpha_1 := 2t^2 \\   
    &\alpha_2 := t^{10^7 t^{26}} \\
    &\eg(\surf) \leq t(t+1),
\end{align*} 
and such that $G$ contains a bipartite wall $W$ of height $h$ that can be extended to an Escher wall $W'$ of $G$, and $G'_0$ contains a flat wall $W_0'$ of height $r$ that can be lifted to a subwall $W_0$ of $W$, where 
\begin{align*}    
&h:= 49152t^{24}r + t^{10^7 t^{26}} \\
&t:=\lceil 3c_1(k+1) \sqrt{\log_2 (3(k+1))} \rceil
\end{align*} 
and $c_1$ is the constant from Theorem~\ref{thm:oddKt}.   

Furthermore, for some function $T$, there is an algorithm with running time $T(k,r)\cdot n^{O(1)}$ that, given an $n$-vertex graph $G$, finds these structures. 
\end{theorem}
\begin{proof}
For $k\geq 1$ and $r\geq 0$, let $f_1(k,r) := f(k, h)$, where $f$ is the function from Theorem~\ref{thm:mangoes_blueberries}. 
Let us show that the theorem holds with this choice of $f_1$. 
Let $G$ be a graph with $\ocp(G)\leq k$ and $\oct(G) \geq f_1(k,r)$, for some integers $k\geq 1$ and $r\geq 0$.  
By Theorem~\ref{thm:mangoes_blueberries}, there is a bipartite wall $W$ of height $h$ in $G$ that can be extended to an Escher wall $W'$ of $G$. 
Apply Theorem~\ref{thm:KTW} to $G$ with the wall $W$. 
If the theorem yields a $K_t$ model grasped by $W$, then $\ocp(G) \geq k+1$ by Lemma~\ref{lem:Kt_model_Escher_wall}, a contradiction. 
Thus we obtain the second outcome of the theorem, as desired. 
\end{proof}

\section{Refining a near embedding} \label{sec:clean_near_embeddings}

Theorem~\ref{thm:structure_bounded_OCP} already gives a good description of the structure of a graph $G$ with bounded odd cycle packing number but large odd cycle transversal number. The goal of this section is to refine this structure further, in order to apply our main algorithm to it. 
This refinement relies on the following three ingredients: 
(1) tools developed by Diestel {\em et al.}~\cite{DiestelKMW12} in their work on the excluded minor theorem for graphs with large treewidth; 
(2) the odd $S$-path theorem of Geelen {\em et al.}~\cite{GGRSV09}, and 
(3) the {\EP} property of $2$-sided odd cycles for graphs embedded in surfaces, proved in~\cite{ocpgenus}. 
Below, we first introduce these results together with all the necessary definitions. 
Then we explain how to use them to obtain an improved version of Theorem~\ref{thm:structure_bounded_OCP}.  

\subsection{Definitions and tools}

Given a near embedding $(\sigma,G_0,A,\mathcal{V},\mathcal{W})$ of a graph $G$ in a surface $\mathbb{S}$, and a vortex $V\in \mathcal{W}$, we say that $V$ is {\em properly attached} if $|\Omega(V)|\leq 3$, for every two distinct vertices $u, v\in \Omega(V)$ there is an $u$--$v$ path in $V$ with no inner vertex in $\Omega(V)$, and for every three distinct vertices $u, v, w\in \Omega(V)$ there is an $u$--$v$ path in $V$ and a $v$--$w$ path, each with no inner vertex in $\Omega(V)$, in $V$ that are internally disjoint. 

Consider a graph $H$ embedded in $\surf$. 
A curve $C$ in $\surf$ is said to be {\em $H$-normal} if it intersects (the embedding of) $H$ only at vertices.
The {\em distance in $\surf$} between two points $x,y\in \surf$ is the minimum of $|C\cap V(H)|$ over all $H$-normal curves $C$ linking $x$ to $y$. 
The distance in $\surf$ between two vertex subsets $A, B$ of $H$ is the minimum distance in $\surf$ between a vertex of $A$ and a vertex of $B$. 
In the context of near embeddings, the distance in $\surf$ between two vortices $V$ and $W$ is the distance in $\surf$ between $\Omega(V)$ and $\Omega(W)$. 
A {\em noose} is a simple, closed, $H$-normal and noncontractible curve in $\surf$. For a surface $\surf$ that is not a sphere, the {\em face-width} (also known as {\em representativity}) of the embedding of $H$ in $\surf$ is the minimum of $|C\cap V(H)|$ over all nooses $C$. 

Vertex-disjoint cycles $C_1, \dots, C_t$ of $H$ are {\em concentric} if they bound closed disks with $D(C_1) \supseteq \cdots \supseteq D(C_t)$ in $\surf$. 
These cycles {\em enclose} a vertex subset $\Omega$ if $\Omega \subseteq D(C_t)$. 
They {\em tightly enclose} $\Omega$ if moreover, for every $i\in[t]$ and every point $v\in \bd( D(C_i))$, there is a vertex $w\in \Omega$ at distance at most $t-i+2$ from $v$ in $\surf$.  

In the context of a near-embedding $(\sigma,G_0,A,\mathcal{V},\mathcal{W})$ of a graph $G$ in a surface $\surf$, concentric cycles $C_1, \dots, C_t$ in $G'_0$ {\em (tightly) enclose} a vortex $V\in \mathcal{V}$ if they (tightly) enclose $\Omega(V)$. 

For integers $3\leq \beta \leq r$, an $(\alpha_0,\alpha_1,\alpha_2)$-near embedding $(\sigma,G_0,A,\mathcal{V},\mathcal{W})$ of a graph $G$ in a surface $\surf$ is {\em $(\beta, r)$-good} if the following properties are satisfied. 
\begin{enumerate}[(1)]
\item \label{P:flat}
$G'_0$ contains a flat wall $W'_0$ of height $r$. 
\item \label{P:face-width}
If $\surf$ is not the sphere, then the face-width of $G'_0$ in $\surf$ is at least $\beta$. 
\item \label{P:concentric}
For every vortex $V\in \mathcal{V}$ there are $\beta$ concentric cycles $C_1(V), \dots, C_{\beta}(V)$ in $G'_0$ tightly enclosing $V$ and bounding closed disks $D_1(V) \supseteq \cdots \supseteq D_{\beta}(V)$, such that $D_{\beta}(V)$ contains $\Omega(V)$ and $D(W'_0)$ does not meet $D_1(V)$. 
For distinct $V, V'\in \mathcal{V}$, the disks $D_1(V)$ and $D_1(V')$ are disjoint. 
\item \label{P:properly_attached}
All vortices in $\mathcal{W}$ are properly attached.  
\end{enumerate}
This is a weakening of the notion of {\em $(\beta, r)$-rich} near embeddings from~\cite{DiestelKMW12}, obtained by dropping some of the properties in the latter that we will not need. 

The heart of the proof of the main result in~\cite{DiestelKMW12} (Theorem 2 in that paper) is a procedure that starts with a near embedding of a graph $G$ in a surface $\surf$ together with a large enough flat wall and iteratively improves the near embedding until it becomes `$(\beta, r)$-rich'. 
While this is not explicitly discussed in~\cite{DiestelKMW12}, it can be checked that all the steps of this procedure can be realized in polynomial time when all the parameters involved are bounded by constants. 
Specifically, one can check that the following theorem follows from the proof in~\cite{DiestelKMW12} for the weaker notion of $(\beta, r)$-good embeddings.\footnote{We remark that since we focus on $(\beta, r)$-good embeddings, in the proof of Theorem 2 in~\cite{DiestelKMW12} we can stop as soon as properties (P1)--(P4) are established (\emph{i.e.}\ end of first paragraph on page 1208). 
The bounds in Theorem~\ref{thm:Diestel_et_al} were calculated accordingly.} 

\begin{theorem}[{implicit in \cite{DiestelKMW12}}]
\label{thm:Diestel_et_al}
Suppose we are given integers $3\leq \beta \leq r$, an $(\widetilde{\alpha}_0,\widetilde{\alpha}_1,\widetilde{\alpha}_2)$-near embedding $(\widetilde{\sigma},\widetilde{G}_0,\widetilde{A},\widetilde{\mathcal{V}},\widetilde{\mathcal{W}})$ 
of a graph $G$ in a surface $\mathbb{\widetilde{S}}$, 
with all vortices in $\widetilde{\mathcal{W}}$ being properly attached,   
and a flat wall $\widetilde{W}'_0$ in $\widetilde{G}'_0$  of height 
\[
6^{\widetilde{\alpha} + \eg(\widetilde{\surf})}(r + \widetilde{\alpha}(\beta+\widetilde{\alpha}+3)+p),
\] 
where $\widetilde{\alpha} :=\max_{0\leq i \leq 2}\widetilde{\alpha}_i$ and $p:=2\widetilde{\alpha}(\beta+2\widetilde{\alpha}+\eg(\widetilde{\surf})+2)+4$.   
Then we can find a $(\beta, r)$-good $(\alpha_0,\alpha_1,\alpha_2)$-near embedding $(\sigma,G_0,A,\mathcal{V},\mathcal{W})$ of $G$ in a surface $\surf$, with 
\begin{align*}
&\alpha_0:=\widetilde{\alpha}_0 + p(2\eg(\widetilde{\surf}) + \widetilde{\alpha}_1) \\
&\alpha_1:=\widetilde{\alpha}_1 + \eg(\widetilde{\surf}) \\
&\alpha_2:=\widetilde{\alpha}_2 + \widetilde{\alpha}_1 +  \eg(\widetilde{\surf}) \\
&\eg(\surf) \leq \eg(\widetilde{\surf})
\end{align*}
and such that the flat wall $W'_0$ in $G'_0$ guaranteed by property~\eqref{P:flat} in the definition of $(\beta, r)$-good is a subwall of $\widetilde{W}'_0$. 
Moreover, this near embedding can be computed in polynomial time provided $r, \beta, \widetilde{\alpha}_0,\widetilde{\alpha}_1,\widetilde{\alpha}_2, \eg(\widetilde{\surf})$ are all constants. 
\end{theorem}

An embedding of a graph $G$ in a surface $\surf$ is {\em cellular} if every face is homeomorphic to an open disk. A graph $G$ is {\em cellularly embedded} in $\surf$ if the embedding is cellular. 
 
\begin{theorem}[{Robertson and Vitray~\cite{RV90}, \cite[Proposition 5.5.2]{MoharThom}}]\label{thm:cycle_faces}
Let $\surf$ be a surface with $\eg(\surf) > 0$ and let $G$ be a graph that is embedded in $\surf$ with face-width at least $2$. Then there is precisely one block $Q$ of $G$ that contains a noncontractible cycle. Moreover, $Q$ is cellularly embedded in $\surf$ and all its faces are bounded by cycles. Each block $Q'$ of $G$ distinct from $Q$ is a planar subgraph of $G$ contained in the closure of some face of $Q$. Finally, the face-width of $Q$ is equal to the face-width of $G$.
\end{theorem}

A corollary of this result is that if $G$ is a $2$-connected graph that is embedded in $\surf$ with face-width at least $2$, then all the faces of $G$ are bounded by cycles.

A cycle in a graph embedded in a surface is {\em $1$-sided} if the curve corresponding to the cycle has a neighborhood homeomorphic to a M\"{o}bius strip. A cycle is {\em $2$-sided} if it is not $1$-sided. 
While the odd cycle transversal number of a graph cannot be bounded from above by a function of its odd cycle packing number (as evidenced by Escher walls), such a bound exists if we restrict ourselves to odd cycles that are $2$-sided in a graph embedded in a fixed surface, as proved by Conforti, Fiorini, Huynh, Joret, and Weltge~\cite{ocpgenus}. 
This fact will be used in our refined structure theorem. 

\begin{theorem}[{\cite{ocpgenus}}]
\label{thm:EP_2sided}
There exists a computable function $f':\mathbb{Z}_{\geq 0} \times \mathbb{Z}_{\geq 0} \to \N$ such that the following holds. 
Let $G$ be a cellularly embedded graph in a surface $\surf$ of Euler genus $g$ such that $G$ has no $k+1$ vertex-disjoint $2$-sided odd cycles. 
Then there exists a subset $X$ of vertices of $G$ with $|X|\leq f'(g, k)$ such that $X$ meets all $2$-sided odd cycles of $G$.
\end{theorem} 

We also need the following result due to Geelen \emph{et al.}~\cite{GGRSV09}, and known as the {\em odd $S$-paths theorem}.

\begin{theorem}[{\cite[Lemma 11]{GGRSV09}}]\label{thm:oddSpaths}
Let $H$ be a graph and let $h \in \Z_{\geq 1}$. For every set $S \subseteq V(H)$, either
\begin{enumerate}[(i)]
\item there are $h$ vertex-disjoint paths, each of which has an odd number of edges and both its endpoints in $S$, or
\item there is a set $X$ such that $|X|\le 2h-2$ and $H-X$ contains no such path. 
\end{enumerate}
\end{theorem}

Notice that in case (i) we can assume that each path is actually an $S$-path, that is, none of its internal vertices is in $S$. This can be achieved by shortening the paths if necessary.

\subsection{Resilience and gadgets}
\label{secknear}

First, we introduce some extra definitions. 

\begin{definition}[$\rho$-resilience and rich component]
Given $\rho \in \Z_{\ge 0}$, a graph $G$ is said to be {\em $\rho$-resilient} if, for all sets $X \subseteq V(G)$ with $|X| \leq \rho$, there is a component $H$ of $G - X$ such that $\ocp(H) = \ocp(G)$. We call $H$ a {\em rich component} of $G-X$. 
\end{definition}
Notice that, in the above definition, the rich component $H$ of $G-X$ is unique when $\ocp(G)>0$, and that $G - V(H)$ is bipartite. 
Notice also that if $\rho' \geq \rho$, $\rho'$-resilient graphs are automatically $\rho$-resilient.

It is not hard to see that, in order to solve the maximum weight stable set problem on graphs $G$ with $\ocp(G) \leq k$, one can reduce to $\rho$-resilient instances where $\rho = \rho(k)$ can be chosen arbitrarily. The idea is that if one can find a set $X \subseteq V(G)$ with $|X| \leqslant \rho$ such that each component $H$ of $G - X$ satisfies $\ocp(H) < \ocp(G)$, then after guessing which vertices of $X$ to include in the stable set, one can reduce to the maximum weight stable set problem on graphs with odd cycle packing number at most $k-1$. This is discussed in more detail at the beginning of Section~\ref{sec:preprocessing}.

Given two vertices $x,y$ of a bipartite graph $H$, we let $p_H(x,y)$ denote the parity of all paths between $x,y\in V(H)$ (if there is any).

Let $(\sigma,G_0,A,\mathcal{V},\mathcal{W})$ be an $(\alpha_0,\alpha_1,\alpha_2)$-near embedding of a graph $G$ in a surface $\surf$, and suppose further that every vortex $V\in \mathcal{W}$ is bipartite. 
(Essentially, this property will be ensured by assuming that the near embedding is $(\beta,r)$-good and that $G$ is $\rho$-resilient for a sufficiently large $\rho \in \Z_{\ge 0}$, as we will explain later.)  
When solving a maximum weight stable set problem on $G$, we may then replace each vortex $V\in \mathcal{W}$ by a small gadget that is embedded in the surface $\surf$, as we explain below.  
This gadget-replacement operation is seen as a modification of the graph $G_0$, and the resulting graph will be denoted $G^+_0$ (not to be confused with the graph $G'_0$). 

Let us describe the gadget-replacement operation. 
First, we make sure that the graphs in $\mathcal{W}$ are connected: 
If a vortex $V \in \mathcal{W}$ is not connected, we simply replace $V$ with a vortex for each component of $V$, with the appropriate subset of $\Omega(V)$ as boundary. 
Now, start with $G^+_0 := G_0$, and for each vortex $V \in \mathcal{W}$ in turn, modify $G^+_0$ as follows, see also Fig.~\ref{fig83dis9} (on page \pageref{fig83dis9}).

\begin{enumerate}[(W1)]
    \item If $|\Omega(V)| \le 1$, we do not modify $G^+_0$. 
    \item If $|\Omega(V)| = 2$, then let $\Omega(V)=\{v_1,v_2\}$. If $p_{V}(v_1,v_2)$ is even, then we add a path of length $2$ between $v_1$ and $v_2$ in $G^+_0$, that is, we add a new vertex $x$ and add the edges $v_1x$ and $v_2x$. 
    If $p_{V}(v_1,v_2)$ is odd then we add a path of length $3$ between $v_1$ and $v_2$ in $G^+_0$, that is, we add two new vertices $x,y$ and add the edges $v_1x,xy,yv_2$. Moreover, if $V$ contains the edge $v_1v_2$, then $v_1v_2$ is also added to $G^+_0$.
    \item If $|\Omega(V)| = 3$, then let $\Omega(V)=\{v_1,v_2,v_3\}$. Note that there are only two options for the parities of the paths between $v_1,v_2,v_3$. Either (a) $p_{V}(v_1,v_2)$, $p_{V}(v_1,v_3)$ and $p_{V}(v_2,v_3)$ are all even or (b) one of them is even, say without loss of generality $p_{V}(v_1,v_2)$, and the other two are odd. In case (a), we add a star with the leaves $v_1,v_2,v_3$ where each edge is subdivided once, more precisely, we add four new vertices $a_1,a_2,a_3,x$ and add the edges $v_1a_1,v_2a_2,v_3a_3,a_1x,a_2x,a_3x$. In case (b), we add a star with the leaves $v_1,v_2,v_3$ where two edges are subdivided once and the last is subdivided twice, more precisely, we add five new vertices $a_1,a_2,a_3,a_3',x$ and add the edges $v_1a_1,v_2a_2,v_3a_3,a_3a_3',a_1x,a_2x,a_3'x$. Again, if any two vertices in $\Omega(V)$ are adjacent in $V$, then the corresponding edge is also added to $G^+_0$.
\end{enumerate}

We refer to the vertices and edges we add in the above process as the set of \emph{virtual vertices} and \emph{virtual edges} of $G^+_0$, respectively. 
Let us point out the following easy observation. 

\begin{observation}\label{obs:sameocp}
$\ocp(G^+_0) \leq \ocp(G-A)$.
\end{observation}

\subsection{Bipartizing enlarged vortices}

\begin{figure}[h!]
\centering
\includegraphics[width=6cm]{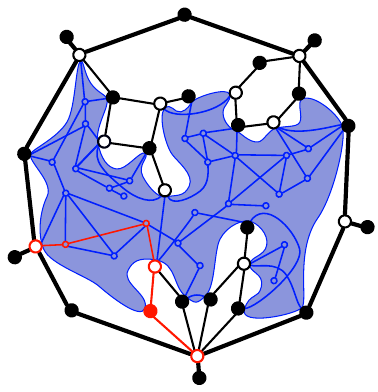}
\caption{The graph $H_2$ from the proof of Theorem~\ref{thm:bipartite_enlarged}. The disk $\Delta(V)$ of vortex $V$ is the shaded blue area. One parity-breaking path is highlighted in red.}
\label{fig:cacti}
\end{figure}

Later, we will assume that the vortices $V \in \mathcal{V}$ are bipartite, even when augmented with the boundary of the face that contains the corresponding disk $\Delta(V)$. The next result explains how that can be achieved, and is used later within the proof of Theorem~\ref{thm:structure_bounded_OCP_refined}. 
Before stating the result, let us introduce some standard terminology regarding walks. 
A \emph{walk} in a graph $G$ is a sequence $W = (v_0, e_1, v_1, \ldots, e_t, v_t)$ of vertices $v_i$ and edges $e_i$ such that for all $i \in [t]$, the edge $e_i$ links $v_{i-1}$ to $v_i$.
Let $W = (v_0, e_1, v_1, \ldots, e_t, v_t)$ be a walk.  The \emph{length} of $W$ is $t$, and $W$ is \emph{odd} or \emph{even} according to the parity of its length.  The walk $W$ is \emph{closed} if $v_0 = v_t$.  
Let $W$ be a closed walk. Seeing $W$ as an Eulerian multigraph, we can partition the edges of $W$ as $C_1 \cup \dots \cup C_\ell$, where each $C_i$ is a cycle.  We say that $W$ is \emph{$1$-sided} if the number of $1$-sided cycles among $C_1, \dots , C_\ell$ is odd. Otherwise, $W$ is \emph{$2$-sided}.

\begin{theorem}\label{thm:bipartite_enlarged}
Let $k \in \Z_{\ge 0}$ and let $G$ be a graph with $\ocp(G) \leq k$. Assume that $G$ has an $(\alpha_0,\alpha_1,\alpha_2)$-near embedding $(\sigma,G_0,A,\mathcal{V},\mathcal{W})$ such that all vortices in $\mathcal{W}$ are bipartite. Assume that $G^+_0$ has no $2$-sided odd closed walk. Let $V \in \mathcal{V}$ be a fixed vortex. Let $C_1, C_2$ be concentric cycles in $G^+_0$ bounding closed disks $D_1, D_2$ such that $D_1 \supseteq D_2 \supseteq \Delta(V)$, and $D_1$ is disjoint from $D(V')$ for all vortices $V' \in \mathcal{V} \setminus \{V\}$. 
For $i \in [2]$, let $H_i$ denote the graph obtained by taking the union of $G_0^+ \cap D_i$ and $V$. Let $q := \max \{k+1,2\alpha_2+5\}$. Assume that $G$ is $(2q^4)$-resilient, and that there is no set $Z \subseteq V(G)$ with $|Z| \leq 2q^4$ such that the rich component of $G - Z$ is entirely contained in $H_2$. Then there is a vertex subset $X \subseteq V(H_2)$ of size $|X| < 2q^4$ such that $H_1 - X$ is bipartite.
\end{theorem}

\begin{proof}
Firstly note that as we assume that $G_0^+$ has no $2$-sided odd closed walks, $G_0^+ \cap D_1$ is a planar graph with all faces even, hence a bipartite planar graph. Therefore, there is a partition of $V(G_0^+ \cap D_1)$ into two disjoint stable sets $B$ and $W$. Observe that every odd cycle contained in $H_1$ must contain an internal vertex from $V$. 

To simplify the presentation of the proof, we perform a small modification to our graphs. For every vertex $v \in W \cap V(C_2)$, we add a unique pendant  vertex $v'$. 
Let $N$ be the collection of the new pendant vertices we added in the modification. 

We apply the odd $S$-path theorem, Theorem~\ref{thm:oddSpaths}, to $H_2$ and $S := (B \cap V(C_2)) \cup N$, with $h := q^4$.\medskip

\noindent {\em Case (i).} There is a packing $\mathcal{P}$ of $h$ vertex-disjoint odd $S$-paths in $H_2$. We trim each path in $\mathcal{P}$ by removing every end belonging to $N$. Now $\mathcal{P}$ is a packing of  $V(C_2)$-paths which we call {\em parity-breaking paths}. 

Pick any edge $e_0$ of $C_2$. For a parity-breaking path $P \in \mathcal{P}$, we call the unique path linking the ends of $P$ and contained in $C_2 - e_0$ the {\em arc} of $P$. Notice that taking the union of $P$ and its arc gives an odd cycle contained in $H_2$, which we call the {\em special cycle} of $P$ and denote by $C(P)$. We say that $C(P)$ {\em extends} the parity-breaking path $P$. Since $C(P)$ contains an internal vertex of $V$ for each $P \in \mathcal{P}$, we conclude that each path $P \in \mathcal{P}$ contains an internal vertex of $V$. 

We call a {\em transaction} any subset $\mathcal{Q} \subseteq \mathcal{P}$ consisting only of $L$--$R$ paths, where $L$ and $R$ are disjoint sets of consecutive vertices of $C_2$. We claim that the number of paths in any transaction $\mathcal{Q} = \{Q_1,\ldots,Q_s\} \subseteq \mathcal{P}$ is at most $4 \alpha_2 + 2$. For $i \in [s]$, let $\ell_i \in L$ be the left endpoint of $Q_i$, $r_i \in R$ the right endpoint of $Q_i$ and $m_i \in \Omega(V)$ denote the first vertex of $Q_i$ that is in vortex $V$. 

Let $(X_1,\ldots,X_n)$ be a linear decomposition of $V$ of adhesion at most $\alpha_2$. By renumbering the paths if necessary, we may assume that $m_i$ appears before $m_{i+1}$ in $\Omega(V) := (u_1,\ldots,u_n)$, for each $i < s$. Let $j(i)$ denote the index of $m_i$ in $\Omega(V)$, for each $i \in [s]$. Hence $j(1) < \ldots < j(s)$. 

Now, setting $X_0 := X_{n+1} := \emptyset$, consider the set
$$
Y := (X_{j(1)-1}\cap X_{j(1)})\cup (X_{j(1)}\cap X_{j(1)+1}) \cup (X_{j(s)-1}\cap X_{j(s)})\cup (X_{j(s)}\cap X_{j(s)+1})\cup \{u_{j(1)},u_{j(s)}\}\,.
$$ 
This set is of size at most $4 \alpha_2 + 2$. Let $K$ denote the union of all components of $V - Y$ whose interval is contained in $\{j(1)+1,\ldots,j(s)-1\}$. (By the {\em interval} of component $K$, we mean the union over all vertices $v\in V(K)$ of the sets $\{j \in [n]: v\in X_j\}$, which is an interval in $[n]$ since $K$ is connected.)

Consider any $V(C_2)$-path $Q \subseteq H_2$ starting at one of the vertices $\ell_2, \ldots, \ell_{s-1}$ and vertex-disjoint from $Q_1$, $Q_s$ and $Y$. Let us follow the path starting from its left endpoint $\ell_i$, $1 < i < s$. Until it enters the vortex $V$, the path is confined within the planar region $\mathcal{R}$ delimited by $\ell_1 Q_1 m_1$, $\ell_s Q_s m_s$, $C_2$ and the boundary of $\Delta(V)$. When the path $Q$ enters $V$, it is confined within $K$. When the path $Q$ leaves $V$, it goes back to the planar region $\mathcal{R}$. And so on. See Fig.~\ref{fig:parity_breaking} for an illustration. Hence, $Q$ will never reach a vertex in $R$. 

This implies that $Y$ hits all the paths in transaction $\mathcal{Q}$.  Since $\mathcal{Q}$ is a packing, we see that $|\mathcal{Q}| \leq |Y| \leq 2 \alpha _2 + 4$, which proves our claim.

\begin{figure}[h!]
\centering
\includegraphics[width=7cm,clip=true,trim=2.25cm 4.75cm 11.5cm 6cm]{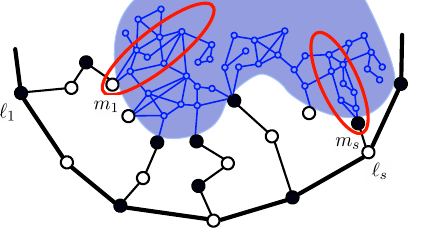}
\caption{Illustration of the proof of Theorem~\ref{thm:bipartite_enlarged}. The figure shows the initial segments $\ell_i Q_i m_i$, $i \in [s]$ of a transaction $\mathcal{Q} := \{Q_1,\ldots,Q_s\}$ and the cutset $Y$ (in red). }
\label{fig:parity_breaking}
\end{figure}

Next, break $C_2$ into a path by removing the edge $e_0$ that was selected above. Choosing a direction of traversal of this path, we obtain a linear ordering on the vertices of $C_2$. For each parity-breaking path $P \in \mathcal{P}$, let $s(P)$ and $t(P)$ be the endpoints of $P$ where $s(P)$ appears before $t(P)$ with respect to their order on $C_2$. We write $s(P) < t(P)$ to indicate this. Let $\sqsubseteq$ be a relation on the paths in $\mathcal{P}$ with $P \sqsubseteq P'$ if and only if $P = P'$ or $s(P')< s(P)< t(P)< t(P')$ (equivalently, the arc for $P$ is contained in the arc for $P'$). It is easy to see that $\sqsubseteq$ is a partial order on the paths in $\mathcal{P}$. 

By Dilworth's theorem there must be either a chain or an antichain of size $\sqrt{h} = q^2$ in the poset $(\mathcal{P},\sqsubseteq)$. Since every chain defines a transaction and $\sqrt{h} > q > 2 \alpha_2 + 4$, our claim implies that there is an antichain of size $\sqrt{h}$ in $(\mathcal{P},\sqsubseteq)$. Let $\mathcal{P}' \subseteq \mathcal{P}$ be such an antichain.

We define a new relation $\preceq$ on $\mathcal{P}'$, with $P \preceq P'$ if and only if $P = P'$ or $s(P) < t(P) < s(P') < t(P')$ (equivalently, the arc of $P$ is entirely to the left of that for $P'$). It is easy to see that $\preceq$ is again a partial order\footnote{In fact, $\preceq$ is a \emph{unit interval order}.}, this time on the paths in $\mathcal{P}'$. By Dilworth's theorem there is either a chain or antichain of size $\sqrt[4]{h} = q$ in $(\mathcal{P}',\preceq)$. If there is an antichain in $(\mathcal{P}',\preceq)$ of size $q$ then we get a transaction of size $q$. We then reach a contradiction on the number of paths in $\mathcal{P}'$, since $q > 2\alpha_2 + 4$. 

Hence there is a chain in $(\mathcal{P}',\preceq)$ of size $q$. Let $\mathcal{P}'' \subseteq \mathcal{P}'$ be the set of paths in this chain. Let 
$\mathcal{C} := \{C(P) : P \in \mathcal{P}''\}$ be the collection of special cycles extending each of the $q$ paths in $\mathcal{P}''$. This is a packing of $q > k$ odd cycles in $H_2 \subseteq G$, contradicting our hypothesis that $\ocp(G) \le k$. Hence, Case (i) cannot occur.\medskip

\noindent {\em Case (ii).} There is a vertex subset $X \subseteq V(H_2)$ such that $|X| \le 2h-2 < 2 q^4$ that hits all the odd $S$-paths in $H_2$. We claim that $X$ is the required set in our theorem. If $X$ happens to contain any vertex $v' \in N$, we delete this vertex and add its neighbor $v \in W \cap V(C_2)$ to the hitting set $X$. Below, we assume that $X$ and $N$ are disjoint. 

Toward a contradiction, assume that $H_1 - X$ contains an odd cycle $C_3$. This odd cycle contains an internal vertex of $V$. By Menger's theorem applied to the graph $H_1 - X$, we can either find two disjoint $V(C_2)$--$V(C_3)$ paths or one vertex $v$ that hits all such paths. In the first case, we can find an odd $S$-path $Q$ in $H_1 - X$ using the odd cycle $C_3$, the two $V(C_2)$--$V(C_3)$ paths and possibly some vertices in $N$ with the incident edges. Notice that $Q$ contains an internal vertex of $V$, because $G_0^+ \cap D_1$ is bipartite, and thus contains no odd $S$-path. It follows that $Q$ is fully contained in $H_2-X$, contradicting the fact that $X$ meets all odd $S$-paths in $H_2$. In the second case, consider the set $X \cup \{v\}$. Observe that $C_3$ is fully contained in $H_2$, as follows from the fact that $G_0^+ \cap D_1$ is bipartite and $\{v\}$ separates $V(C_2)$ from $V(C_3)$ in $H_1 - X$. Since $|X \cup \{v\}| \leq 2h - 2 + 1 \leq 2h$ and $G$ is $(2h)$-resilient, $G - (X \cup \{v\})$ has a rich component. If this rich component contains some vertex of odd the cycle $C_3$, then the rich component has to be fully contained in $H_2$ since $X \cup \{v\}$ separates $V(C_2)$ from $V(C_3)$, which contradicts our hypothesis. Hence, the rich component is disjoint from $C_3$. This is a contradiction to the fact that the odd cycle packing number of the rich component equals that of $G$. This concludes the proof of the theorem.
\end{proof}

\subsection{Final structure}

We may now summarize the final structure that we obtain using the tools from this section. 

\begin{theorem} \label{thm:structure_bounded_OCP_refined}
There are computable functions $\rho(k)$ and $g(k)$ such that, for every integer $k\geq 1$, and for every graph $G$ with $\ocp(G)=k$ that is $\rho(k)$-resilient, there is an $(g(k),g(k),g(k))$-near embedding $(\sigma,G_0,A,\mathcal{V},\mathcal{W})$ of $G$ in a non-orientable surface $\mathbb{S}$ with $\eg(\surf) \leq g(k)$ with the following properties: 
\begin{itemize}
\item all vortices in $\mathcal{W}$ are bipartite;
\item there is no $2$-sided odd walk in $G_0^+$;
\item the embedding of $G_0^+$ in $\surf$ has face-width at least $2$;
\item for each vortex $V \in \mathcal{V}$, there is a flat cycle $C(V)$ in $G^+_0$ bounding a closed disk $D(V)$ containing $\inter(\Delta(V))$, where $\Delta(V)$ is the disk in the near embedding that is associated to $V$; for every two distinct vortices $V, W \in \mathcal{V}$, the closed disks $D(V)$ and $D(W)$ do not intersect;
\item for each vortex $V \in \mathcal{V}$, the union of the subgraph of $G_0^+$ contained in $D(V)$ and $V$ is bipartite;   
\item every face of $G_0^+$ that does not contain $\inter(\Delta(V))$ for any $V \in \mathcal{V}$ is bounded by a cycle.
\end{itemize}

Furthermore, there is a polynomial-time algorithm finding such a near embedding of $G$ when $k$ is a fixed constant. 
\end{theorem}

\begin{proof}
Let $k\geq 1$, and let $G$ be as in the theorem statement. 
Let 
\begin{align*}
&t:=\lceil 3c_1(k+1) \sqrt{\log_2 (3(k+1))} \rceil \\
&\widetilde{\alpha}_0 := t^{10^7 t^{26}} \\
&\widetilde{\alpha}_1 := 2t^2 \\   
&\widetilde{\alpha}_2 := t^{10^7 t^{26}} \\
&\widetilde{\alpha} :=\max_{0\leq i \leq 2}\widetilde{\alpha}_i \\
&\alpha_1:=\widetilde{\alpha}_1 + t(t+1) \\
&\alpha_2:=\widetilde{\alpha}_2 + \widetilde{\alpha}_1 +  t(t+1) \\    
&q:= \max(k+1, 2\alpha_2 + 5) \\
&\beta := 3f'(t(t+1),k) + 3t(t+1) + 6q^4 \alpha_1 + 2 \\
&p:=2\widetilde{\alpha}(\beta+2\widetilde{\alpha}+t(t+1)+2)+4 \\
&\alpha_0:=\widetilde{\alpha}_0 + p(2t(t+1) + \widetilde{\alpha}_1) \\
&r := \max(\alpha_0+4, 3f'(t(t+1),k) + 3t(t+1) + 2^5q^4+2) \\
&\tilde{r} := 6^{\widetilde{\alpha} + t(t+1)}(r + \widetilde{\alpha}(\beta+\widetilde{\alpha}+3)+p) \\
&h:= 49152t^{24}\tilde{r} + t^{10^7 t^{26}} 
\end{align*} 
where $c_1$ is the constant from Theorem~\ref{thm:oddKt} and $f'$ is the function from Theorem~\ref{thm:EP_2sided}. 
Define    
\begin{align*}
&g(k) := \max(\alpha_0 + 3f'(t(t+1),k) + 3t(t+1) + 2q^4 \alpha_1, \alpha_2) \\
&\rho(k) := \max(f_1(k,\tilde{r}), \alpha_0 + 3f'(t(t+1),k) + 3t(t+1) + 2q^4 \alpha_1 + 1), 
\end{align*}
where $f_1(k,\tilde{r})$ is the function from Theorem~\ref{thm:structure_bounded_OCP}. 
Observe that $\oct(G) \geq \rho(k) \geq f_1(k,\tilde{r})$. 
Apply Theorem~\ref{thm:structure_bounded_OCP} to $G$, using $\tilde{r}$ for the value of $r$ in that theorem.   
Let $(\widetilde{\sigma},\widetilde{G}_0,\widetilde{A},\widetilde{\mathcal{V}},\widetilde{\mathcal{W}})$ denote the resulting
$(\widetilde{\alpha}_0,\widetilde{\alpha}_1,\widetilde{\alpha}_2)$-near embedding 
of $G$ in a surface $\mathbb{\widetilde{S}}$, with $\eg(\mathbb{\widetilde{S}}) \leq t(t+1)$.
Let also $W$ be the resulting bipartite wall of height $h$ in $G$, that can be extended to an Escher wall $W'$ of $G$, and let $\widetilde{W}'_0$ denote the given flat wall of height $\tilde{r}$, which can be lifted to a subwall $\widetilde{W}_0$ of $W$.

We may assume that all vortices in $\widetilde{\mathcal{W}}$ are properly attached, as is easily checked.    
Apply Theorem~\ref{thm:Diestel_et_al} to the near embedding $(\widetilde{\sigma},\widetilde{G}_0,\widetilde{A},\widetilde{\mathcal{V}},\widetilde{\mathcal{W}})$ with the flat wall $\widetilde{W}'_0$, to obtain a $(\beta, r)$-good near embedding. 
Let $(\sigma,G_0,A,\mathcal{V},\mathcal{W})$ denote the resulting $(\beta, r)$-good $(\alpha_0,\alpha_1,\alpha_2)$-near embedding of $G$ in a surface $\surf$ with $\eg(\surf)\leq \eg(\mathbb{\widetilde{S}})$, 
and let $W'_0$ denote the resulting flat wall in $G'_0$ of height $r$, which can thus be lifted to a subwall $W_0$ of $W$. 

Next, let us show that every vortex $V \in \mathcal{W}$ must be bipartite. 
Since $|A| + |\Omega(V)| \leq \alpha_0 + 3  \leq \rho(k)$ by resilience there is a rich component $C$ of $G-(A\cup \Omega(V))$. 
Also, since $|A| + |\Omega(V)| < r$, there is a $W_0$-majority component $C'$ of $G-(A\cup \Omega(V))$, and a $W$-majority component $C''$ of $G-(A\cup \Omega(V))$. 
Observe that $C''=C'$, since $C''$ includes an horizontal path of $W$ avoided by $A\cup \Omega(V)$, which intersects a vertical path of $W_0$ avoided by $A\cup \Omega(V)$. 
Since $|A| + |\Omega(V)| < (h-1)/4$, by Lemma~\ref{lem:odd_cycle_Escher_wall}, $C''$ contains an odd cycle. 
Hence we must have $C=C'=C''$. 
Now, a key observation is that the component $C$ must be vertex disjoint from the vortex $V$, because $C$ contains a vertex not in $V$ (namely, a vertex of $W_0$ not in $\Omega(V)$). 
Since $\ocp(G)=\ocp(C)$, it follows that $V$ is bipartite (otherwise we would have $\ocp(G)>\ocp(C)$), as claimed. 

Now that we have established that all vortices in $\mathcal{W}$ are bipartite, we may consider the graph $G_0^+$ resulting from gadget-replacement. 
Apply Theorem~\ref{thm:EP_2sided} to $G_0^+$. 
Note that we cannot have $k+1$ vertex-disjoint $2$-sided odd cycles in $G_0^+$, since $\ocp(G_0^+) \leq \ocp(G) \leq k$. 
Hence, we obtain a set $Y$ of vertices of $G_0^+$ meeting all $2$-sided odd cycles of $G_0^+$, of size  $|Y| \leq f'(\eg(\surf), k)$, where $f'$ is the function from Theorem~\ref{thm:EP_2sided}. 
By Lemma 10.8 in~\cite{ocpgenus}, we may extend $Y$ by adding at most $\eg(\surf)$ vertices of $G_0^+$, so that the resulting set $Y'$ meets all $2$-sided odd walks of $G_0^+$. 
Finally, using the definition of the gadgets, we observe that for every vortex $V\in \mathcal{W}$, we may replace the corresponding virtual vertices that are in $Y'$ (if any) by the at most three vertices in $\Omega(V)$, keeping the property that all $2$-sided odd walks of $G_0^+$ are hit. 
Let $X$ denote the resulting set, which is thus a subset of $V(G_0)$.  
Note that  
\[
|X| \leq 3|Y'| \leq 3(f'(\eg(\surf), k) + \eg(\surf)) \leq 3f'(t(t+1),k) + 3t(t+1).  
\]

Modify the $(\alpha_0,\alpha_1,\alpha_2)$-near embedding $(\sigma,G_0,A,\mathcal{V},\mathcal{W})$ of $G$ in $\surf$ by removing all vertices in $X$ from $G_0$ and from the vortices, and adding them to the apex set $A$, and replacing the flat wall $W'_0$ by a subwall of $W'_0$ of height $r-|X|$ avoiding $X$. 
For the sake of readability, with some abuse of notation we denote the resulting near embedding by $(\sigma,G_0,A,\mathcal{V},\mathcal{W})$ again, and the flat subwall by $W'_0$ again (and let $W_0$ denote a lift of the new wall $W'_0$ in the new near embedding).  
Thus $|A| \leq \alpha_0 + |X|$ now. 
Observe that this near embedding is $(\beta', r')$-good for $\beta':=\beta - |X|$ and $r' := r- |X|$. 

Next, for each vortex $V\in \mathcal{V}$, we apply Theorem~\ref{thm:bipartite_enlarged} to the near embedding for vortex $V$, with $q$ as defined in the beginning of the proof, and with $D_1, D_2$ being two of the $\beta'$ concentric cycles around $V$ that survived.  
Let us quickly justify that the condition about sets $Z$ in that theorem is indeed satisfied: Suppose $Z\subseteq V(G)$ with $|Z| \leq 2q^4$ is such that the rich component $K$ of $G-Z$ is entirely contained in $H_2$ (using the notations of the theorem). 
Then $K$ is vertex disjoint from the wall $W_0$, by our choice of $D_2$. 
On the other hand, similarly as argued above for vortices in $\mathcal{W}$, since $W_0$ has height $r' > |Z|$, the $W_0$-majority component of $G-Z$ is the same as the $W$-majority component of $G-Z$, and must contain an odd cycle (as argued above with the Escher wall $W'$). 
Hence, this component must be $K$, a contradiction. 
The application of  Theorem~\ref{thm:bipartite_enlarged} results in a set $Y_V$ of vertices meeting all odd cycles contained in the union of $G_0^+ \cap D_1$ and $V$, with $|Y_V|\leq 2q^4$.  
Again, modifying $Y_V$ to avoid virtual vertices as above, we obtain a set $X_V$ that contains no virtual vertex from $G_0^+$, with $|X_V|\leq 3|Y_V|\leq 6q^4$, such that $X_V$ meets all odd cycles contained in the union of $G_0^+ \cap D_1$ and $V$. 

Let $X'$ be the union of $X_V$ for all vortices $V\in \mathcal{V}$. 
Thus $|X'| \leq 6q^4 |\mathcal{V}| \leq 6q^4 \alpha_1$. 
Similarly as before, modify the current near embedding by removing all vertices of $X'$ from $G_0$ and from the vortices, and adding them to the apex set $A$. 
Thus, now we have  $|A| \leq \alpha_0 + |X| + |X'| \leq g(k)$. 

Observe that, since $\beta' > |X'|$, for each vortex $V\in \mathcal{V}$ one of the $\beta'$ concentric cycles around $V$ survives.   

If $\surf$ is orientable, then $G_0^+$ must be bipartite since $G_0^+$ cannot have $2$-sided odd cycles. Since all vortices in $\mathcal{V}$ and in $\mathcal{W}$ are bipartite, we deduce that $G-A$ is bipartite, a contradiction with resilience since $|A|\leq \alpha_0 + |X| + |X'| < \rho(k)$. 
Hence, $\surf$ must be non-orientable. 

Furthermore, since $\beta \geq |X|+|X'|+2$, the embedding of $G_0'$ in $\surf$ has face-width at least $2$, and the same holds for $G_0^+$, as is easily checked. 

Finally, the last property can be achieved easily by possibly creating extra bipartite vortices in $\mathcal{W}$ attaching on at most one vertex of $G_0$. 
Indeed, suppose that there is a face of $G_0^+$ that does not contain $\inter(\Delta(V))$ for any $V \in \mathcal{V}$, and that is not bounded by a cycle. 
Then, by Theorem~\ref{thm:cycle_faces} and using the notations of that theorem, there is a block $Q' \neq Q$ of $G_0^+$ that is drawn (in a planar way) in the closure of the corresponding face of $Q$. 
Noticing that $Q'$ must be bipartite, it can be made into a vortex in $\mathcal{W}$ attaching on at most one vertex of $G_0$, by modifying the near embedding in an appropriate way. 
(Note that all properties of our near embedding established above still hold after this modification.)

In summary, the resulting near embedding satisfy all the desired properties. 
It is not difficult to check that the above proof can be carried out in polynomial time when $k$ is a constant. 
\end{proof}    

For simplicity, let us simply call {\em $k$-near embedding} the near-embedding of $G$ described in Theorem~\ref{thm:structure_bounded_OCP_refined}.

\section{Preprocessing and postprocessing} \label{sec:preprocessing}

While the previous two sections provided general structural results, we will now employ techniques specific to the stable set problem to show that every instance of the maximum weight stable set problem for graphs with bounded odd cycle packing number can be reduced to very specific instances in strongly polynomial time.
In other words, we will describe the first steps of our main algorithm to obtain a highly structured instance.

To this end, let $G$ be a graph with $\ocp(G) \le k$, where $k \ge 1$ is a constant.
We assume that we have already established a strongly polynomial time algorithm for the stable set problem on graphs with odd cycle packing number at most $k - 1$.
Note that such an algorithm exists for graphs with odd cycle packing number equal to $0$, \emph{i.e.}, bipartite graphs (see \emph{e.g.}~\cite{SchrijverLPIP}). 

As a first step, we check whether $G$ is $\resiliencebd(k)$-resilient.
To this end, for each $X \subseteq V(G)$ with $|X| \le \resiliencebd(k)$ we check whether every connected component $H$ of $G - X$ satisfies $\ocp(H) \le k - 1$, which can be efficiently done using the algorithm of Kawarabayashi and Reed~\cite{KR_STOC10}. 
In the case that such a subset $X$ has been found, we simply enumerate all stable sets $S_1,\dots,S_\ell \subseteq X$ and compute a maximum-weight stable set $S_i'$ in $G - (X \cup N(S_i))$ for each $i \in [\ell]$.
The latter can be done efficiently since all connected components of $G - (X \cup N(S_i))$ have odd cycle packing number at most $k - 1$.
A maximum-weight stable set for $G$ is given by $S_i \cup S_i'$ whose weight is maximum over all $i \in [\ell]$, and we are done.

Suppose now that $G$ is $\resiliencebd(k)$-resilient. Since $\resiliencebd(k) \ge \apexbd(k)$, we obtain from Theorem~\ref{thm:structure_bounded_OCP_refined} a $k$-near embedding of $G - A$ for some apex set $A \subseteq V(G)$ with $|A| \leq \apexbd(k)$. 
As a next step, it appears natural to enumerate all stable sets of the apex set $A$ and perform as for the set $X$ above.
However, given a stable set $S \subseteq A$, instead of computing a maximum weight stable set in $G - (A \cup N(S))$, we will compute a maximum weight stable set $S'$ in $G-A$, by setting the weights of the vertices in $N(S) \setminus A$ to be zero, and replace $S'$ by $S' \setminus N(S)$.
This slight modification will be needed since we want to work with the whole of $G-A$. We do not want to explicitly delete vertices beyond those of $A$ since the structure we have on $G-A$ from  Theorem~\ref{thm:structure_bounded_OCP_refined} is quite delicate.

At this point, we may assume that $A = \emptyset$.

\subsection{Edge-induced weights} \label{secEdgeInduced}

Next, it will be convenient to replace the vertex weights by ``equivalent'' weights that are induced by edge costs.
This notion was introduced in~\cite{ocpgenus} and is crucial for several parts of our algorithm.
We say that vertex weights $w : V(G) \to \Q_{\ge 0}$ are \emph{edge-induced} if there exist nonnegative edge costs $c : E(G) \to \Q_{\geq 0}$ such that $w(v) = \sum_{e \in \delta(v)} c(e)$, where $\delta(v) := \{e \in E(G) \mid v \in e\}$ denotes the set of edges incident to $v$.
In this case we will also say that $w$ is \emph{induced by} $c$.
The following result states that we may efficiently reduce to edge-induced weights. It is a consequence of an earlier result by Nemhauser and Trotter~\cite{nemhauser1975vertex}.

\begin{proposition}
    Given a graph $G$ and vertex weights $w : V(G) \to \Q_{\ge 0}$, in strongly polynomial time we can compute edge-induced vertex weights $w' : V(G) \to \Q_{\ge 0}$ and sets $S_0,S_1 \subseteq V(G)$ such that for every $w'$-maximal stable set $S$ in $G$, the set $(S \setminus S_0) \cup S_1$ is a $w$-maximal stable set in $G$.
\end{proposition}
\begin{proof}
    Let $x^*$ denote an optimal solution to the LP relaxation
    \[
        \max \Big\{ \sum \nolimits_{i \in V(G)} w(i) x_i : x_i + x_j \le 1 \text{ for all } ij \in E(G), \, x \in [0,1]^{V(G)} \Big\},
    \]
    which we can compute in strongly polynomial time with Tardos's  algorithm~\cite{Tardos86}.
    We define $S_0 := \{ i \in V(G) : x^*_i = 0 \}$, $S_1 := \{ i \in V(G) : x^*_i = 1\}$, and $w' : V(G) \to \Q_{\ge 0}$ via
    \[
        w'(i) = \begin{cases} w(i) & \text{ if } i \in V(G - S_0 - S_1) \\ 0 & \text{ else}. \end{cases}
    \]
    Consider any $w'$-maximal stable set $S$ in $G$ and note that $S \setminus (S_0 \cup S_1)$ is a $w$-maximal stable set in $G - S_0 - S_1$.
    The result by Nemhauser and Trotter~\cite{nemhauser1975vertex} states that adding $S_1$ to $S \setminus (S_0 \cup S_1)$ yields a $w$-maximal stable set in $G$.

    It remains to argue that $w'$ is edge-induced.
    Consider the dual of the above LP, given by
    \[
        \min \{ y(E(G)) + z(V(G)) : y(\delta(i)) + z_i \ge w(i) \text{ for all } i \in V(G), \, y \in \R^{E(G)}_{\ge 0}, \, z \in \R^{V(G)}_{\ge 0} \},
    \]
    where $y(F) := \sum_{e \in F} y_e$ and $z(V(G)) := \sum_{i \in V(G)} z_i$.
    Let $(y^*, z^*) \in \Q^{E(G)} \times \Q^{V(G)}$ denote an extremal optimal solution of the dual.
    
    Define edge costs $c : E(G) \to \R_{\ge 0}$ via
    \[
        c(e) = \begin{cases} y^*_e & \text{ if } e \subseteq V(G - S_0 - S_1) \\ 0 & \text{ else.} \end{cases}
    \]
    We claim that $w'(i) = \sum_{e \in \delta(i)} c(e)$ holds for all $i \in V(G)$, which implies that $w'$ is edge-induced.
    Clearly, the claim holds for all vertices in $S_0 \cup S_1$.
    For a vertex $i \in V(G - S_0 - S_1)$ recall that we have $0 < x_i^* < 1$.
    Thus, by complementary slackness we have $y^*(\delta(i)) + z^*_i = w(i)$ and $z^*_i = 0$.
    Moreover, observe that for every edge $ij \in \delta(i)$ with $j \in S_0 \cup S_1$ we cannot have $x_i^* + x_j^* = 1$, and hence, again by complementary slackness, we must have $y^*_{ij} = 0$.
    This yields
    \[
        w'(i) = w(i) = y^*(\delta(i)) = \sum_{ij \in \delta(i) : j \notin S_0 \cup S_1} y^*_{ij} = \sum_{e \in \delta(i)} c(e). \qedhere
    \]
\end{proof}

\subsection{Slack sets and edge costs}
\label{secSlackSets}

Given edge-induced vertex weights, we will associate a cost to every stable set that allows us to treat some of the following steps in a more elegant way.

Consider any stable set $S \subseteq V(G)$. We say that an edge of $G$ is \emph{slack} with respect to $S$ if none of its two endpoints belongs to $S$. The other edges are called \emph{tight}. Notice that an edge is tight if and only if exactly one of its endpoints is in $S$.
By $\sigma(S) \subseteq E(G)$ we denote the set of edges that are slack with respect to $S$, \emph{i.e.}, $\sigma(S) = \{ e \in E(G) : e \cap S = \emptyset \}$, and say that $\sigma(S)$ is the \emph{slack set} of $S$.
Moreover, we call a set of edges $F \subseteq E(G)$ a \emph{slack set} if $F$ is the slack set of some stable set in $G$.

Given edge costs $c : E \to \R_{\ge 0}$, we say that the \emph{cost} of $S$ is the total cost of its slack edges, \emph{i.e.}, $c(S) := c(\sigma(S)) = \sum_{e \in \sigma(S)} c(e)$.
With this definition, finding a maximum-weight stable set is equivalent to finding a minimum-cost slack set:

\begin{lemma} \label{lemdh47eu}
    Let $w : V(G) \to \R_{\ge 0}$ be induced by $c : E(G) \to \R_{\ge 0}$.
    For every stable set $S$ in $G$ we have $w(S) + c(S) = c(E(G))$.
\end{lemma}
\begin{proof}
    Since $w$ is induced by $c$, we have $w(S) = \sum_{v \in S} w(v) = \sum_{v \in S} \sum_{e \in \delta(v)} c(e)$.
    Since $S$ is a stable set, the latter is equal to $c(E(G) \setminus \sigma(S)) = c(E(G)) - c(\sigma(S)) = c(E(G)) - c(S)$.
\end{proof}

For the following parts it will be useful to observe that, in a bipartite graph with edge-induced vertex weights, the minimum cost of a stable set is always zero, which is attained by each side of the bipartition. The corresponding slack set is the empty set.

\subsection{Replacing small vortices by gadgets} \label{sec:gadgets}

Recall that the $k$-near embedding $(\sigma,\varnothing,\mathcal{V},\mathcal{W})$ we get from Theorem~\ref{thm:structure_bounded_OCP_refined} 
includes vortices $W\in \mathcal{W}$ with $|\Omega(W)| \leq 3$ (recall that we assume $A = \varnothing$, and that this is without loss of generality due to our preprocessing). Suppose that we replace the interior of each $W \in \mathcal{W}$ by a gadget as in the definition of $G_0^+$ in Section~\ref{secknear}.
In what follows, we will show that we can compute edge costs $c^{+}$ for the resulting graph such that every minimum $c^{+}$-cost stable set in that graph can be turned into a minimum $c$-cost stable set in the original graph, both in strongly polynomial time. To this end, it suffices to apply the following lemma sequentially, which is implicit in~\cite[§4]{ocp1} but reproven here:

\begin{lemma}
    \label{lemhd8sg7}
    Let $G,W$ be two given graphs such that $W$ is bipartite and $|V(G) \cap V(W)| \le 3$, and let $c : E(G) \cup E(W) \to \R_{\ge 0}$ be given edge costs.
    Let $G^{+}$ be the graph that arises from $G$ according to (W1)--(W3) in Section~\ref{secknear}.
    In strongly polynomial time, we can compute edge costs $c^{+} : E(G^{+}) \to \R_{\ge 0}$ such that for any given minimum $c^{+}$-cost stable set in $G^{+}$ we can obtain a minimum $c$-cost stable set in $G \cup W$, again in strongly polynomial time.
\end{lemma}

The proof of Lemma~\ref{lemhd8sg7} relies on the following fact:
\begin{lemma}
    \label{lemks9dh7}
    Let $W$ be a graph with two vertices $v_1,v_2$ such that all $v_1$--$v_2$ paths in $W$ have the same parity $p$.
    \begin{enumerate}[(i)]
        \item Suppose that $p$ is odd, and let $S_1, S_2$ denote stable sets in $W$ with $S_1 \cap \{v_1,v_2\} = \{v_1,v_2\}$ and $S_2 \cap \{v_1,v_2\} = \emptyset$.
        Then there exist stable sets $S_3,S_4$ in $W$ satisfying $S_3 \cap \{v_1,v_2\} = \{v_1\}$, $S_4 \cap \{v_1,v_2\} = \{v_2\}$, $S_1 \cup S_2 = S_3 \cup S_4$, and $S_1 \cap S_2 = S_3 \cap S_4$.
        \item Suppose that $p$ is even, and let $S_1, S_2$ denote stable sets in $W$ with $S_1 \cap \{v_1,v_2\} = \{v_1\}$ and $S_2 \cap \{v_1,v_2\} = \{v_2\}$.
        Then there exist stable sets $S_3,S_4$ in $W$ satisfying $S_3 \cap \{v_1,v_2\} = \{v_1,v_2\}$, $S_4 \cap \{v_1,v_2\} = \emptyset$, $S_1 \cup S_2 = S_3 \cup S_4$, and $S_1 \cap S_2 = S_3 \cap S_4$.
    \end{enumerate}
\end{lemma}
\begin{proof}
    We repeat the arguments from~\cite[proof of Thm.~1]{witt2018polyhedral}.
    In each of the above cases, let $K$ denote the connected component of $W[S_1 \cup S_2]$ that contains $v_1$.
    Note that $v_2 \notin K$.
    It is easy to see that $S_3 := (S_1 \cap K) \cup (S_2 \setminus K)$ and $S_4 := (S_1 \setminus K) \cup (S_2 \cap K)$ are stable sets in $W$ satisfying the claim.
\end{proof}

\begin{proof}[Proof of Lemma~\ref{lemhd8sg7}]
    Let $\Omega := V(G) \cap V(W)$ and let $c_G,c_W$ denote the restrictions of $c$ onto $E(G),E(W)$, respectively.
    If $|\Omega| \le 1$, observe that $G^{+} = G$.
    In this case, we may simply set $c^{+} := c_G$.
    To see this, first observe that for every stable set $S$ in $G \cup W$, the set $S' = S \cap V(G)$ is a stable set in $G^{+}$ with $c^{+}(S') \le c(S)$.
    Moreover, every stable set $S'$ in $G^{+}$ can be easily extended to a stable set $S$ in $G \cup W$ with $c(S) = c^{+}(S')$:
    If $|S' \cap \Omega| = 1$, $S$ arises from $S'$ by adding the side of the bipartition of $W$ that contains the vertex in $S' \cap \Omega$.
    Otherwise, $S$ arises from $S'$ by adding the other side of the bipartition.

    Next, let us consider the case $|\Omega| = 3$ and assume that all paths in $W$ joining two vertices from $\Omega$ have even length.
    The remaining cases (two paths have odd length or the case $|\Omega| = 2$) are similar or easier and are left to the reader.
    Suppose that $\Omega = \{v_1,v_2,v_3\}$ and let $a_1,a_2,a_3,x$ denote the virtual vertices as in (W3) in Section~\ref{secknear}.
    For each $I \subseteq [3]$ let $S_I$ denote a minimum $c_W$-cost stable set in $W$ with $S_I \cap \Omega = \{v_i : i \in I\}$.
    Note that we may compute all these stable sets in strongly polynomial time since $W$ is bipartite.
    Moreover, observe that we have $c_W(S_\emptyset) = c_W(S_{\{1,2,3\}}) = 0$.
    We define $c^{+}$ by setting
    \begin{alignat*}{10}
        c^{+}(e) & := c(e) & \quad & \text{for } e \in E(G), \\
        c^{+}(v_ia_i) & := c_W(S_{[3] \setminus \{i\}}) & & \text{for } i=1,2,3, \text{ and} \\
        c^{+}(a_ix) & := c_W(S_{\{i\}}) & & \text{for } i=1,2,3.
    \end{alignat*}
    The costs $c^{+}$ for the other cases are depicted in Fig.~\ref{fig83dis9}.

    \begin{figure}
        \begin{center}
            \includegraphics[width=0.8\textwidth]{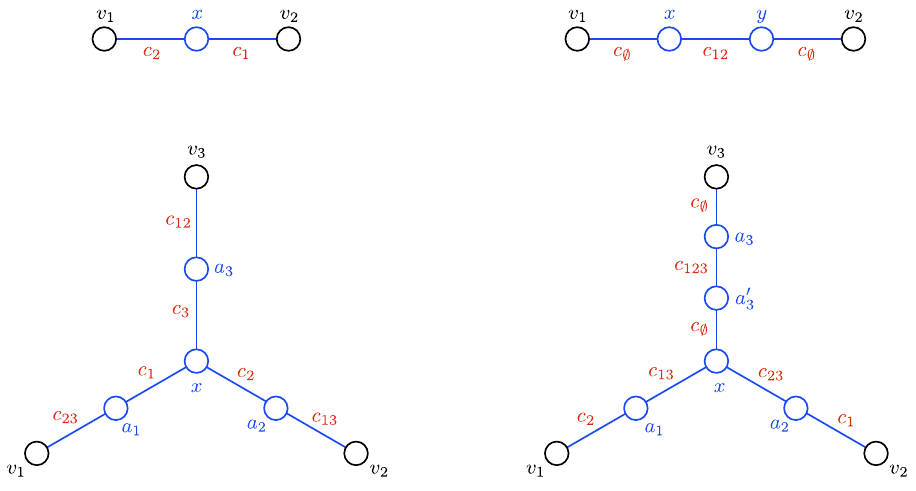}
        \end{center}        
        \caption{\label{fig83dis9}
        Gadgets used in the definition of $G_0^+$ in Section~\ref{secknear}.
        Vertices of $G_0$ are depicted in black, virtual vertices and virtual edges in blue.
        The edge costs $c^{+}$ for the proof of Lemma~\ref{lemhd8sg7} are given in red, where we use the shorthand notation $c_\emptyset := c_W(S_\emptyset)$, $c_i := c_W(S_{\{i\}})$, $c_{ij} := c_W(S_{\{i,j\}})$, and $c_{123} := c_W(S_{\{1,2,3\}})$.}
    \end{figure}

    It remains to prove the following: a) For every stable set $S$ in $G \cup W$, there is a stable set $S'$ in $G^{+}$ with $c^{+}(S') \le c(S)$. b) Moreover, for every stable set $S'$ in $G^{+}$, in strongly polynomial time we can compute a stable set $S$ in $G \cup W$ with $c(S) \le c^{+}(S')$.

    To prove a), let $S$ be a stable set in $G \cup W$ and set $I := \{ i : v_i \in S\}$.
    Let $S'$ arise from $S$ as follows.
    First, we delete the vertices from $V(W) \setminus V(G)$.
    Next, if $|I| \in \{0,3\}$, then we may add virtual vertices to $S'$ so that all virtual edges are tight with respect to $S'$.
    In this case, we obtain $ c^{+}(S') = c_G(S \cap V(G)) \le c(S) $.
    If $|I| \in \{1,2\}$, then we may add virtual vertices to $S'$ so that exactly one virtual edge $e$ with $c^{+}(e) = c_W(S_I)$ is slack with respect to $S'$.
    In this case, we obtain
    \begin{align*}
        c^{+}(S') = c_G(S \cap V(G)) + c^{+}(e) & = c_G(S \cap V(G)) + c_W(S_I) \\
        & \le c_G(S \cap V(G)) + c_W(S \cap V(W)) \\
        & = c(S).
    \end{align*}

    To prove b), let $S'$ be a stable set in $G^{+}$ and set $I := \{ i : v_i \in S'\}$.
    We let $S$ arise from $S'$ by deleting its virtual vertices and adding $S_I$.
    If $|I| \in \{0,3\}$, then we immediately obtain
    \[
        c(S)
        = c_G(S \cap V(G)) + c_W(S_I)
        = c_G(S \cap V(G))
        \le c^{+}(S').
    \]
    If $|I| = 1$, we may assume that $I = \{3\}$ holds.
    Let $F$ denote the set of virtual edges that are slack with respect to $S'$.
    It suffices to show that $c^{+}(F) \ge c_W(S_{\{3\}})$ holds since then we obtain
    \[
        c(S)
        = c_G(S \cap V(G)) + c_W(S_{\{3\}})
        \le c_G(S \cap V(G)) + c^{+}(F)
        = c^{+}(S').
    \]
    If $a_3x \in F$, then $c^{+}(F) \ge c^{+}(a_3x) = c_W(S_{\{3\}})$.
    Otherwise, $v_1a_1 \in F$ and $v_2a_2 \in F$ and hence
    \begin{align*}
        c^{+}(F)
        \ge c^{+}(v_1a_1) + c^{+}(v_2a_2)
        = c_W(S_{\{2,3\}}) + c_W(S_{\{1,3\}})
        & \ge c_W(S_{\{3\}}) + c_W(S_{\{1,2,3\}}) \\
        & = c_W(S_{\{3\}}),
    \end{align*}
    where the last inequality is due to Lemma~\ref{lemks9dh7} and Lemma~\ref{lemdh47eu}.

    It remains to consider the case $|I| = 2$, in which we may assume that $I = \{1,2\}$ holds.
    Again, let $F$ denote the set of virtual edges that are slack with respect to $S'$.
    It suffices to show that $c^{+}(F) \ge c_W(S_{\{1,2\}})$ holds since then we obtain
    \[
        c(S)
        = c_G(S \cap V(G)) + c_W(S_{\{1,2\}})
        \le c_G(S \cap V(G)) + c^{+}(F)
        = c^{+}(S').
    \]
    If $v_3a_3 \in F$, then $c^{+}(F) \ge c^{+}(v_3a_3) = c_W(S_{\{1,2\}})$.
    Otherwise, $a_1x \in F$ and $a_2x \in F$ and hence
    \begin{align*}
        c^{+}(F)
        \ge c^{+}(a_1x) + c^{+}(a_2x)
        = c_W(S_{\{1\}}) + c_W(S_{\{2\}})
        & \ge c_W(S_\emptyset) + c_W(S_{\{1,2\}}) \\
        & = c_W(S_{\{1,2\}}),
    \end{align*}
    where the last inequality is again due to Lemma~\ref{lemks9dh7} and Lemma~\ref{lemdh47eu}.
\end{proof}

Going back to our instance $G$ and its $k$-near embedding $(\sigma,\varnothing,\mathcal{V},\mathcal{W})$, from now on we may assume that $\mathcal{W} = \emptyset$, and hence $G_0 = G_0^+$.

\subsection{Increasing connectivity}

In subsequent steps, we will require that $G_0$ is $2$-connected. If this happens not to be the case, we enlarge $G_0$ using the following construction. Below, we assume the notations of Theorem~\ref{thm:structure_bounded_OCP_refined}.

Recall that the embedding of $G_0 = G_0^+$ has face-width at least $2$. 
It is known that, for any cellularly embedded graph $H$, the faces of $H$ are all bounded by cycles if and only if $H$ is $2$-connected and has face-width at least $2$, see Mohar~\cite[Proposition 3.8]{Mohar97}. Hence, it suffices to make sure that the faces of $G_0$ are all bounded by cycles. 

Suppose that there is a face $f$ of $G_0$ whose boundary is not a cycle. By Theorem~\ref{thm:structure_bounded_OCP_refined}, $f$ contains the disk $\Delta(V)$ for some vortex $V \in \mathcal{V}$. By Theorem~\ref{thm:cycle_faces}, $G_0$ has a unique block $H_0$ containing a noncontractible cycle. This block $H_0$ is cellularly embedded in $\surf$ and all its facial walks are cycles. Each block of $G_0$ that is distinct from $H_0$ is a planar graph drawn inside some face of $H_0$. Let $\hat{f}$ denote the face of $H_0$ containing $f$. The boundary of $\hat{f}$ is a cycle, say $C$. Moreover, the boundary of $f$ is a {\it cactus graph} containing $C$ together with some extra cacti hanging from vertices of $C$, where a cactus graph is a graph where every two simple cycles have at most one vertex in common. 

Let $W = (v_0,e_1,v_1,\ldots,e_\ell,v_\ell)$ denote the facial walk of $f$. We modify the graph $G$ by adding, for each index $j \in \{0,1,\ldots,\ell\}$, two new vertices $v'_j$ and $v''_j$, as well as the edges $v_j v'_j$, $v'_j v''_j$ and $v''_jv''_{j+1}$ (indices are computed modulo $\ell+1$). All these new vertices and edges are added to $G_0$. Next, the vortex $V$ is updated by replacing each vertex $v_j$ of the walk $W$ belonging to $\Omega(V)$ by the corresponding vertex $v''_j$. In case some vertex $v \in \Omega(V)$ repeats in $W$, we perform the operation only once, for the occurrence of $v$ in $W$ that is found by following $\bd(f)$ until it reaches $v$. See Fig.~\ref{fig:thicken} for an illustration. Finally, we set sufficiently high costs on the edges $v_j v'_j$ and $v'_j v''_j$ for $j \in [\ell]$ to force these edges to be tight in every optimal solution, and we set zero costs on all the edges of the form $v''_jv''_{j+1}$.

\begin{figure}[h!]
\centering
\includegraphics[width=11cm,clip=true,trim=2cm 5cm 2.5cm 3.5cm]{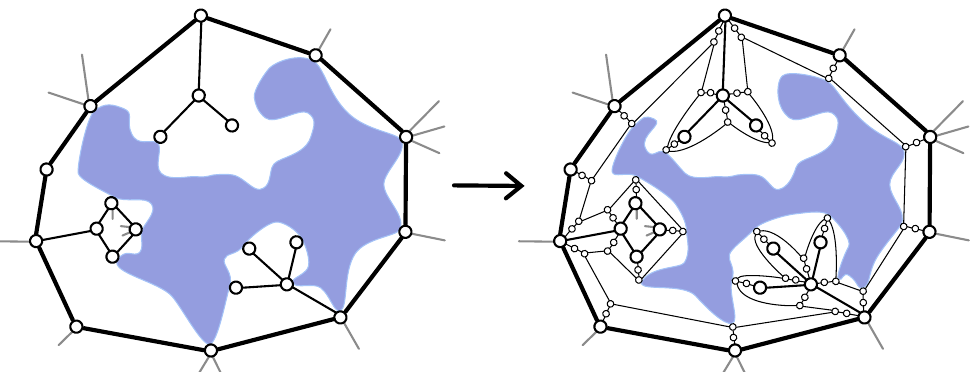}
\caption{\label{fig:thicken}Thickening the boundary of a face of $G_0$ toward $2$-connectivity. The disk $\Delta(V)$ is shown in blue.}
\end{figure}

Notice that these modifications do not change the minimum cost of a stable set in $(G,c)$. Notice that for every minimum cost stable set $S$ in the new edge-weighted graph, $S \ni v_j$ if and only if $S \ni v''_j$, for all $j$. Hence, we can trivially map any minimum cost stable set in the new edge-weighted graph to a minimum cost stable set in the original edge-weighted graph. 

Notice also that the thickening process described above preserves the property that $G_0$ has no $2$-sided odd closed walk, and also the property that, for each vortex $V\in \mathcal{V}$, the union of $G_0 \cap D(V)$ and $V$ is bipartite (where the disk $D(V)$ is modified as described as above). 

From now on, we may assume that each face of $G_0 = G_0^+$ is bounded by a cycle, and in particular that $G_0$ is $2$-connected.

\subsection{Slack vectors}

In Section~\ref{secSlackSets} we have reduced the task of finding a maximum-weight stable set to that of finding a minimum-cost stable set.
The latter task is equivalent to finding a minimum-cost slack set $F$, provided that we are able to efficiently compute a stable set $S$ with $\sigma(S) = F$.
To this end, let us consider the set
\[
    \mathcal{Y} = \mathcal{Y}(G) := \left\{ y \in \Z_{\ge 0}^{E(G)} \mid \exists x \in \Z^{V(G)} : y_{vw} = 1 - x_v - x_w \text{ for all } vw \in E(G) \right\},
\]
which we call the set of \emph{slack vectors} of $G$.
Notice that if $F$ is a slack set of $G$, then its characteristic vector belongs to $\mathcal{Y}$, in which case $x$ can be chosen as the characteristic vector of the corresponding stable set.

\begin{lemma}
    \label{lemgzh6f8}
    Let $G$ be a connected non-bipartite graph and let $y \in \mathcal{Y}(G)$.
    In strongly polynomial time, we can compute the unique vector $x \in \Z^{V(G)}$ with $y_{vw} = 1 - x_v - x_w$ for all $vw \in E(G)$.
\end{lemma}
\begin{proof}
    Since $y \in \mathcal{Y}$ we know that there exists some $x$ with the desired properties.
    Let $C$ denote an odd cycle in $G$ and let $e_1,\dots,e_t$ denote the edges along $C$.
    Letting $u \in e_1 \cap e_t$, we can compute $x_u$ by observing that $\sum_{i=1}^t (-1)^{i-1} y_{e_i} = 1 - 2x_u$ holds.
    The remaining entries of $x$ can be computed by simple propagation:
    If we know the value $x_v$ for some $v \in V(G)$, then for each neighbor $w$ of $v$ we have $x_w = 1 - x_v - y_{vw}$. 
\end{proof}

In particular, given a minimum-cost slack set of a graph $G$, we obtain a minimum-cost stable set by applying the above lemma for each non-bipartite connected component, and by selecting one side of the bipartition of each bipartite connected component.

In our setting, instead of computing minimum-cost slack sets it turns out to be more convenient to compute minimum-cost slack vectors. Assume for now that graph $G$ is cellularly embedded on a surface in such a way that there is a $2$-sided odd walk. We will use the key observation from \cite{ocpgenus} that slack vectors of $G$ can be viewed as nonnegative integer circulations in (some orientation of) the dual graph of $G$ that satisfy a small number of additional constraints. This interpretation of the solution space allows us to apply algorithmic techniques commonly used for optimization over circulations (such as decomposition into cycles), which turn out to be very useful. Also, we will use the fact that the sum of two integer circulations is an integer circulation, which has no natural counterpart for slack sets.

Clearly, most slack vectors are not characteristic vectors of slack sets.
However, for nonnegative edge costs, the minimum cost of a slack vector is always attained by a characteristic vector of a slack set:

\begin{lemma}
\label{lem:convert}
Let $G$ be a graph with edge costs $c : E(G) \to \Q_{\ge 0}$ and let $y \in \mathcal{Y}(G)$.
In strongly polynomial time, we can find a slack set $F$ with $c(F) \leq c(y)$ and a stable set $S$ with $\sigma(S) = F$.
\end{lemma}
\begin{proof}
We may assume that $G$ is connected.
If $G$ is bipartite, we may return $F = \emptyset$.
Suppose that $G$ is not bipartite.
Given $y \in \mathcal{Y}(G)$, we first compute $x \in \Z^{V(G)}$ with $y_{uv} = 1 - x_u - x_v$ for all $uv \in E(G)$ using Lemma~\ref{lemgzh6f8}.

In what follows, we will iteratively modify $x$ in a way that the resulting vector $y$ remains nonnegative and without increasing its cost.
In each iteration, none of the sets $V_{01} := \{ v \in V(G) : x_v \in \{0,1\}\}$ and $E_0 := \{ e \in E(G) : y_e = 0 \}$ will become smaller, and the size of at least one of the sets will strictly increase. Thus, the number of iterations will be at most $|V(G)|+|E(G)|$. 

If $x \in \{0,1\}^{V(G)}$, then $x$ is the characteristic vector of a stable set and the support of $y$ is the associated slack set, in which case we are done.

Otherwise, let $H$ be a connected component of the graph $(V(G), E_0)$ containing a vertex $v$ with $x_v \notin \{0,1\}$.
Note that $H$ is bipartite:
In fact, there exists some $\alpha \ge 2$ such that $x_a = \alpha$ for all $a \in A$ and $x_b = 1 - \alpha$ for all $b \in B$, where $A,B$ is a bipartition of $H$.
For $X \in \{A,B\}$ let $E_X$ denote the sets of edges of $G$ with at least one vertex in $X$.
Note that $E_A \cap E_B = E(H)$.

If $c(E_A) > c(E_B)$, we may add some integer $t > 0$ to all $x_a$, $a \in A$ and subtract it from all $x_b$, $b \in B$ in such a way that a new edge becomes part of $E_0$.
If $c(E_A) \le c(E_B)$, we may add some integer $t > 0$ to all $x_b$, $b \in B$ and subtract it from all $x_a$, $a \in A$ in such a way that a new edge becomes part of $E_0$ or all vertices of $H$ become part of $V_{01}$. 

Note that the number of arithmetic operations is polynomial in the size of $G$ and that only additions and subtractions are performed.
Hence, this procedure runs in strongly polynomial time.
\end{proof}

The fact that optimal solutions are attained at characteristic vectors of slack sets will be exploited in the next section to reduce the search space.

Let $V_1 = (H_1, \Omega_1)$, \ldots, $V_t = (H_t, \Omega_t)$ denote the vortices in $\mathcal{V}$. For each $i \in [t]$, let $G_i$ denote the graph that is obtained by augmenting $H_i$ with the boundary of the face $f_i$ of $G_0$ containing $\Delta(V_i)$. Thus $G_0 \cap G_i$ is the boundary cycle of face $f_i$. In our approach, we will decompose a slack vector $y \in \mathcal{Y}(G)$ into $y_i \in \Z^{E(G_i)}$ for $i=0,1,\dots,t$. Notice that each $y_i$ is a slack vector of $G_i$.
We will regard $y_0$ as a global solution, and each $y_i$ with $i \in [t]$ as a local solution. The following lemma allows us to combine feasible solutions for each of $G_0$, $G_1$, \ldots, $G_t$ into a solution for $G$, as long as the solutions coincide on the boundary of each $f_i$, $i \in [t]$.

\begin{lemma} \label{lem:composing_solutions}
Let $G$ be a connected graph that is the union of connected graphs $G_0$, $G_1$, \ldots, $G_t$ such that for each $i \in \{1, \ldots, t\}$, $G_i$ is bipartite and $G_0 \cap G_i$ is connected.
A vector $y \in \R^{E(G)}$ satisfies $y \in \mathcal{Y}(G)$ if and only if $y_0 \in \mathcal{Y}(G_0), y_1 \in \mathcal{Y}(G_1), \dots, y_t \in \mathcal{Y}(G_t)$, where $y_i \in \R^{E(G_i)}$ denote the restriction of $y$ to the coordinates in $E(G_i)$.
\end{lemma}


The proof of Lemma~\ref{lem:composing_solutions} is based on the following characterization of $\mathcal{Y}(G)$ from~\cite{ocpgenus}.
For every walk $W = (v_0,e_1,v_1,\ldots,e_\ell,v_\ell)$ in $G$ and vector $y \in \R^{E(G)}$, we use the notation
$$
\omega_W(y) := \sum_{i=1}^\ell (-1)^{i-1} y_{e_i}\,.
$$
 
\begin{proposition}[{\cite[Proposition 6.1]{ocpgenus}}] 
\label{prop:characterization_y}
Let $G$ be a connected graph. If $G$ is bipartite, then 
$$
\mathcal{Y}(G) = \{y \in \Z_{\geq 0}^{E(G)} \mid \omega_W(y) = 0 \text{ for all even closed walks } W \text{ in } G\}\,.
$$ 
Otherwise, for any odd cycle $C$ contained in $G$, we have
$$
\mathcal{Y}(G) = \{y \in \Z_{\geq 0}^{E(G)} \mid \omega_C(y) \text{ is odd},\ \omega_W(y) = 0 \text{ for all even closed walks } W \text{ in } G\}\,.
$$ 
\end{proposition}

\begin{proof}[Proof of Lemma~\ref{lem:composing_solutions}]
The forward implication is clear. 
Assume that the vector $y \in \R^{E(G)}$ satisfies $y_i \in \mathcal{Y}(G_i)$ for each $i \in \{0,1,\ldots,t\}$. Let $W = (v_0,e_1,v_1,\ldots,e_\ell,v_\ell)$ denote any closed walk in $G$. We claim that $\omega_W(y) = 0$ in case $W$ is even, and $\omega_W(y) \equiv 1 \pmod{2}$ in case $W$ is odd. By Proposition~\ref{prop:characterization_y}, the claim implies that $y \in \mathcal{Y}(G)$.

Assume that $W$ is not fully contained in any of the graphs $G_i$ for $i \ge 1$. This can be done without loss of generality, since otherwise $W$ is contained in $G_i$ for some $i \geq 1$. In this case, $\omega_W(y) = \omega_W(y_i)$, which is $0$ if $W$ is even or an odd number if $W$ is odd, because $y_i \in \mathcal{Y}(G_i)$. 

We will show that there is a closed walk $W_0$ in $G_0$ with the same parity as $W$ such that $\omega_{W}(y) = \omega_{W_0}(y) = \omega_{W_0}(y_0)$ (notice that the last equality is trivial). The claim then directly follows from Proposition~\ref{prop:characterization_y}, since no matter what the parity of $W$ is, we find a correct value for $\omega_{W}(y)$.  

If $W$ is not fully contained in $G_0$, then there are indices $i \in \{1,\ldots,t\}$ and $j, k \in \{0,\ldots,\ell\}$ such that $j < k$, $v_j, v_k \in V(G_0) \cap V(G_i)$, and the subwalk
$
v_j W v_k := (v_j,e_{j+1},v_{j+1},\ldots,e_k,v_k)
$
has all its edges in $E(G_i) \setminus E(G_0)$. Let $P$ denote any $v_j$--$v_k$ path contained in $G_0 \cap G_i$, and let $Z := (v_j W v_k) P^{-1}$ denote the closed walk obtained by composing $v_j W v_k$ and $P^{-1}$.

Since $G_i$ is bipartite, $Z$ is even. Because $y_i \in \mathcal{Y}(G_i)$, this yields $\omega_{v_j W v_k}(y_i) = \omega_{P}(y_i)$. Let $W' := P (v_k W v_j)$ denote the closed walk obtained from $W$ by substituting $P$ for $v_j W v_k$. Since $Z = (v_j W v_k) P^{-1}$ is even, $W'$ has the same parity as $W$. Moreover,
$$
\omega_{W'}(y) = \omega_{W}(y) \pm \underbrace{(\omega_{P}(y) - \omega_{v_j W v_k}(y))}_{=0} = \omega_{W}(y)\,.
$$
Notice that $W'$ contains less edges of $E(G_i) \setminus E(G_0)$ than $W$. 

Iterating this reasoning a finite number of times, we find a closed walk $W_0$ contained in $G_0$, such that $W_0$ has the same parity as $W$ and $\omega_{W}(y) = \omega_{W_0}(y)$, establishing the claim.
\end{proof}

\section{Final algorithm} \label{sec:DP}

We start this section with a summary of the structural results and preprocessing steps established in the previous sections.

\begin{corollary} \label{cor:structure_DP}
For every constant $k \in \Z_{\ge 1}$, there exist computable $\rho(k), g(k) \in \Z_{\ge 1}$ such that for every instance of the maximum weight stable set problem on an $n$-vertex graph with at most $k$ vertex-disjoint odd cycles, one of the following holds.

\begin{enumerate}[(I)]
\item The instance can be reduced, in strongly polynomial time, to at most $2^{\rho(k)}$ instances on induced subgraphs of the input graph such that each component has at most $k-1$ vertex-disjoint odd cycles. 

\item The input graph can be transformed, in strongly polynomial time, into a graph $G$ with $O(n + \vortexbd(k))$ vertices as well as $s \le 2^{\vortexbd(k)}$ many edge-induced vertex weights $w_1, \ldots, w_s \in \Q^{V(G)}_{\ge 0}$, such that the given instance reduces to computing a maximum weight stable set in each weighted graph $(G,w_i)$ for $i \in [s]$. Moreover, $G$ decomposes as the union of 
$2$-connected subgraphs $G_0, G_1, \ldots, G_t$, where $t \le \vortexbd(k)$, such that:

\begin{enumerate}[(i)]
\item $G_0$ is cellularly embedded in a non-orientable surface $\surf$ of Euler genus at most $\genusbd(k)$, with face-width at least $2$ (in particular, every face of $G_0$ is bounded by a cycle);
\item Every odd cycle of $G_0$ is $1$-sided (in particular, every facial cycle of $G_0$ is even);
\item For every $i \in [t]$, there is a vortex $V_i^+ := (G_i,\Omega^+_i)$ that has a linear decomposition of adhesion at most $\adhesionbd(k) + 2$, where $\Omega^+_i$ is the vertex set of the boundary of some face $f_i$ of $G_0$, and the linear ordering of $\Omega^+_i$ is compatible with the cyclic ordering of the vertices incident to $f_i$;
\item For all $i \in [t]$, $G_i$ is bipartite, and $G_0 \cap G_i$ is the cycle bounding $f_i$;
\item For distinct $i, j \in [t]$, the subgraphs $G_i$ and $G_j$ are vertex-disjoint.
\end{enumerate}
\end{enumerate}
\end{corollary}

\begin{proof}
The result follows by combining \Cref{thm:structure_bounded_OCP_refined} with the results and observations from \Cref{sec:preprocessing}. Case (I) corresponds to input graphs that are not $\rho(k)$-resilient, and case (II) to input graphs that are $\rho(k)$-resilient.

The vortices in (iii) are slightly larger than those in \Cref{thm:structure_bounded_OCP_refined} in the sense that each $V^+_i = (G_i, \Omega^+_i)$ is the union of a (large) vortex $V_i = (H_i,\Omega_i)$ in $\mathcal{V}$ and its corresponding cycle $C(V_i)$. We can easily modify the linear decomposition of any such vortex $V_i$ in order to obtain a linear decomposition of $V_i^+$ without increasing the adhesion by more than $2$.

We remark that \Cref{thm:structure_bounded_OCP_refined} does not guarantee that each $G_i$ is $2$-connected for $i \ge 1$. Notwithstanding, we may assume this without loss of generality. Indeed, $G_i$ has a unique block $B_i$ containing $C(V_i)$. Since $G_i$ is bipartite, we can set $y(e) = 0$ for each edge $e$ belonging to a block of $G_i$ distinct from $B_i$, while preserving optimality of a slack vector $y \in \mathcal{Y}(G)$. 
\end{proof}

For the rest of the section, we assume that $G$ is a graph as in case (II) of \Cref{cor:structure_DP} and that $c \in \Q^{E(G)}_{\ge 0}$ are edge costs. Our goal is to compute a minimum $c$-cost slack vector of $G$. We design a dynamic program (DP) to achieve this. Recall that, by \Cref{lem:convert}, we can convert in strongly polynomial time any minimum cost slack vector $y \in \mathcal{Y}(G)$ to a maximum weight stable set $S \subseteq V(G)$.

Consider any slack vector $y \in \mathcal{Y}(G)$. For $i \in \{0\} \cup [t]$, let
$y_i$ denote the restriction of $y$ to $E(G_i)$. It should be clear that each $y_i$ is a slack vector of the corresponding subgraph $G_i$. Moreover, for $i \ge 1$, we have $y_0(e) = y_i(e)$ for each edge $e$ of the even cycle $G_0 \cap G_i$. In virtue of \Cref{lem:composing_solutions}, we can view $y$ as being obtained by composing slack vectors $y_0$, $y_1$, \ldots, $y_t$ where each $y_i \in \mathcal{Y}(G_i)$ and $y_0$ agrees with each $y_i$ on $E(G_0 \cap G_i)$ whenever $i \ge 1$. For each $i \ge 1$, we know that $G_i$ is bipartite, hence $\mathcal{Y}(G_i)$ has a simple structure. Here we focus mainly on $y_0$ and $\mathcal{Y}(G_0)$. Notice that since $\surf$ is non-orientable, we know that $G_0$ is not bipartite.

\subsection{Dual graph and circulations} 

In this section, we recall from \cite{ocpgenus} how every slack vector $y_0 \in \mathcal{Y}(G_0)$ can be interpreted as an integer circulation in a particular orientation of $G_0^*$, the dual graph of $G_0$.


By \cite[Lemma 7.1]{ocpgenus}, $G_0^*$ has an orientation such that in the cyclic ordering of the edges incident to each dual vertex $f$, the edges alternatively leave and enter $f$. This is known as an \emph{alternating orientation}. We denote by $D = (V(D),A(D))$ any alternating orientation of $G_0^*$. (In fact, $D$ is unique up to reversal, but we will not need this.) 
Notice that $D$ has no loops or parallel directed edges, but could have anti-parallel directed edges. 

For $i \in [t]$, we call the face $f_i$ a \emph{vortex face} and the cycle $G_0 \cap G_i$ the \emph{boundary cycle} of $G_i$. The set of vortex faces $\{f_1,\ldots,f_t\}$ is a bounded size subset of $V(D)$.

\begin{proposition}[{Conforti \emph{et al.}~\cite[Proposition 3.1]{ocpgenus}}] \label{prop:dual_representation} Let $G_0$ be a graph that is $2$-connected, non-bipartite and cellularly embedded in a surface $\surf$ of Euler genus $g$ in such a way that every odd cycle of $G_0$ is $1$-sided. One can compute in polynomial time an alternating orientation $D$ of $G_0^*$ and even closed walks $W_1$, \ldots, $W_{g-1}$ in $G_0$ such that, letting $C$ be any odd cycle of $G_0$ and letting $\omega : \Z^{A(D)} \to \Z_2 \times \Z^{g - 1}$ be the map defined as
$$
\omega(y_0) :=  (\omega_C(y_0) \ (\mathrm{mod} \, 2), \omega_{W_1}(y_0), \dotsc, \omega_{W_{g-1}}(y_0))\,,
$$
we get the representation
$$
\mathcal{Y}(G_0) = \{y_0 \in \Z_{\ge 0}^{A(D)} \mid y_0 \text{ is a circulation in } D \text{ and } \omega(y_0) = (1,\mathbf{0})\}\,.
$$
\end{proposition}

In the proposition above, we identify the spaces $\R^{E(G_0)}$ and $\R^{A(D)}$ through the bijection between $E(G_0)$ and $A(D)$.
Topologically, the condition $\omega(y_0) = (1,\mathbf{0})$ means that $y_0$ belongs to the same integer homology class as the all-one circulation, which encodes the empty stable set.
We remark that the proof in~\cite{ocpgenus} shows that each $W_i$ can be chosen such that $W_i$ does not use an edge more than twice.

\subsection{Sketches} \label{sec:sketches}

Consider any minimum cost slack vector $y \in \mathcal{Y}(G)$ and let $y_0$ denote the restriction of $y$ to $E(G_0)$. By Lemma~\ref{lem:convert}, we may assume that $y_0 \in \{0,1\}^{E(G_0)}$. By Proposition~\ref{prop:dual_representation}, $y_0$ is a 0/1-circulation in $D$ such that $\omega(y_0) = (1,\mathbf{0})$, see Fig.~\ref{figSketch1}. 

The aim of this section is to define a directed graph embedded in $\surf$ and with vertex set $\{f_1,\ldots,f_t\}$, which we call a ``sketch'' of $y_0$ (see below for the definition). This ``sketch'' will serve as roadmap for (the relevant part of) $y_0$. Its two main purposes are to motivate the inner workings of our algorithm, and also guarantee that the optimal solution $y$ is somehow represented in our DP. The ``sketch'' is not explicitly constructed or maintained by the algorithm. 

\begin{figure}[ht]
    \begin{center}
        \includegraphics[width=0.6\textwidth]{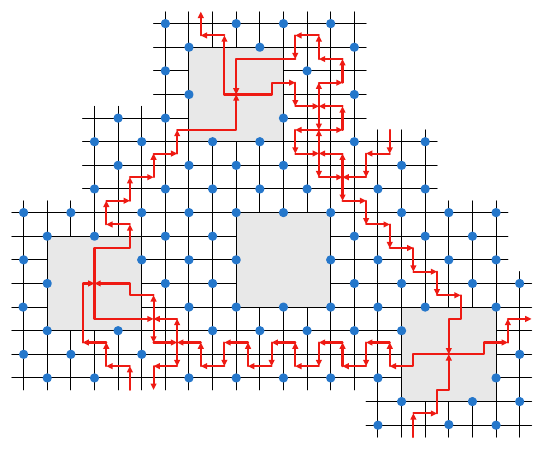}
    \end{center}    
    \caption{\label{figSketch1}Illustration of $G_0$ (black) with vortex faces (gray) and $y_0$ (red). Here, $y_0$ corresponds to the slack edges of a stable set in $G_0$ (blue).}
\end{figure}

Let $\supp(y_0) \subseteq A(D)$ denote the support of $y_0$, and let $\Sigma := (V(D),\supp(y_0))$ be the \emph{support graph} of $y_0$. Since $y_0$ is a 0/1-circulation, $\Sigma$ is a Eulerian subgraph of $D$. Consider a decomposition of $\Sigma$ into edge-disjoint directed cycles $C_1$, \ldots, $C_q$. We say that the decomposition is \emph{cross-free} if for every vertex $f \in V(\Sigma)$ and every two distinct cycles $C_i$ and $C_j$ through $f$, in the cyclic ordering around $f$ restricted to the four edges of $C_i \cup C_j$ incident to $f$, the two edges of $C_i$ are consecutive, as well as the two edges of $C_j$.

For the sake of completeness, we prove in \Cref{lem:cross-free_decomp} below that cross-free cycle decompositions exist. For a vertex $f \in V(\Sigma)$, we let $\delta_\Sigma^+(f)$ (resp.\ $\delta_\Sigma^-(f)$) denote the set of directed edges of $\Sigma$ leaving (resp.\ entering) $f$.

\begin{lemma} \label{lem:cross-free_decomp}
Every Eulerian directed graph $\Sigma$ embedded in $\surf$ admits a cross-free decomposition into edge-disjoint directed cycles.
\end{lemma} 

\begin{proof}
Let $f \in V(\Sigma)$ be an arbitrary vertex of $\Sigma$, and consider the directed edges $e^+ \in \delta_\Sigma^+(f)$ and $e^- \in \delta_\Sigma^-(f)$. Since $\Sigma$ is Eulerian, it has the same number of edges of each type for each $f$. Hence, we can match each $e^+$ with some $e^-$. We call such a perfect matching \emph{cross-free} if there are no two distinct matched pairs $\{e_1^-,e_1^+\}$ and $\{e_2^-,e_2^+\}$ such that, in the cyclic ordering around $f$, $e_2^-$ (resp.\ $e_2^+$) is between $e_1^-$ and $e_1^+$,  while $e_2^+$ (resp.\ $e_2^-$) is between $e_1^+$ and $e_1^-$. 

A cross-free perfect matching of the edges incident to any vertex $f$ always exists. An inductive way to build such a matching is to select one edge $e^- \in \delta^{-}_\Sigma(f)$ entering $f$ and one edge $e^+ \in \delta^{+}_\Sigma(f)$ leaving $f$ such that $e^-$ and $e^+$ are consecutive, match them, delete them and repeat.

After selecting a cross-free perfect matching for each vertex of $\Sigma$, we can define a corresponding decomposition of $\Sigma$ into edge-disjoint directed cycles $C_1, \ldots, C_q$ such that, for each vertex $f$ and each cycle $C_i$ through $f$, the two edges of $C_i$ incident to $f$ are matched together. Hence, $C_1, \ldots, C_q$ is a cross-free cycle decomposition of $\Sigma$.
\end{proof}

Let $C_1$, \ldots, $C_q$ be a cross-free cycle decomposition of $\Sigma$, the support graph of circulation $y_0$. After renumbering if necessary, we may assume that each cycle $C_i$ with $i \leq p$ goes through some vortex face $f_j$ (where $j \in [t]$), and that no cycle $C_i$ with $i > p$ goes through a vortex face.

\begin{figure}
\centering
\begin{tikzpicture}[scale=.35]
\tikzstyle{vtx}=[rectangle,draw,thick,inner sep=2.5pt,fill=white]
\tikzstyle{vtx0}=[rectangle,draw,thick,inner sep=2.5pt,fill=green]
\node[vtx0] (u) at (0,0) {};
\node[vtx] (v1) at (0:4) {};
\node[vtx] (v2) at (20:4) {};
\node[vtx] (v3) at (50:4) {};
\node[vtx] (v4) at (100:4) {};
\node[vtx] (v5) at (135:4) {};
\node[vtx] (v6) at (180:4) {};
\node[vtx] (v7) at (220:4) {};
\node[vtx] (v8) at (290:4) {};
\draw[ultra thick,red,->] (v1) -- (u);
\draw[ultra thick,red,->] (u) -- (v8);
\draw[ultra thick,teal,->] (v7) -- (u);
\draw[ultra thick,teal,->] (u) -- (v2);
\draw[ultra thick,blue,->] (v4) -- (u);
\draw[ultra thick,blue,->] (u) -- (v3);
\draw[ultra thick,orange,->] (v5) -- (u);
\draw[ultra thick,orange,->] (u) -- (v6);
\end{tikzpicture}
\qquad
\qquad
\begin{tikzpicture}[scale=.35]
\tikzstyle{vtx}=[rectangle,draw,thick,inner sep=2.5pt,fill=white]
\node[vtx] (v1) at (0:4) {};
\node[vtx] (v2) at (20:4) {};
\node[vtx] (v3) at (50:4) {};
\node[vtx] (v4) at (100:4) {};
\node[vtx] (v5) at (135:4) {};
\node[vtx] (v6) at (180:4) {};
\node[vtx] (v7) at (220:4) {};
\node[vtx] (v8) at (290:4) {};
\draw[ultra thick,red,->] (v1) .. controls ($(u)!0.01!(v1)$)
and ($(u)!0.01!(v8)$) .. (v8);
\draw[ultra thick,teal,->] (v7) .. controls ($(u)!0.01!(v7)$)
and ($(u)!0.01!(v2)$) .. (v2);
\draw[ultra thick,blue,->] (v4) .. controls ($(u)!0.01!(v4)$)
and ($(u)!0.01!(v3)$) .. (v3);
\draw[ultra thick,orange,->] (v5) .. controls ($(u)!0.01!(v5)$)
and ($(u)!0.01!(v6)$) .. (v6);
\end{tikzpicture}
\caption{\label{fig:cross-free} Cross-free cycle decomposition and local perturbation. Edges of the same color belong to the same cycle of the decomposition.}
\end{figure}

The \emph{sketch} is the directed graph $\tilde{\Sigma}$ embedded in $\surf$ with $V(\tilde{\Sigma}) := \{f_1,\ldots,f_t\}$ obtained by taking the union of the directed cycles $C_i$ with $i \leq p$, and slightly perturbing the embedding in a small neighborhood of each vertex $f \in V(D) \setminus \{f_1,\ldots,f_t\}$ in such a way to make all cycles disjoint in the neighborhood. Fig.~\ref{fig:cross-free} illustrates the local perturbation, and Fig.~\ref{figSketch2} illustrates a sketch $\tilde{\Sigma}$. Notice that each directed cycle $C_i$ with $i \leq p$ is either turned to a loop in $\tilde{\Sigma}$ or split in several directed edges of $\tilde{\Sigma}$ forming a directed cycle. Hence, loops of $\tilde{\Sigma}$ correspond to directed cycles in $\Sigma$ and directed edges of $\tilde{\Sigma}$ to directed paths in $\Sigma$.

\begin{figure}
    \begin{center}
        \includegraphics[width=0.6\textwidth]{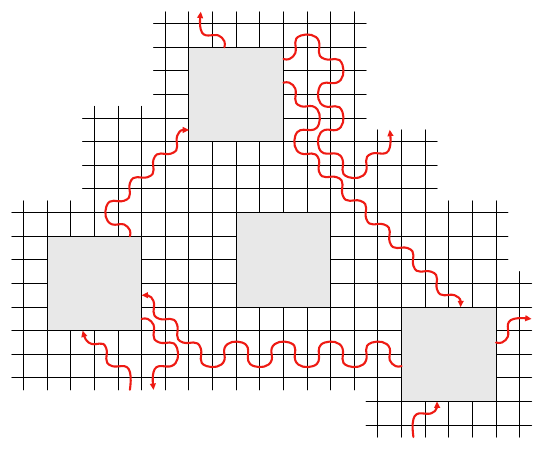}
    \end{center}    
    \caption{\label{figSketch2} Possible sketch (red) for the circulation from Fig.~\ref{figSketch1}.}
\end{figure}

Along with the sketch, we get a decomposition the circulation $y_0$ as a sum $y_{0} = y^{\mathrm{sketch}}_{0} + y^{\mathrm{non-sketch}}_{0}$ of two circulations:
$$
y_0 = \underbrace{\chi^{A(C_1)} + \cdots + \chi^{A(C_p)}}_{= y^{\mathrm{sketch}}_{0}} + \underbrace{\chi^{A(C_{p+1})} + \cdots + \chi^{A(C_q)}}_{= y^{\mathrm{non-sketch}}_{0}}\,.
$$
Notice that the first term correspond to cycles that are represented in the sketch (namely, $C_1$, \ldots, $C_p$), and the second one to cycles that are not (namely, $C_{p+1}$, \ldots, $C_q$). As a subset of $\surf$, the sketch $\tilde{\Sigma}$ is a perturbed version of the support graph of $y^{\mathrm{sketch}}_{0}$, which is a Eulerian subgraph of $\Sigma$. 

\subsection{Animated sketches and recursive surface decompositions} \label{sec:animated_sketches}

We point out that, although $\tilde{\Sigma}$ has a bounded number of vertices (namely, $t \le \vortexbd(k)$), we cannot necessarily bound its number of edges by a constant. In our DP, we will follow the construction of a sketch step by step but only remember certain parts of it.
To this end, consider a sketch $\tilde{\Sigma}$, and take an arbitrary linear ordering of its edges, say $\epsilon_1 < \cdots < \epsilon_r$. Together with this linear ordering, $\tilde{\Sigma}$ becomes an \emph{animated sketch}. For each $j \in \{0\} \cup [r]$, we let $\tilde{\Sigma}_j$ denote the subgraph of $\tilde{\Sigma}$ with $V(\tilde{\Sigma}_j) := V(\tilde{\Sigma}) = \{f_1,\ldots,f_t\}$ and $A(\tilde{\Sigma}_j) := \{\epsilon_1,\ldots,\epsilon_j\}$. We call each $\tilde{\Sigma}_j$ a \emph{frame} of the animated sketch $\tilde{\Sigma}$.

Now fix $j \in \{0\} \cup [r]$ and consider the faces of $\tilde{\Sigma}_j$. For technical reasons, we slightly change the definition of the faces of an embedding by taking a small \emph{open} neighborhood of the embedding before removing it from the surface and computing the connected components. Formally, for $X \subseteq \surf$ and $\varepsilon \in \R_{> 0}$, we let $N_\varepsilon(X)$ denote the set of points of $\surf$ that are at distance strictly less than $\varepsilon$ from some point of $X$.\footnote{We measure distances on $\surf$ using any suitable distance function, for instance endowing $\surf$ with a Riemannian structure.}
For $\delta > \varepsilon > 0$ sufficiently small to ensure genericity, we define the \emph{faces} of $\tilde{\Sigma}_j$ as the connected components of $\surf \setminus (N_\delta(\{f_1,\ldots,f_t\}) \cup N_\varepsilon(\{\epsilon_1,\ldots,\epsilon_j\}))$, see Fig.~\ref{fig:local}.\footnote{For the sake of simplicity, we denote by $f_i \in \surf$ the point of $\surf$ corresponding to the dual vertex $f_i$ (which is a vortex face), for $i \in [r]$, and by $\epsilon_i \subseteq \surf$ the simple curve in $\surf$ corresponding to directed edge $\epsilon_i$, for $i \in [r]$.} 
In general, the embedding of $\tilde{\Sigma}_j$ in $\surf$ is non-cellular since $\tilde{\Sigma}_j$ might have some faces that are not homeomorphic to (closed) disks. 
However, each face of $\tilde{\Sigma}_j$ is a \emph{surface with boundary}, that is, a topological space homeomorphic to a surface from which a bounded number of open disks with disjoint closures are removed. It is a well-known fact that the precise location of the disks in the surface is irrelevant, in the sense that the topological type of the resulting space only depends on the number of disks removed (and of course also on the initial surface).

\begin{figure}[ht]
\centering
\includegraphics[width=6cm]{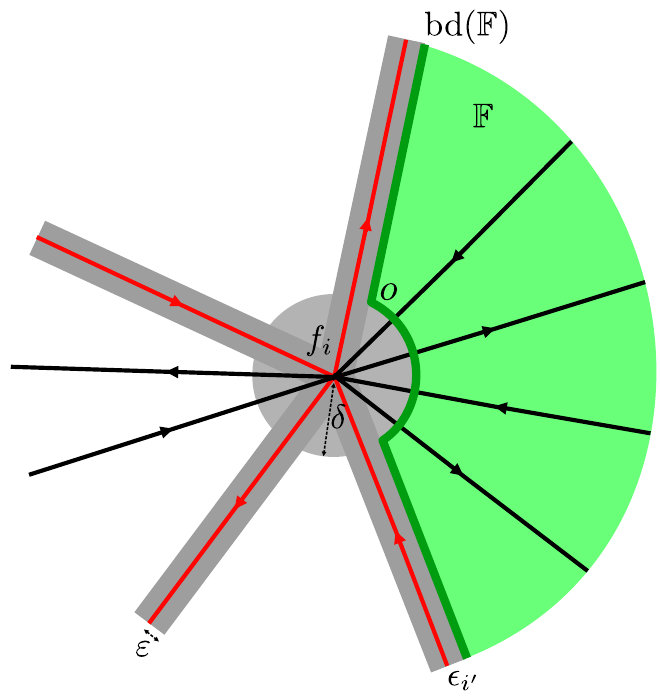}
\caption{Local view of $\tilde{\Sigma}_j$ around some $f_i$, where $i \in [t]$. The area in light green is contained in a face $\mathbb{F}$ of $\tilde{\Sigma}_j$. The dark green curve is $\bd(\mathbb{F})$. The light gray disk is $N_{\delta}(f_i)$ and the medium gray strips are $N_{\varepsilon}(\epsilon_{i'})$ for $i' \leq j$. The point $o$ is a cusp in $\bd(\mathbb{F})$ (see the definition below).}
\label{fig:local}
\end{figure}

Let $\mathbb{F}$ be a surface with boundary (for instance, $\mathbb{F}$ could be a   face of $\tilde{\Sigma}_j$). We denote by $\bd(\mathbb{F})$ the boundary of $\mathbb{F}$, which is a $1$-manifold and more precisely the union of finitely many disjoint simple closed curves. We denote by $\holes(\mathbb{F})$ the number of holes of $\mathbb{F}$, that is, the number of connected components of $\bd(\mathbb{F})$. If $\mathbb{F}$ has a non-empty boundary, then $\holes(\mathbb{F}) \ge 1$. For instance, if $\mathbb{F}$ is a disk then $\mathbb{F}$ has exactly one hole and thus $\holes(\mathbb{F})=1$. Moreover, if $\mathbb{F}$ is homeomorphic to a surface from which $h$ open disks with disjoint closures are removed, then $\holes(\mathbb{F}) = h$. We denote by $\hat{\mathbb{F}}$ the surface (without boundary) obtained from $\mathbb{F}$ by capping each hole with a disk (we sometimes say that $\mathbb{F}$ is \emph{$\hat{\mathbb{F}}$ with $h$ holes}). The Euler genus $\eg(\mathbb{F})$ of $\mathbb{F}$ is defined as the Euler genus of its capped version $\hat{\mathbb{F}}$. 

For $j = 0$, the frame $\tilde{\Sigma}_j$ has a single face which is simply $\surf$ with $t$ holes. Each time $j$ is increased by $1$, a single edge gets added to the frame. For each $j \in [r]$, edge $\epsilon_j$ intersects a unique face of $\tilde{\Sigma}_{j-1}$, which we denote by $\mathbb{F}_{j-1}$. When $\epsilon_j$ is added to $\tilde{\Sigma}_{j-1}$ to form $\tilde{\Sigma}_j$, the unique face that changes is $\mathbb{F}_{j-1}$. All the other faces remain unchanged. The following result describes the different cases that arise.


\begin{lemma} \label{lem:topological_op}
Assume the notation above. For each fixed $j \in [r]$, one of the following mutually exclusive cases arises when the edge $\epsilon_j$ is added to frame $\tilde{\Sigma}_{j-1}$. For simplicity, we denote $\epsilon_j$ by $\epsilon$, and $\mathbb{F}_{j-1}$ by $\mathbb{F}$ below. 

\begin{description}
\item[Merge operation:] $\epsilon$ connects two distinct holes of $\mathbb{F}$. Removing $N_\varepsilon(\epsilon)$ from face $\mathbb{F}$ does not disconnect the face, and merges the two holes that are connected by $\epsilon$. This operation decreases the number of holes of $\mathbb{F}$ by $1$, and does not change its Euler genus.

\item[Simplify operation:] $\epsilon$ intersects a single hole and does not separate $\mathbb{F}$. Let $\hat{\epsilon}$ be any simple closed curve in $\hat{\mathbb{F}}$ arising by completing $\epsilon$ through the hole. Removing $N_\varepsilon(\epsilon)$ from $\mathbb{F}$ decreases the Euler genus by $1$ if $\hat{\epsilon}$ is $1$-sided, or $2$ if $\hat{\epsilon}$ is $2$-sided. In the first case, the number of holes of $\mathbb{F}$ remains the same and in the second case it goes up by $1$. 

\item[Split operation:] $\epsilon$ intersects a single hole and separates $\mathbb{F}$. Removing $N_\varepsilon(\epsilon)$ from $\mathbb{F}$ splits the face into two surfaces with boundary $\mathbb{F}'_1$ and $\mathbb{F}'_2$ such that $\eg(\mathbb{F}'_{1}) + \eg(\mathbb{F}'_{2}) = \eg(\mathbb{F})$, $\holes(\mathbb{F}'_1) + \holes(\mathbb{F}'_2) = \holes(\mathbb{F}) + 1$ and $\holes(\mathbb{F}'_i) \geq 1$ for $i \in [2]$.
\end{description}
\end{lemma}

\begin{proof}
Let $h := \holes(\mathbb{F})$, thus $\mathbb{F}$ is (homeomorphic to) $\hat{\mathbb{F}}$ with $h$ holes, and let $\Delta_1$, \ldots, $\Delta_h$ disjoint (closed) disks in $\hat{\mathbb{F}}$ such that $\mathbb{F}$ is $\hat{\mathbb{F}} \setminus (\inter(\Delta_1) \cup \cdots \cup \inter(\Delta_h))$.

First, assume that $\epsilon$ connects two different components of $\bd(\mathbb{F})$, say $\bd(\Delta_1)$ and $\bd(\Delta_2)$. In $\hat{\mathbb{F}}$, the union of $N_\varepsilon(\epsilon)$, $\inter(\Delta_1)$ and $\inter(\Delta_2)$ equals $\inter(\Delta')$ for some disk $\Delta'$ disjoint from $\Delta_3$, \ldots, $\Delta_h$. Hence, $\mathbb{F} \setminus N_\varepsilon(\epsilon) = \hat{\mathbb{F}} \setminus (\inter(\Delta_1) \cup \cdots \cup \inter(\Delta_h) \cup N_\varepsilon(\epsilon)) = \hat{\mathbb{F}} \setminus (\inter(\Delta') \cup \inter(\Delta_3) \cup \cdots \cup \inter(\Delta_h))$. It follows that $\mathbb{F} \setminus N_\varepsilon(\epsilon)$ is homeomorphic to $\hat{\mathbb{F}}$ with $h-1$ holes. Hence, the first case of the statement holds.

From now on, assume that $\epsilon$ intersects a single component of $\bd(\mathbb{F})$, say $\bd(\Delta_1)$. Notice that $\epsilon$ separates $\mathbb{F}$ if and only if $\hat{\epsilon}$ separates $\hat{\mathbb{F}}$. In order to understand the effect of removing $N_\varepsilon(\epsilon)$ from $\mathbb{F}$, we remove $N_\varepsilon(\hat{\epsilon})$ from $\hat{\mathbb{F}}$. This is a standard operation, see for instance \cite[Section 7.4]{Armstrong83} (removing $N_\varepsilon(\epsilon)$ from $\hat{\mathbb{F}}$, in case $\hat{\epsilon}$ does not separate $\hat{\mathbb{F}}$, is the first step of what is known as a \emph{surgery}). Notice that $N_\varepsilon(\hat{\epsilon})$ is either an open cylinder or an open M\"obius strip, which leads to the following two cases.\medskip

\noindent\emph{Case 1. $N_\varepsilon(\hat{\epsilon})$ is an open cylinder.} 
If $\hat{\epsilon}$ separates $\hat{\mathbb{F}}$, then removing $N_\varepsilon(\hat{\epsilon})$ from $\hat{\mathbb{F}}$ produces a topological space that is the disjoint union $\hat{\mathbb{F}}_1 \cup \hat{\mathbb{F}}_2$ of two surfaces with boundary, each with one hole. Removing $N_\varepsilon(\hat{\epsilon})$ from $\Delta_1$ splits it into two disks $\Delta_{1,1}$ and $\Delta_{1,2}$, where $\Delta_{1,i}$ is contained in $\hat{\mathbb{F}}_i$ for each $i \in [2]$. Removing $\inter(\Delta_{1})$ from $\hat{\mathbb{F}}_i$ just makes the hole of $\hat{\mathbb{F}}_i$ larger, but does not change its topological type. Each one of $\Delta_2$, \ldots, $\Delta_h$ is either in $\hat{\mathbb{F}}_1$ or in $\hat{\mathbb{F}}_2$. We conclude that the third case of the statement holds.

If $\hat{\epsilon}$ does not separate $\hat{\mathbb{F}}$, then removing $N_\varepsilon(\hat{\epsilon})$ from $\hat{\mathbb{F}}$ yields a surface with boundary $\hat{\mathbb{F}}'$ with $\holes(\hat{\mathbb{F}}') = 2$ and $\eg(\hat{\mathbb{F}}') = \eg(\hat{\mathbb{F}}) - 2$. Again, disk $\Delta_1$ gets split into two disjoint disks $\Delta_{1,1}$ and $\Delta_{1,2}$, this time both in the same component. Removing $\inter(\Delta_{1})$ from $\hat{\mathbb{F}}'$ just makes the two holes of $\hat{\mathbb{F}}'$ larger, but nothing else. Here, the second case of the statement holds: we have a simplify operation where $\hat{\epsilon}$ is $2$-sided. Going from $\mathbb{F}$ to $\mathbb{F} \setminus N_\varepsilon(\epsilon)$, the number of holes goes up by $1$.\medskip

\noindent\emph{Case 2. $N_\varepsilon(\hat{\epsilon})$ is a M\"obius strip.} Then $\hat{\epsilon}$ does not separate $\hat{\mathbb{F}}$. Removing $N_\varepsilon(\hat{\epsilon})$ from $\hat{\mathbb{F}}$ yields a surface with boundary $\hat{\mathbb{F}}'$ with $\holes(\hat{\mathbb{F}}') = 1$ and $\eg(\hat{\mathbb{F}}') = \eg(\hat{\mathbb{F}}) - 1$. Once more, $\Delta_1$ gets split into two disjoint disks $\Delta_{1,1}$ and $\Delta_{1,2}$. Removing $\inter(\Delta_{1,1})$ and $\inter(\Delta_{1,2})$ from $\hat{\mathbb{F}}'$ enlarges the unique hole of $\hat{\mathbb{F}}'$. The second case of the statement holds: we have a simplify operation where $\hat{\epsilon}$ is $1$-sided. We see that $\holes(\mathbb{F} \setminus N_\varepsilon(\epsilon)) = \holes(\mathbb{F})$.
\end{proof}

\begin{figure}[ht]
\centering
\includegraphics[width=0.7\textwidth]{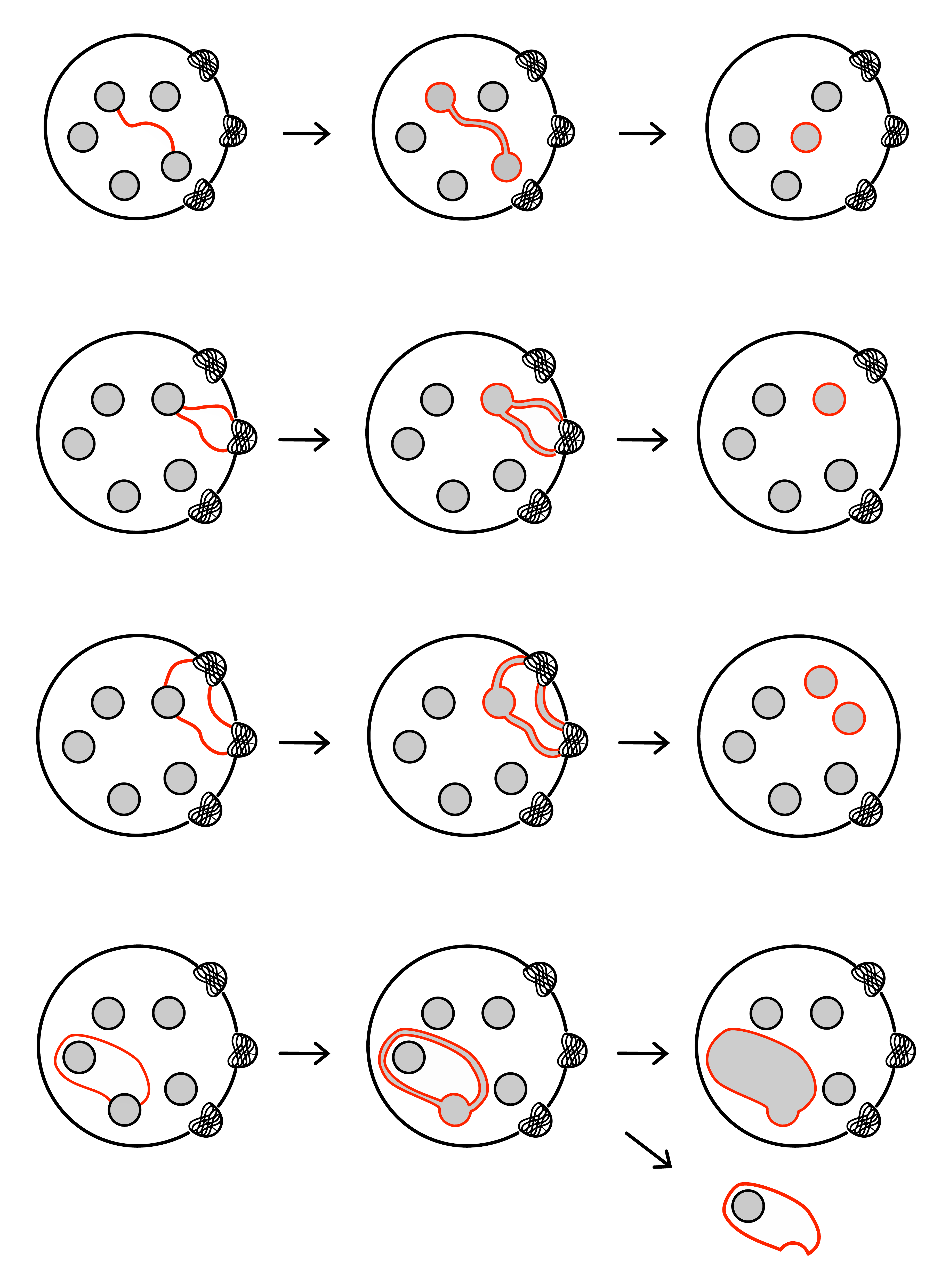}
\caption{The operations arising when an edge is added to a frame. From top to bottom: merge, simplify (1-sided), simplify (2-sided), and split.}
\label{fig:ops}
\end{figure}

In our DP, we will capture combinatorially what is happening topologically on the surface when the sketch $\tilde{\Sigma}$ is built one edge at a time. Each possible animated sketch $(\tilde{\Sigma},<)$ has a corresponding out-arborescence $A = A(\tilde{\Sigma},<)$ contained in the dependence graph of the DP, which is a directed acyclic graph. The out-arborescence $A$ is meant to represent the recursive decomposition of $\surf$ defined by the animated sketch. The vertices of $A$ bijectively correspond to the faces $\mathbb{F}_j$, where $j \in [r]$, in a min-heap fashion. That is, if the face corresponding to a non-root vertex is $\mathbb{F}_j$ then the label of its parent is $\mathbb{F}_{j'}$ for some $j' < j$. The root of $A$ corresponds to $\mathbb{F}_0$, the unique face of $\tilde{\Sigma}_0$ (that is, $\surf$ with $t$ holes). The inner vertices of $A$ correspond to the faces of the frames $\tilde{\Sigma}_j$ with $j < r$ that are modified in a later frame. The leaves of $A$ correspond to the faces of the final frame $\tilde{\Sigma}_r = \tilde{\Sigma}$. Each non-leaf vertex of $A$ has one or two children, depending on which case of \Cref{lem:topological_op} holds for the corresponding face $\mathbb{F} = \mathbb{F}_{j-1}$. In case of a merge or simplify operation, the vertex has one child. In case of a split operation, the vertex has two children.

In the next two sections, we develop the definitions of several combinatorial objects that are used in our algorithm, and inspired by animated sketches and the corresponding recursive surface decompositions discussed in this section. We follow this with a section describing and analyzing the DP.

\subsection{Sectors, clocks and windows} \label{sec:sectors_and_clocks}


We define combinatorial objects that will allow us to encode (relevant parts of) the faces of the frames $\tilde{\Sigma}_0$, $\tilde{\Sigma}_1$, \ldots, $\tilde{\Sigma}_r$. Right after stating the definitions, we relate them to the topological considerations of \Cref{sec:sketches} and \Cref{sec:animated_sketches}.

A \emph{proper sector} is a sequence $\sigma=(e_0;e_1,e_2,\ldots,e_q;e_{q+1})$ of directed edges of $D$ ($q \in \Z_{\geqslant 0}$), all incident to some common vortex face $f_i \in V(D)$, and consecutive in some cyclic ordering of the edges incident to $f_i$. The edges $e_1, \ldots, e_q$ are called the \emph{internal} edges and are assumed to be distinct. 
The edges $e_0$ and $e_{q+1}$ are called \emph{delimiters}. We assume that these are distinct from $e_1, \ldots, e_q$, thus no delimiter can be an internal edge. Edge $e_0$ is the \emph{left delimiter}, and edge $e_{q+1}$ the \emph{right delimiter}. Possibly, $e_0 = e_{q+1}$, in which case all the edges incident to $f_i$ distinct from the delimiter appear as internal edges. A \emph{full sector} $\sigma=(;e_1,e_2,\ldots,e_q;)$ is defined similarly. There is no delimiter, and all the edges incident to $f_i$ appear.

A \emph{clock} is any sequence $C = (\sigma_1,\ldots,\sigma_\ell)$ of internally disjoint (proper or full) sectors. 

A \emph{window} is a pair $\mathcal{W} = (\mathcal{C}(\mathcal{W}),\eg(\mathcal{W}))$, where $\mathcal{C}(\mathcal{W})$ is a set of disjoint clocks, and $\eg(\mathcal{W})$ is an integer in $\{0,\ldots,g\}$, where $g \le \vortexbd(k)$ is the Euler genus of $\surf$. If $\mathcal{W}$ is a window, we let $\Del(\mathcal{W})$ denote the set of all delimiters in $\mathcal{W}$. That is, $\Del(\mathcal{W})$ is the set of edges $e$ that appear as delimiters of at least one sector inside a clock of $\mathcal{W}$. We let $\del(\mathcal{W})$ denote the total number of delimiters in $\mathcal{W}$, and $||\mathcal{W}||$ denote the total number of edges in $\mathcal{W}$ (both delimiters and internal edges). Also, we let $\holes(\mathcal{W}) := |\mathcal{C}(\mathcal{W})|$ denote the number of clocks of $\mathcal{W}$. It will follow from the analysis done later that $\holes(\mathcal{W}) \in \{1,\ldots,t+g\}$.

Windows are meant to represent faces $\mathbb{F} = \mathbb{F}_j$ of frames $\tilde{\Sigma}_j$. 

The clocks of a window represent the different holes of the corresponding face $\mathbb{F}$, that is, the different components of $\bd(\mathbb{F})$, see Fig.~\ref{fig:window}.

The interpretation of the sectors within a clock is as follows. Let $\vartheta$ denote a component of $\bd(\mathbb{F})$. 
Thanks to genericity, each point of $\vartheta$ is in at most one of the sets $\bd(N_{\delta}(\{f_1\}))$, \ldots, $\bd(N_{\delta}(\{f_t\}))$, and in at most one of the sets $\bd(N_{\varepsilon}(\epsilon_1))$, \ldots, $\bd(N_{\varepsilon}(\epsilon_j))$, where $\epsilon_1$, \ldots, $\epsilon_j$ denote the edges of frame $\tilde{\Sigma}_j$. We call \emph{cusp} any point of $\vartheta$ that is at the same time at distance $\delta$ from some $f_i$ ($i \in [t]$) and at distance $\varepsilon$ from some $\epsilon_{i'}$ ($i' \in [j]$). An example of a cusp can be found in \Cref{fig:local}. By removing the cusps, we split $\vartheta$ into a finite number of connected components which we call \emph{pieces}. The pieces are of two types: \emph{vertex pieces} are those which are contained in $\bd(N_{\delta}(\{f_i\}))$ for some $i \in [t]$, and \emph{edge pieces} those which are contained in $\bd(N_{\delta}(\epsilon_{i'}))$ for some $i' \in [j]$. In the combinatorial representation of $\vartheta$, we essentially ignore the edge pieces and focus on the vertex pieces, which correspond to the sectors within the clock representing $\vartheta$. 

A vertex piece may take the whole boundary of some disk $N_\delta(f_i)$, in which case the corresponding sector is full. Otherwise, vertex pieces are homeomorphic to open intervals. Consider such a vertex piece $\pi$ and its two endpoints, say $p_1$ and $p_2$, which are distinct cusps of $\vartheta$. There exist unique indices $i_1, i_2 \in [j]$ such that $p_1$ (resp.\ $p_2$) is at distance $\varepsilon$ from $\epsilon_{i_1}$ (resp.\ $\epsilon_{i_2}$). Notice that $i_1 = i_2$ is possible. Recall that sketch $\tilde{\Sigma}$ closely follows the support graph of $y^{\mathrm{sketch}}_{0}$, which is itself a subgraph of $D$. There are well-defined edges of $D$ that are extended by $\epsilon_{i_1}$ and $\epsilon_{i_2}$. These are the delimiters of the sectors, and were denoted by $e_0$ and $e_{q+1}$ above. The internal edges of the sector are the edges of $D$ that are intersected by $\pi$.

\begin{figure}
    \begin{center}
        \includegraphics[width=0.5\textwidth]{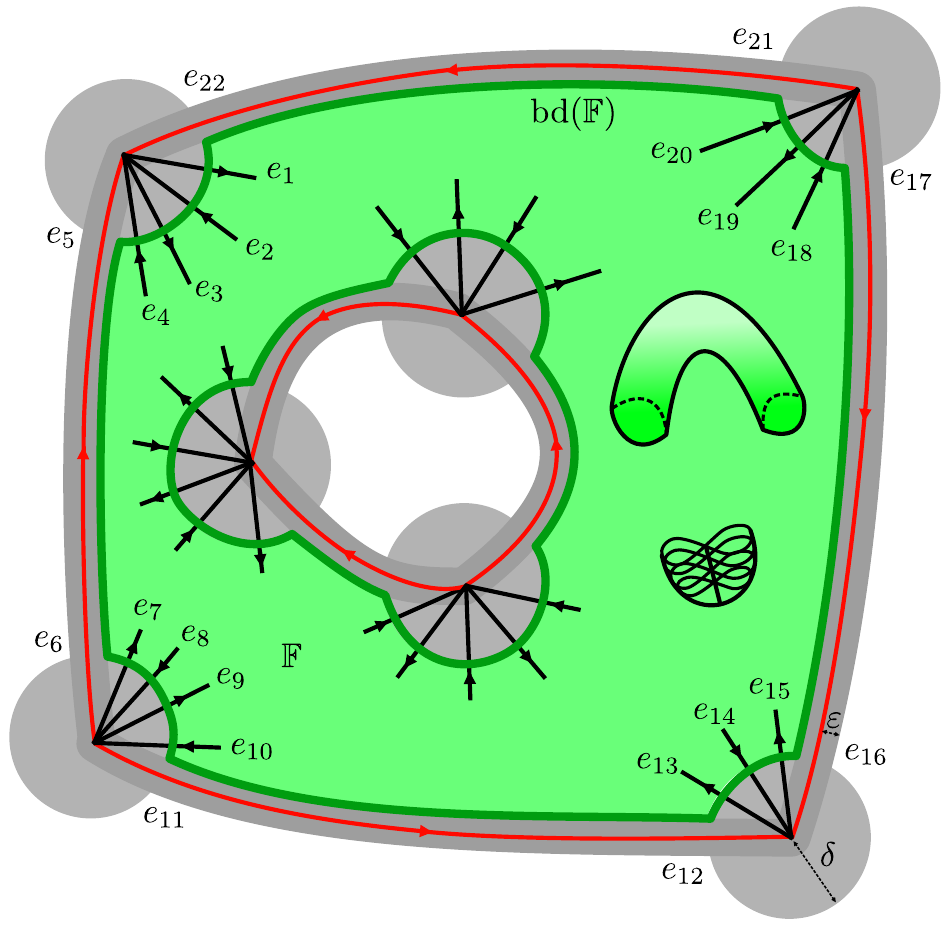}
    \end{center}    
    \caption{\label{fig:window}A window (green) with two clocks, and Euler genus $3$. The displayed clock is $((e_{22};e_1,\ldots,e_4;e_5), (e_6;e_7,\ldots,e_{10};e_{11}), (e_{12};e_{13},e_{14},e_{15};e_{16}),(e_{17};e_{18},e_{19},e_{20};e_{21}))$.}
\end{figure}

\subsection{Operations on windows} \label{sec:operations}

The aim of this section is to define combinatorial versions of the merge, simplify and split operations introduced in \Cref{lem:topological_op}. We begin by discussing some topological and algorithmic aspects. Consider the addition of edge $\epsilon := \epsilon_j$ to frame $\tilde{\Sigma}_j$, inside face $\mathbb{F} := \mathbb{F}_{j-1}$. 
By removing $N_{\varepsilon}(\epsilon)$ from $\mathbb{F}$, we obtain either one new face $\mathbb{F}'$ (for merge or simplify operations) or two new faces $\mathbb{F}'_1$ and $\mathbb{F}'_2$ (for split operations).

Let $\vartheta_1$ and $\vartheta_2$ denote the components of $\bd(\mathbb{F})$ intersected by $\epsilon$ (we allow $\vartheta_1 = \vartheta_2$ for simplify and split operations). For each $\vartheta_i$, pick a cusp $o_i$ and a direction of travel of $\vartheta_i$. Then, consider the vertex pieces of $\vartheta_i$ in the order in which they are visited when one goes around $\vartheta_i$ from $o_i$ in the chosen direction of travel. Combinatorially, these vertex pieces are represented by the sectors of the corresponding clock, and the ordering of the pieces is recorded in the clock. Removing $N_{\varepsilon}(\epsilon)$ from $\mathbb{F}$ changes the vertex pieces of each $\vartheta_i$ in certain ways, which are described below. Before delving more deeply into this, we discuss an important detail related to the linear ordering of the edges of the animated sketch.

A key observation for setting up our DP is that most of our operations decrease the quantity $\holes(\mathbb{F}) + \eg(\mathbb{F})$. More precisely, we have $\holes(\mathbb{F}') +\eg(\mathbb{F}') = \holes(\mathbb{F}) +\eg(\mathbb{F}) - 1$ for merge and simplify operations. For split operations, in all cases except one, we have $\holes(\mathbb{F}'_i) +\eg(\mathbb{F}'_i) \le \holes(\mathbb{F}) +\eg(\mathbb{F}) - 1$ for all $i \in [2]$. The only exception arises when one of the two connected components of $\mathbb{F} \setminus N_{\varepsilon}(\epsilon)$ is a disk. In that case, we have $\holes(\mathbb{F'}_i) = 1$ and $\eg(\mathbb{F}'_i) = 0$ for some $i \in [2]$, and hence $\holes(\mathbb{F'}_{3-i}) + \eg(\mathbb{F'}_{3-i}) = \holes(\mathbb{F}) + \eg(\mathbb{F})$ for the other connected component.

In order to avoid building an exponentially large DP table, we have to carefully deal with split operations where one of the two resulting faces is a disk. Toward this aim, we take advantage of the fact that we encode holes of $\mathbb{F}$ by  sequences (which are linearly ordered, with well-defined initial and final elements), and constrain the linear ordering of the edges of the animated sketch. 

Assume that $\epsilon$ intersects a single hole and separates $\mathbb{F}$. Let $\vartheta = \vartheta_1 = \vartheta_2$ denote the component of $\bd(\mathbb{F})$ intersected by $\epsilon$. Let $o = o_1 = o_2$ denote the cusp serving as the origin of $\vartheta$. Let $p$ denote the first point of $\epsilon$ encountered when leaving $o$ in the chosen direction of travel. We choose the linear ordering $<$ of the edges of $\tilde{\Sigma}$ in such a way that no edge $\epsilon'$ with $\epsilon < \epsilon'$ intersects the part of $\vartheta$ between $o$ and $p$ (we say that these points are to the \emph{left} of $p$). At the combinatorial level, this will be achieved by \emph{deleting} the corresponding edges in the clock that encodes $\vartheta$.

Let $\mathcal{W}$ be a window, and let $C$ be a clock. Let $e$ denote a directed edge of $D$. If $e$ is an internal edge of (one of the sectors of) $C$, we let $Ce$ denote the clock obtained by deleting all the edges and sectors to the right of $e$, and making $e$ the right delimiter of the sector it belongs to. The clock $eC$ is defined similarly. We let $C^{-1}$ denote the clock $C$ read backwards (from right to left). If $C_1$, \ldots, $C_p$ are internally disjoint clocks, we denote by $(C_1,\ldots,C_p)$ the clock obtained by chaining them.

Let $C$ and $C'$ be clocks of $\mathcal{W}$ (possibly $C = C'$), let $e$ be an internal edge of $C$, and $e'$ be an internal edge of $C'$ (possibly $e = e'$, but only if $C \neq C'$). The merge and simplify operations yield a new window $\mathcal{W}'$. The split operation yields two new windows $\mathcal{W}'_1$ and $\mathcal{W}'_2$. These are defined as follows, see \Cref{fig:ops,figOperationsCombinatorial}. We point out that the window(s) resulting from the operations are not uniquely defined. However, there are a bounded number of choices for what they can be.

\begin{figure}[ht]
\centering
\begin{tabular}{c@{\qquad}c}
\includegraphics[width=0.4\textwidth]{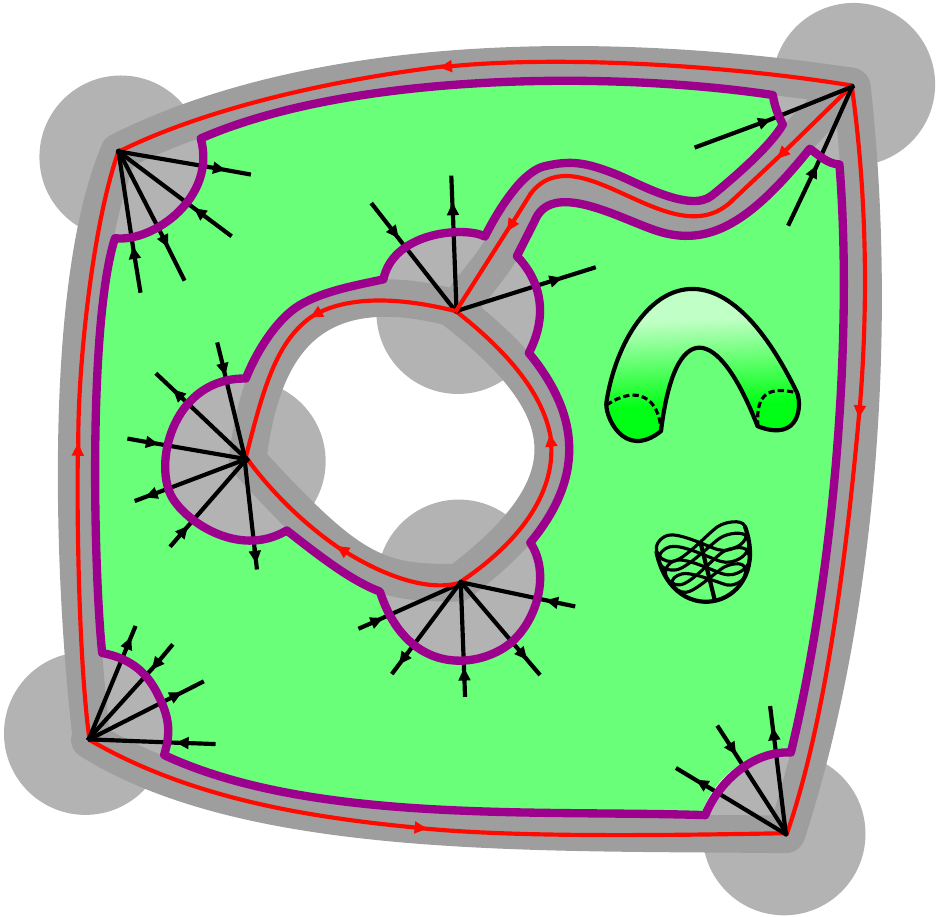} 
&\includegraphics[width=0.4\textwidth]{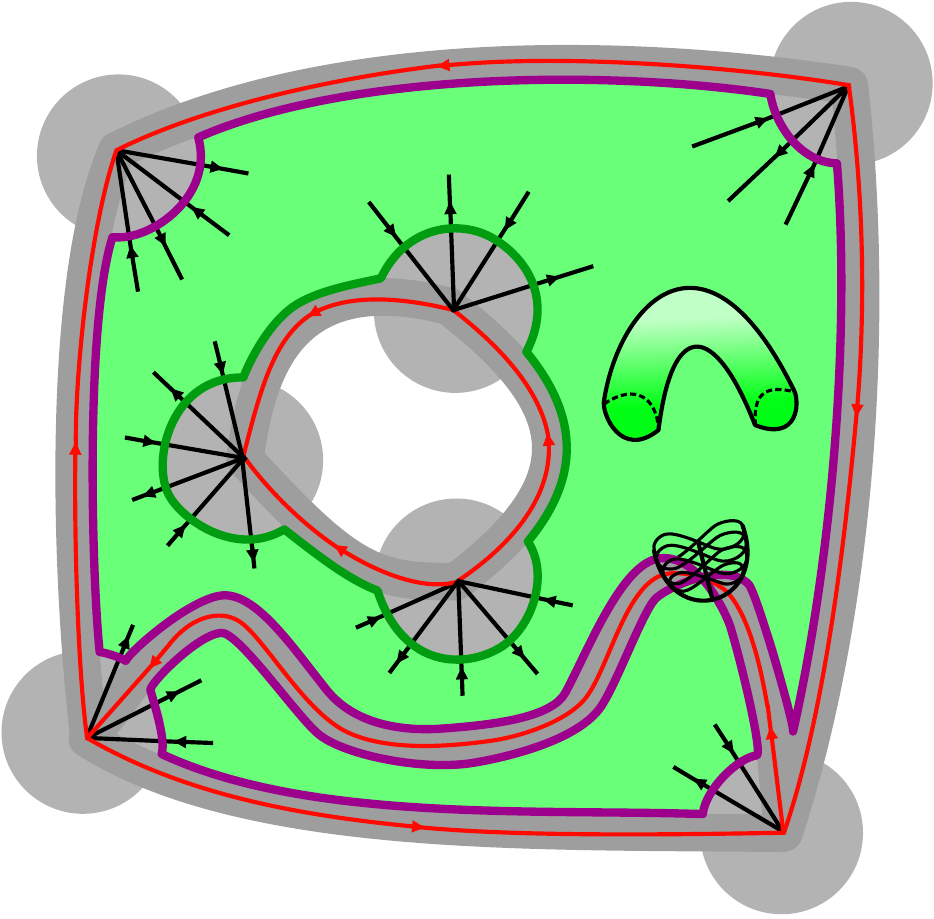}\\
\includegraphics[width=0.4\textwidth]{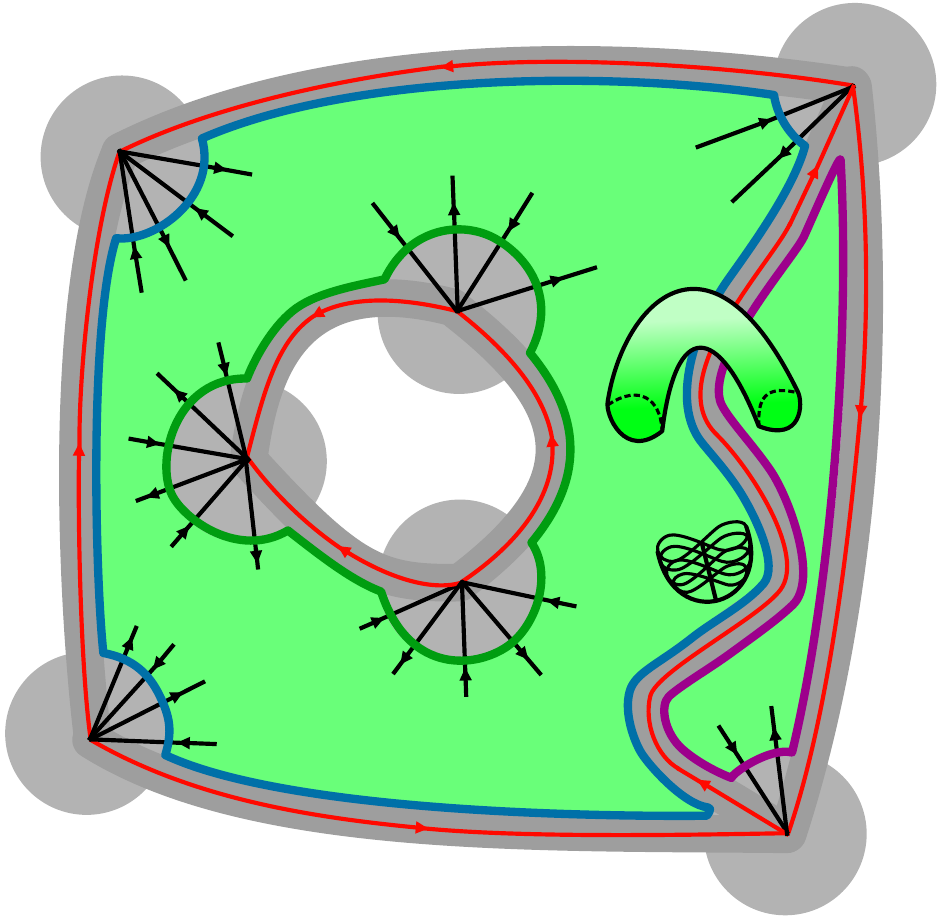} 
&\includegraphics[width=0.4\textwidth]{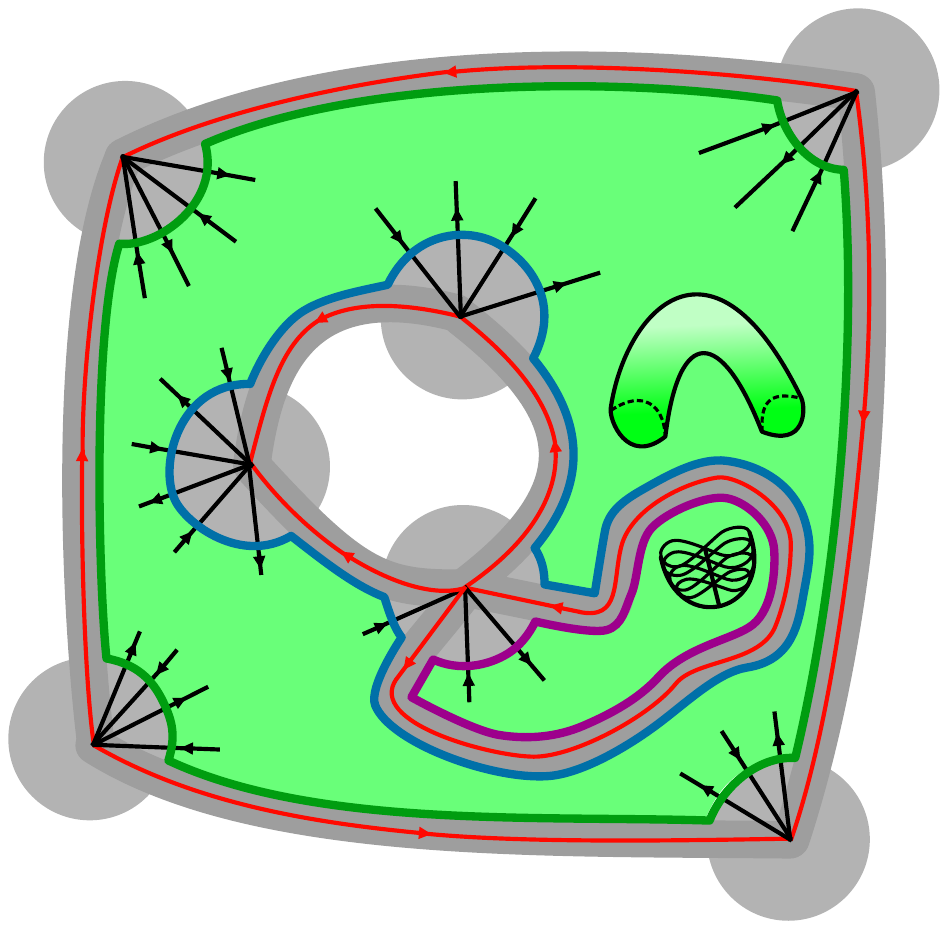}\\
\end{tabular}
\caption{\label{fig:ops}Operations performed on the window of \Cref{fig:window}. From left to right: merge, simplify ($1$-sided), simplify ($2$-sided), split. The new windows are depicted in magenta and blue.}
\end{figure}

\begin{description}
\item[Merge operation (combinatorial)] Assume $C \neq C'$. We \emph{merge} $C$ and $C'$ through $e$ and $e'$ by replacing $C$ and $C'$ by a single clock that is either $(C e,e'C',C'e',eC)$ or $(C e,(C'e')^{-1},(e' C')^{-1},eC)$. Hence, $\holes(\mathcal{W}') = \holes(\mathcal{W}) - 1$. We let $\eg(\mathcal{W}') := \eg(\mathcal{W})$.

\item[Simplify operation (combinatorial)] Assume that $C = C'$ and $e$ is strictly to the left of $e'$ in clock $C$. We \emph{simplify} $C$ through $e$ and $e'$ by replacing $C$ either by the single clock $(Ce,(eCe')^{-1},e'C)$, or by the two clocks $(Ce,e'C)$ and $(eCe')$. In the first case, we have $\holes(\mathcal{W}') = \holes(\mathcal{W})$ and let $\eg(\mathcal{W}') := \eg(\mathcal{W}) - 1$ and in the second case we have $\holes(\mathcal{W}') = \holes(\mathcal{W})+1$ and let $\eg(\mathcal{W}') := \eg(\mathcal{W}) - 2$.

\item[Split operation (combinatorial)] Assume that $C = C'$ and $e$ is strictly to the left of $e'$ in clock $C$. We \emph{split} $C$ through $e$ and $e'$ by replacing clock $C$ by the two clocks $C''_1 := e'C$ and $C''_2 := eCe'$, and partitioning the resulting window into two (non-empty) windows $\mathcal{W}'_1$ and $\mathcal{W}'_2$, one containing $C''_1$ and the other containing $C''_2$. Hence, $\holes(\mathcal{W}'_1) + \holes(\mathcal{W}'_2) = \holes(\mathcal{W}) + 1$ and $\holes(\mathcal{W}'_i) \ge 1$ for all $i \in [2]$. Notice that the sub-clock $Ce$ gets deleted in the process, in accordance to our precedence rule for animated sketches. We let $\eg(\mathcal{W}'_{1})$ and  $\eg(\mathcal{W}'_{2})$ be nonnegative integers such that $\eg(\mathcal{W}'_{1}) + \eg(\mathcal{W}'_{2}) = \eg(\mathcal{W})$.
\end{description}

As expected, no operation increases $\holes(\mathcal{W}) + \eg(\mathcal{W})$. We formally record this in the following result.

\begin{lemma}\label{lem:h+eg}
For each merge or simplify operation, we have $\holes(\mathcal{W}') + \eg(\mathcal{W}') = \holes(\mathcal{W}) + \eg(\mathcal{W}) - 1$. Moreover, for each split operation and $i \in [2]$, we have $\holes(\mathcal{W}'_i) + \eg(\mathcal{W}'_i) \le \holes(\mathcal{W}) + \eg(\mathcal{W})$.
\end{lemma}

\begin{proof}
The first part of the statement follows directly from the definition of the merge and simplify operations. For any split operation, we have 
$$
\holes(\mathcal{W}'_1) +  \eg(\mathcal{W}'_{1}) + \holes(\mathcal{W}'_2) + \eg(\mathcal{W}'_{2}) = \holes(\mathcal{W}) + \eg(\mathcal{W}) + 1\,.
$$
Since we also have $\holes(\mathcal{W}'_i) \ge 1$ for all $i \in [2]$, the second part of the statement follows.
\end{proof}

\begin{figure}[ht!]
    \centering
    \includegraphics[width=0.8\textwidth]{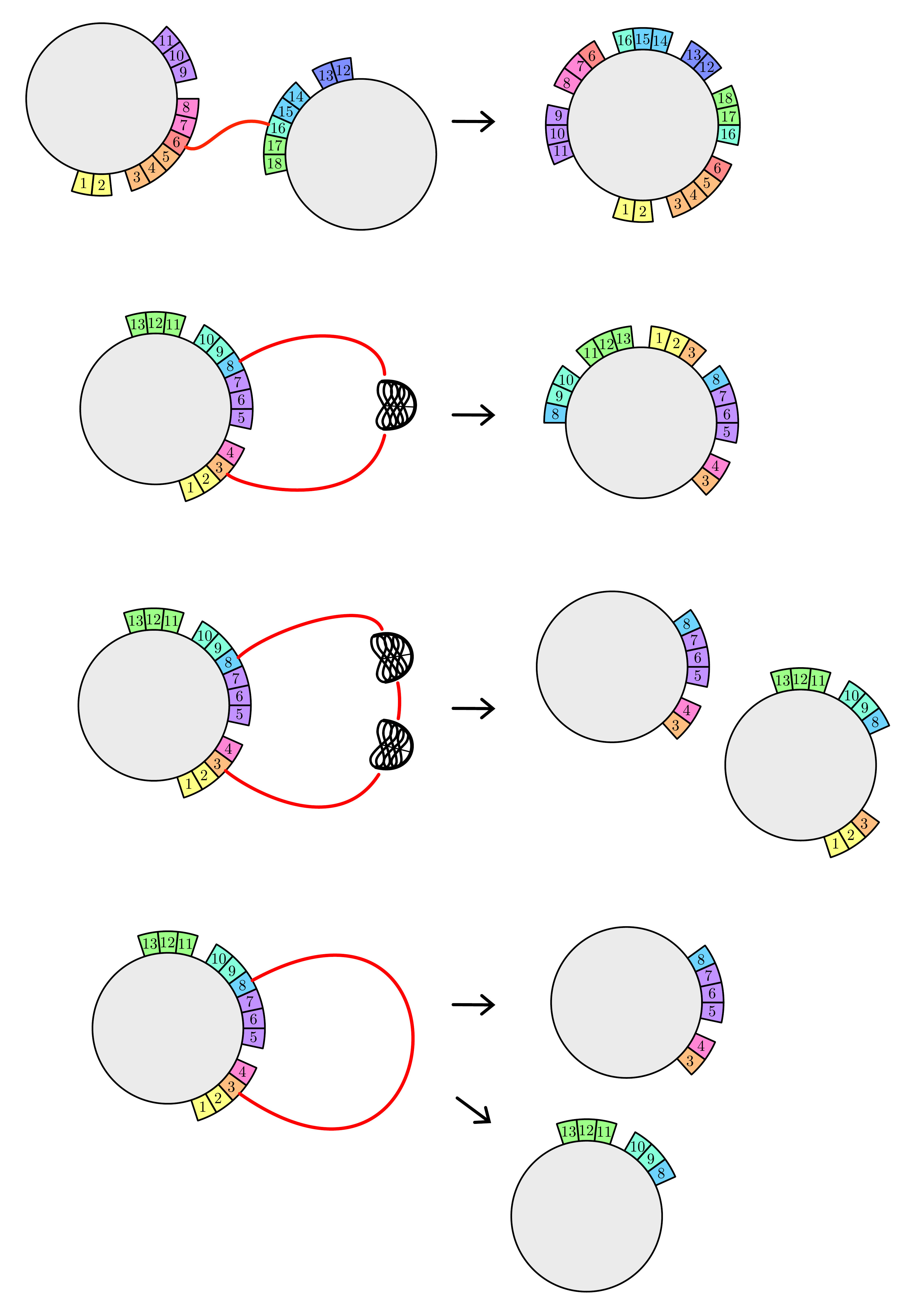}
    \caption{\label{figOperationsCombinatorial}Combinatorial operations on windows. From top to bottom: merge, simplify (1-sided), simplify (2-sided), and split.}
\end{figure}

\subsection{Dynamic program}

%

This section describes a strongly-polynomial time algorithm for computing a minimum $c$-cost slack vector of $G$. Recall that $G$ is a graph as in case (II) of \Cref{cor:structure_DP} and that $c \in \Q^{E(G)}_{\ge 0}$ are given edge costs. 

\subsubsection{Cutsets and local optimization in vortices} 

Consider a vortex $V_i = (G_i,\Omega_i)$, where $i \in [t]$. Let $u_1$, \ldots, $u_n$ denote the vertices of $\Omega_i$, enumerated consecutively. By assumption, $u_1$, \ldots, $u_n$ is also a consecutive enumeration of all the vertices on the boundary of face $f_i$. Each directed edge $e \in A(D)$ incident to $f_i$ corresponds to a unique edge $u_ju_{j+1}$ of the boundary cycle of $G_i$. Below, we do not always distinguish between a dual edge in $\delta^+(f_i) \cup \delta^-(f_i)$ and its corresponding primal edge in the boundary cycle of $G_i$.

Let $(X_1,\ldots,X_n)$ denote a linear decomposition of $V_i$ with adhesion at most $\adhesionbd(k) + 2$. The set $X_i$ was called a \emph{bag} and denoted by $X(u_i)$ in \Cref{sec:intro}. Here, we prefer to use the shorter notation $X_i$. Recall that for every $v \in V(G_i)$, the set $\{j \in [n] : X_j \ni v\}$ is an (integer) interval. Let $e := u_ju_{j+1}$ denote any edge of $G_0 \cap G_i$, with $j < n$. We define
$$
Y(e) := (X_{j-1} \cap X_{j}) \cup (X_j \cap X_{j+1}) \cup (X_{j+1} \cap X_{j+2}) \cup \{u_j,u_{j+1}\}\,,
$$ 
where $X_0 := X_{n+1} := \varnothing$. The vertices of $G_i - Y(e)$ can be partitioned into four parts, namely, the vertices whose interval ends before index $j$, the vertices whose interval is $\{j\}$, the vertices whose interval is $\{j+1\}$, and the vertices whose interval starts after index $j+1$. Every two vertices in different parts are non-adjacent. We call $Y(e)$ the \emph{cutset} for edge $e$, although in some boundary cases $Y(e)$ is not a true vertex cutset in $G_i$.\footnote{We remark that it would be more natural to define $Y(e)$ for instance as $(X_{j} \cap X_{j+1}) \cup \{u_j,u_{j+1}\}$. We employ a slightly larger cutset for technical reasons.} These boundary cases are in fact easier to deal with. The vertices whose interval is either $\{j\}$ or $\{j+1\}$ are said to be \emph{detached} by $Y(e)$. For the edge $e := u_1u_n$, we let $Y(e) := \{u_1,u_n\}$. The rest is similar. Notice that each $Y(e)$ is of size at most $3\adhesionbd(k)+8$. 

Let $\sigma := (e_0;e_1,e_2,\ldots,e_q;e_{q+1})$ be a proper sector, where each directed edge $e_j \in A(D)$ is incident to $f_i \in V(D)$. We assume that, as edges of the boundary cycle of $G_i$, $e_0 = u_\ell u_{\ell +1}$ and $e_{q+1} = u_{\ell+q+1}u_{\ell+q+2}$. We let $G[\sigma]$ denote the union of all connected components of $G_i - Y(e_0) - Y(e_{q+1})$ containing vertices whose interval is contained in the interval $\{\ell+2,\ldots,\ell+q\}$. We think of $G[\sigma]$ as the part of $G_i$ ``between'' $Y(e_0)$ and $Y(e_{q+1})$. 

Now let $\mathcal{S}(e_0) \subseteq Y(e_0)$ and $\mathcal{S}(e_{q+1}) \subseteq Y(e_{q+1})$ be given stable sets such that $\mathcal{S}(e_0) \cap Y(e_{q+1}) = \mathcal{S}(e_{q+1}) \cap Y(e_0)$. Consider all stable sets $S$ of $G[\sigma]$ such that $S \cup \mathcal{S}(e_0) \cup \mathcal{S}(e_{q+1})$ is a stable set for which both $e_0$ and $e_{q+1}$ are slack, but none of $e_1$, $e_2$, \ldots, $e_q$ are slack. We call such stable sets \emph{admissible} with respect to $\mathcal{S}(e_0)$ and $\mathcal{S}(e_{q+1})$. (Possibly, there is no admissible stable set.) Intuitively, these stable sets are all local choices for a feasible solution ``between'' $Y(e_0)$ and $Y(e_{q+1})$ that are compatible with $\mathcal{S}(e_0)$ and $\mathcal{S}(e_{q+1})$ and ``use'' no edge of the boundary ``between'' $e_0$ and $e_{q+1}$.

The \emph{internal cost} of an admissible stable set $S$ (with respect to $\mathcal{S}(e_0)$ and $\mathcal{S}(e_{q+1})$) is the sum of the cost $c(e)$ over all \emph{slack} edges $e \in E(G_i)$ that have at least one endpoint in $V(G[\sigma])$. We denote by $c^{\mathrm{between}}(G[\sigma];\mathcal{S}(e_0),\mathcal{S}(e_{q+1}))$ the minimum internal cost of an admissible stable set $S$. In case there is no admissible stable set, we let $c^{\mathrm{between}}(G[\sigma];\mathcal{S}(e_0),\mathcal{S}(e_{q+1})) := +\infty$. 

Let $\mathcal{W}$ denote a window. Let $\mathcal{S} : e \mapsto \mathcal{S}(e)$ be a map associating to each delimiter $e \in \Del(\mathcal{W})$ a stable set $\mathcal{S}(e) \subseteq Y(e)$. Each stable set $\mathcal{S}(e)$ is a guess on how the optimal stable set should intersect cutset $Y(e)$. We say that $\mathcal{S}$ is \emph{conflict-free} if for each distinct $e, e' \in \Del(\mathcal{C})$ we have $\mathcal{S}(e) \cap Y(e') = \mathcal{S}(e') \cap Y(e)$.

\subsubsection{Reachable windows and their properties}

Let $\mathcal{W}^{\mathbb{F}_0}$ denote the window that has one clock per vortex face $f_i$ (where $i \in [t]$), each clock having a single sector, which is a full sector, and with $\eg(\mathcal{W}^{\mathbb{F}_0}) := g$, where $g \le \vortexbd(k)$ denotes the genus of $\surf$. As the notation indicates, $\mathcal{W}^{\mathbb{F}_0}$ is the window representing the single face $\mathbb{F}_0$ of the initial frame $\tilde{\Sigma}_0$. 

Let $\mathcal{W}$ and $\mathcal{W}'$ denote two windows. We say that $\mathcal{W}'$ is \emph{reachable} from $\mathcal{W}$ if $\mathcal{W}'$ can be obtained from $\mathcal{W}$ after performing \emph{any} number of merge, simplify and split operations. If no confusion may arise, we simply say that $\mathcal{W}'$ is reachable to mean that it is reachable from $\mathcal{W}^{\mathbb{F}_0}$.

\begin{lemma} \label{lem:DP_table_polynomial1}
For every fixed $k \in \Z_{\ge 1}$, the number of reachable windows is polynomial in $m_0 := |A(D)| = |E(G_0)|$.
\end{lemma}

\begin{proof}
Let $\mathcal{W}$ denote any reachable window. By \Cref{lem:h+eg}, we have $\holes(\mathcal{W}) + \eg(\mathcal{W}) \le \holes(\mathcal{W}^{\mathbb{F}_0}) + \eg(\mathcal{W}^{\mathbb{F}_0}) = t + g$. In particular, $\mathcal{W}$ has at most $t + g$ clocks.

Let $C$ be any clock of $\mathcal{W}$. Notice that all the clocks that can be obtained from $C$ by performing any number of split operations (and no merge or simplify operation) are of the form $eCe'$ or $e'C$. This implies that the number of windows that can be obtained from $\mathcal{W}$ by performing any number of split operations is polynomial. More precisely, let $C_1$, \ldots, $C_s$ denote the clocks of $\mathcal{W}$ and $||C_i||$ denote the number of edges in $C_i$ for $i \in [s]$. Notice that $s \leq t + g$ and $||C_i|| \leq  m_0$ for all $i \in [s]$. Hence we can bound the number of sets of clocks that can be obtained from $\mathcal{C}(\mathcal{W})$, via split operations only, by 
$$
\prod_{i=1}^s ||C_i||^2 \leq m_0^{2s} \leq m_0^{2t+2g}\,.
$$
And we can bound the number of windows that can be obtained from $\mathcal{W}$, via split operations only, by $m_0^{2t+2g} (g+1)$. (Recall that a window is a pair $\mathcal{W} = (\mathcal{C}(\mathcal{W}),\eg(\mathcal{W}))$ where $\eg(\mathcal{W}) \in \{0,\ldots,g\}$.)

Consider the sequence of operations that were performed on $\mathcal{W}^{\mathbb{F}_0}$ in order to obtain some reachable window $\mathcal{W}$. We can split this sequence into blocks, where each block has some number (possibly zero) of split operations followed by a single merge or simplify operation. By \Cref{lem:h+eg}, there can be at most $h(\mathcal{W}^{\mathbb{F}_0}) +\eg(\mathcal{W}^{\mathbb{F}_0}) = t + g$ blocks. 

It follows that we can bound the possibilities for $\mathcal{W}$ by 
$$
\sum_{b=1}^{t+g} \left( m_0^{2t+2g} (g+1) 2 m_0^2 \right)^b
\leq \left( m_0^{2t+2g} (g+1) 2 m_0^2 \right)^{t+g+1}
\leq m_0^{\poly(g,t)} = m_0^{O_k(1)}.
$$
Indeed, for any $1\le b\le g+t$ blocks, there are at most $m_0^{2t+2g}(g+1)$ windows that can be obtained from the split operations. We have two choices for the last operation in the block which is either merge or simplify, and at most $m_0^2$ choices for the two directed edges $e, e' \in A(D)$ relative to which the operation is performed. (Recall that the different clocks of a window are edge-disjoint, hence the choice of $e$ and $e'$ also determines the clocks $C$ and $C'$ of the window on which the operation is carried out.)
\end{proof}

\begin{lemma} \label{lem:DP_table_polynomial2}
For every reachable window $\mathcal{W}$, the number of delimiters $\del(\mathcal{W})$ is at most $4t+2g = O_k(1)$.
\end{lemma}

\begin{proof}
Let us call a proper sector $\sigma = (e_0;e_1,\ldots,e_q;e_{q+1})$ {\em almost full} if $e_0 = e_{q+1}$. Let $\full(\mathcal{W})$ be defined as twice the number of full sectors in $\mathcal{W}$, plus the number of almost full sectors in $\mathcal{W}$. 

Every merge or simplify operation increases $\del(\mathcal{W}) + \full(\mathcal{W})$ by at most $2$. The split operations do not increase $\del(\mathcal{W}) + \full(\mathcal{W})$. Since initially, $\del(\mathcal{W}^{\mathbb{F}_0}) + \full(\mathcal{W}^{\mathbb{F}_0}) = 2t$ and $\holes(\mathcal{W}^{\mathbb{F}_0}) + \eg(\mathcal{W}^{\mathbb{F}_0}) = t + g$, it follows from \Cref{lem:h+eg} that $\del(\mathcal{W}) \leq 2(t+g)+2t = 4t+2g$. That is, there are at most $4t+2g$ delimiters in $\mathcal{W}$. 
\end{proof}

\subsubsection{The dynamic programming table}

Let $\mathcal{B}_1 \subseteq \Z_2 \times \Z^{g-1}$ be a polynomial size set such that every $z_0 \in \{0,1\}^{A(D)}$ satisfies $\omega(z_0) \in \mathcal{B}_1$. For instance, we may take $\mathcal{B}_1 := \{0,1\} \times ([-2m_0,2m_0]^{g-1} \cap \Z^{g-1})$, where $m_0 = |A(D)|$. (Recall that each closed walk $W_i$ uses an edge at most twice.) We guess the contribution of $y^{\mathrm{sketch}}_0$ to $\omega(y_0)$. We denote this guess by $\psi^{\mathrm{sketch}} \in \mathcal{B}_1$.

We define a DP table as follows. A general cell of the DP table has the form $(\mathcal{W}, 
\mathcal{S}, \psi, d)$, where $\mathcal{W}$ is a reachable window, $\mathcal{S}$ is a conflict-free map associating to each $e \in \Del(\mathcal{W})$ a stable set $\mathcal{S}(e)$ contained in $Y(e)$, $\psi \in \mathcal{B}_1$ represents the contribution toward satisfying the homology constraint ``inside'' $\mathcal{W}$, and $d \in \mathcal{B}_2 := [-2m_0,2m_0]^t \cap \Z^t$ represents the demand in flow balance ``inside'' $\mathcal{W}$. 

The \emph{top cells} of the DP table are the cells $(\mathcal{W}^{\mathbb{F}_0},\varnothing,\psi^{\mathrm{sketch}},\mathbf{0})$, where $\psi^{\mathrm{sketch}} \in \mathcal{B}_1$.

\begin{lemma} \label{lem:DP_table_polynomial3}
The DP table has polynomially many cells. More precisely, it has $m_0^{O_k(1)}$ cells, where $m_0 = |A(D)| = |E(G_0)|$ as before.
\end{lemma}

\begin{proof}
By Lemma~\ref{lem:DP_table_polynomial1}, there are at most $m_0^{O_k(1)}$ choices for the reachable window $\mathcal{W}$. By Lemma~\ref{lem:DP_table_polynomial2}, the number of possibilities for the guesses in $\mathcal{S}$ is bounded by 
\[
m_0^{4t+2g+1} 2^{(3\adhesionbd(k)+8)(4t+2g)} = m_0^{O_k(1)}
\]
since there can be at most $m_0$ choices for each $e \in \Del(\mathcal{W})$, the size of each $Y(e)$ is at most $3\adhesionbd(k)+8$ and $\del(\mathcal{W}) \le 4t + 2g$. Moreover, the size of each of the sets $\mathcal{B}_1$ and $\mathcal{B}_2$ is bounded by $m_0^{O(t+g)} = m_0^{O_k(1)}$. It follows easily from these observations that the number of cells in the DP table is $m_0^{O_k(1)}$. 
\end{proof}

We say that cell $\xi' := (\mathcal{W}',\mathcal{S}',\psi',d')$ is a \emph{successor} of cell $\xi := (\mathcal{W},\mathcal{S},\psi,d)$ if the following conditions are satisfied. First, $\mathcal{W}'$ is obtained from $\mathcal{W}$ by a single merge, simplify or split operation. For each choice of two distinct internal edges $e^+$ and $e^-$ of $\mathcal{W}$, say $e^+ \in \delta^+(f_{i^+})$ and $e^- \in \delta^-(f_{i^-})$ where $i^+, i^- \in [t]$ (possibly, $i^+ = i^-$), there is a bounded number of choices of what $\mathcal{W}'$ can be, see Section~\ref{sec:operations}. Second, $\mathcal{S}'$ extends $\mathcal{S}$ by defining stable sets $\mathcal{S}'(e^+) \subseteq Y(e^+)$ and $\mathcal{S}'(e^-) \subseteq Y(e^-)$ in the corresponding subgraphs $G_{i^+}$ and $G_{i^-}$. Third, $\psi' \in \mathcal{B}_1$ is arbitrary. Fourth, $d' = d$ in case of a simplify operation, and in case of a merge operation, $d'$ equals $d$ with the component corresponding to $f_{i^+}$ decreased by $1$, and the component corresponding to $f_{i^-}$ increased by $1$. In case of a split operation, $d'$ can be an arbitrary element of $\mathcal{B}_2$. 

We call a cell $\xi := (\mathcal{W},\mathcal{S},\psi,d)$ of the DP table a \emph{bottom cell} if $\xi$ has no successor. This occurs when no operation can be applied to $\mathcal{W}$, because there is no suitable pair $e^+, e^-$ of internal edges in $\mathcal{W}$. 

We define the \emph{dependence graph} of the DP table as the directed graph with one vertex per cell, and directed edges of the form $(\xi,\xi')$ where cell $\xi'$ is a successor of cell $\xi$. The dependence graph has one source vertex per top cell. The sink vertices coincide with the bottom cells.

\begin{lemma} \label{lem:DP_table_acyclic}
The dependence graph of the DP table is acyclic.
\end{lemma}

\begin{proof}
We use once again \Cref{lem:h+eg}, and the fact that most operations decrease $\holes(\mathcal{W}) + \eg(\mathcal{W})$. Notice that if an operation leaves $\holes(\mathcal{W}) + \eg(\mathcal{W})$ unchanged, then it is a split operation, and in most cases $\del(\mathcal{W})$ decreases. In case $\del(\mathcal{W})$ is unchanged, then $||\mathcal{W}||$ decreases. The result follows easily from these observations. 
\end{proof}

Now, we explain how to compute the \emph{value} $\val(\xi) \in \Q_{\ge 0} \cup \{+\infty\}$ of each cell $(\mathcal{W},\mathcal{S},\psi,d)$. We will define $\val(\xi)$ as the minimum of four auxiliary values associated to the cell:
$$
\val(\xi) := \min \left\{ \val^{\mathrm{no-op}}(\xi), 
\val^{\mathrm{merge}}(\xi), \val^{\mathrm{simplify}}(\xi), 
\val^{\mathrm{split}}(\xi) \right\}\,.
$$
We explain how each auxiliary value is computed below.

If $d = \mathbf{0}$ and $\psi = \mathbf{0}$, we define $\val^{\mathrm{no-op}}(\xi)$ as the sum of $c^{\mathrm{between}}(G[\sigma];\mathcal{S}(e_0),\mathcal{S}(e_{q+1}))$ taken on the proper sectors $\sigma = (e_0;e_1,e_2,\ldots,e_q;e_{q+1})$ of $\mathcal{C}(\mathcal{W})$. This corresponds to completing the current solution to a solution where no further boundary edge of a vortex is made slack. That is, we set $y_0(e) := 0$ for each boundary edge $e$ that is still undecided. In case $\psi \neq \mathbf{0}$ or $d \neq \mathbf{0}$, we let $\val^{\mathrm{no-op}}(\xi) := +\infty$.

If $\xi$ is a bottom cell, we let $\val^{\mathrm{merge}}(\xi) :=  \val^{\mathrm{simplify}}(\xi) := \val^{\mathrm{split}}(\xi) := +\infty$. From now on, assume that $\xi$ is not a bottom cell.

Let $e^+ \in \delta^+(f_{i^+})$, and $e^- \in \delta^-(f_{i^-})$, where $i^+, i^- \in [t]$. For $\psi \in \mathcal{B}_1$, we let $\varphi(e^+,e^-,\psi)$ denote the minimum cost of a unit integer flow $\phi$ in $A(D)$ from $f_{i^+}$ to $f_{i^-}$ such that $\phi(e^+) = \phi(e^-) = 1$, $\phi(e) = 0$ for all other edges incident to a vortex face, and $\omega(\phi) = \psi$. These values can be precomputed in strongly polynomial time, see \cite{ocpgenus} or Morell, Seidel and Weltge~\cite{MSW21}.

Assume that both $e^+$ and $e^-$ are internal edges in the corresponding sectors of $\mathcal{C}(\mathcal{W})$. Furthermore, assume that $\mathcal{S}'(e^+) \subseteq Y(e^+)$ and $\mathcal{S}'(e^-) \subseteq Y(e^-)$ are stable sets that extend map $\mathcal{S}$ to a conflict-free map $\mathcal{S}'$ defined on $\Del(\mathcal{W}) \cup \{e^-,e^+\}$. We let $c^{\mathrm{add}}(\mathcal{S}'(e^+),\mathcal{S}'(e^-))$ denote the additional cost incurred by adding $\mathcal{S}'(e^+)$ and $\mathcal{S}'(e^-)$ to our stable set, defined below. When guesses $\mathcal{S}'(e^+)$ and $\mathcal{S}'(e^-)$ are chosen, we also optimize over the vertices that are detached by $Y(e^+)$ or $Y(e^-)$ to figure out which of those vertices should be added to the solution, and which should be left out. The cost $c^{\mathrm{add}}(\mathcal{S}'(e^+),\mathcal{S}'(e^-))$ is defined as the sum of $c(e)$ over the new edges $e$ that are made slack by our decisions. 

We let 
$$
\val^{\mathrm{merge}}(\xi) := 
\min \left(\val(\xi') + \varphi(e^+,e^-,\psi-\psi') + c^{\mathrm{add}}(\mathcal{S}'(e^+),\mathcal{S}'(e^-))\right)
$$
where the minimum is taken over successor cells $\xi' := (\mathcal{W}',\mathcal{S}',\psi',d')$ reached from $\xi$ by performing a single merge operation for a suitable choice of edges $e^+, e^- \in A(D)$ and such that $\psi-\psi' \in \mathcal{B}_1$.

We define $\val^{\mathrm{simplify}}(\xi)$ similarly, with the extra constraint that $e^+$ and $e^-$ are in the same clock.

Finally, we let
$$
\val^{\mathrm{split}}(\xi)
:= \min 
\left( \val(\xi'_1) + \val(\xi'_2) + 
\varphi(e^+,e^-,\psi-\psi'_1-\psi'_2) + c^{\mathrm{add}}_{\leftarrow} (\mathcal{S}'(e^+),\mathcal{S}'(e^-))\right)\,,
$$ 
where the minimum is taken over all pairs $\xi'_1 := (\mathcal{W}'_1,\mathcal{S}'_1,\psi'_1,d'_1)$, 
$\xi'_2 := (\mathcal{W}'_2,\mathcal{S}'_2,\psi'_2,d'_2)$ of successor cells that result from a split operation performed on $\xi$ for a suitable choice of edges $e^+, e^- \in A(D)$ from the same clock $C \in \mathcal{C}(\mathcal{W})$, guesses $\mathcal{S}'(e^+) \subseteq Y(e^+)$ and $\mathcal{S}'(e^-) \subseteq Y(e^-)$ that extend $\mathcal{S}$ to a conflict-free map $\mathcal{S}'$ defined on $\Del(\mathcal{W}) \cup \{e^+,e^-\}$ and $\mathcal{S}'_1$, $\mathcal{S}'_2$ are the restrictions of $\mathcal{S}'$ to $\Del(\mathcal{W}'_1)$ and $\Del(\mathcal{W}'_2)$ respectively, guesses $\psi'_1, \psi'_2 \in \mathcal{B}_1$ such that $\psi - \psi'_1 - \psi'_2 \in \mathcal{B}_1$, and guesses $d'_1, d'_2 \in \mathcal{B}_2$ such that $d'_1 + d'_2 = d^{\mathbb{F}}$. 

The cost $c^{\mathrm{add}}_{\leftarrow} (\mathcal{S}'(e^+),\mathcal{S}'(e^-))$ is defined similarly as $c^{\mathrm{add}}(\mathcal{S}'(e^+),\mathcal{S}'(e^-))$, with the extra cost coming from the part of the clock $C$ that is ``to the left'' of both $e^+$ and $e^-$. Assume without loss of generality that $e^+$ is to the left of $e^-$ in the clock $C$. Then we add to $c^{\mathrm{add}}(\mathcal{S}'(e^+),\mathcal{S}'(e^-))$ the sum of $c^{\mathrm{between}}(G[\sigma];\mathcal{S}'(e_0),\mathcal{S}'(e_{q+1})$ for all sectors $\sigma = (e_0;e_1,e_2,\ldots,e_q;e_{q+1})$ of $Ce^+$.

For every $\psi^{\mathrm{non-sketch}} \in \mathcal{B}_1$, we let $\varphi^{\mathrm{non-sketch}}(\psi^{\mathrm{non-sketch}})$ denote the minimum cost of an integer circulation $y_0^{\mathrm{non-sketch}}$ in $D - f_1 - \cdots - f_t$ with $\omega(y_0^{\mathrm{non-sketch}}) = \psi^{\mathrm{non-sketch}}$, see \cite{ocpgenus,MSW21}.

\begin{theorem} \label{thm:DP}
For each cell $\xi$ of the corresponding DP table, the value $\val(\xi)$ is well defined and can be computed in strongly polynomial time. The minimum cost of a slack vector can be derived in strongly polynomial time as
\begin{equation}
\label{eq:DP}
\min \left(
\val(\mathcal{W}^{\mathbb{F}_0},\varnothing,\psi^{\mathrm{sketch}},\mathbf{0}) + \varphi^{\mathrm{non-sketch}}\left((1,\mathbf{0}) - \psi^{\mathrm{sketch}}\right)\right)\,,
\end{equation}
where the minimum is taken over all guesses $\psi^{\mathrm{sketch}} \in \mathcal{B}_1$ such that $(1,\mathbf{0}) - \psi^{\mathrm{sketch}} \in \mathcal{B}_1$. Moreover, a corresponding optimum solution $y \in \mathcal{Y}(G)$ can be constructed in strongly polynomial time.
\end{theorem}

\begin{proof}
The fact that all values in the DP table are well-defined and can be computed in strongly polynomial time follows directly from Lemmas~\ref{lem:DP_table_polynomial3} and \ref{lem:DP_table_acyclic}. For each cell $\xi$ of the DP table, we can define a subset of zero, one or two successor cells to record how the minimum for $\val(\xi)$ is attained. This partitions the dependence graphs in several out-branchings, one for each top cell. 

Consider the out-branching rooted at the top cell $(\mathcal{W}^{\mathbb{F}_0},\varnothing,\psi^{\mathrm{sketch}},\mathbf{0})$ for the guess $\psi^{\mathrm{sketch}} \in \mathcal{B}_1$ that attains the minimum in \eqref{eq:DP}. Following this out-arborescence down from the root, we can construct a vector $y^{\mathrm{sketch}}_0 \in \Z_{\ge 0}^{E(G_0)} \cong \Z_{\ge 0}^{A(D)}$ and vectors $y_i \in \{0,1\}^{E(G_i)}$ for $i \in [t]$. (The subgraphs $G_0$, $G_1$, \ldots, $G_t$ of $G$ are defined as before, see \Cref{cor:structure_DP}.) 

The vector $y^{\mathrm{sketch}}_0$ is obtained as a sum of unit flows in $D$, each from some $f_{i^+}$ to some $f_{i^-}$, where $i^+, i^- \in [t]$.\footnote{Notice that there might exist one or several directed edges $e \in A(D)$ such that $y^{\mathrm{sketch}}_0(e) \ge 2$. This is completely fine since every circulation in $D$ with the correct homology is a slack vector of $G_0$, which can if needed be converted to a stable set via \Cref{lem:convert}.} Each vector $y_i$ with $i \in [t]$ is obtained as ``disjoint sum'' of contributions within $\{0,1\}^{E(G_i)}$. The final vector $y^{\mathrm{sketch}}_0$ for the out-arborescence is an integer circulation in $D$ with $\omega(y^{\mathrm{sketch}}_0) = \psi^{\mathrm{sketch}}$. For each $i \in [t]$, the final vector $y_i$ is the characteristic vector of some slack set in $G_i$. If we augment $y^{\mathrm{sketch}}_0$ with any circulation $y^{\mathrm{non-sketch}}_0$ realizing the minimum in $\varphi^{\mathrm{non-sketch}}\left((1,\mathbf{0}) - \psi^{\mathrm{sketch}}\right)$, we obtain an integer circulation $y_0 \in \Z_{\ge 0}^{A(D)} \cong \Z_{\ge 0}^{E(G_0)}$ such that $\omega(y_0) = (1,\mathbf{0})$. Moreover, $y_0(e) = y_i(e)$ for each $i \in [t]$ and each edge $e$ of $G_0 \cap G_i$. By Lemma~\ref{lem:composing_solutions} and Proposition~\ref{prop:dual_representation}, we can combine the vectors $y_0$, $y_1$, \ldots, $y_t$ in a single vector $y \in \Z_{\ge 0}^{E(G)}$. Since $y_i \in \mathcal{Y}(G_i)$ for all $i \ge 0$, we get that $y \in \mathcal{Y}(G)$. Moreover, the cost of $y$ equals the right-hand side of \eqref{eq:DP}. Thus, the minimum cost of a slack vector is at most the right-hand side of \eqref{eq:DP}. (Recall that if $y$ is not the characteristic vector of a slack set, then it can be written as a convex combination of such characteristic vectors.)

Finally, consider any minimum cost slack set $F = \sigma(S)$ in $G$, and the corresponding slack vector $y \in \{0,1\}^{E(G)}$. As before, let $y_0$ denote the restriction of $y$ to $E(G_0)$. By considering some animated sketch for $y_0$, we can infer an out-branching in the dependence graph rooted at a top cell of the DP table. This out-branching proves that the minimum cost of a slack vector (which is equal to the cost of $F$) is at least the right-hand side of \eqref{eq:DP}.
\end{proof}

\section{Integer programs with two nonzero entries per column} \label{sec2nonzerospercolumn}

In this section, we provide the proof of Theorem~\ref{thmMain} for the case of two nonzero entries per \emph{column}.
That is, we describe a strongly polynomial time algorithm for integer programs of the form~\eqref{eq:IP} where $A$ is totally $\Delta$-modular for some constant $\Delta$ and contains at most two nonzero entries in each column.
To this end, we describe an efficient reduction to the case where the nonzero entries of $A$ are within $\{-1,+1\}$.
For this case, the respective integer program can be solved in strongly polynomial time by known reductions to the $b$-matching problem due to Tutte and Edmonds, as explained in Schrijver's book:

\begin{theorem}[{\cite[Thm.~36.1]{SchrijverCombOpt}, \cite{EJ70}}] \label{thmIPcolumns}
    There is a strongly polynomial-time algorithm for solving integer programs of the form~\eqref{eq:IP} whose coefficient matrix has at most two nonzero entries per column that are within $\{-1,+1\}$.
\end{theorem}

As a first step, we solve the linear programming relaxation $\max \{ w^\intercal x : Ax \le b\}$.
This can be done in strongly polynomial time using Tardos's algorithm~\cite{Tardos86}.
If the LP is infeasible, then so is the integer program and we are done.
If the LP is unbounded, we may first rerun our algorithm for $w = \zero$.
This may return a feasible integer solution, in which case we can conclude that the integer program is also unbounded (due to the rationality of $A$).
Otherwise, the integer program is infeasible.

Thus, it remains to consider the (main) case in which we have obtained an optimal LP solution $x^*$.
In this case, we partition the matrix $A$ as follows:

\begin{claim}
    In strongly polynomial time, we can efficiently compute $I \subseteq [m]$ and $J \subseteq [n]$ such that
    \begin{enumerate}[(i)]
        \item all entries in the submatrix $A_{[m] \setminus I, [n] \setminus J}$ are within $\{-1,0,1\}$,
        \item $|I| \le 2 \log_2 \Delta$, and $|J| \le \log_2 \Delta$.
    \end{enumerate}
\end{claim}
\begin{proof}
    First, we compute any square submatrix $A'$ of $A$ that is upper-triangular and whose diagonal elements all have absolute value at least 2, and that is not contained in any larger submatrix with these properties.
    Let $J \subseteq [n]$ correspond to the columns of $A'$.
    Note that $\Delta \ge |\det(A')| \ge 2^{|J|}$ and hence $|J| \le \log_2 \Delta$.
    Let $I \subseteq [m]$ denote the union of the supports of the columns (of $A$) indexed by $J$.
    Recall that $A$ contains at most two nonzero entries in each column and hence $|I| \le 2|J|$, which yields~(ii).

    To see that~(i) holds, consider any $i \in [m] \setminus I$, $j \in [n] \setminus J$.
    Let us enlarge the submatrix $A'$ by first adding row $i$, which results in a zero-row since $i \notin I$.
    Thus, by also adding column $j$ we see that $A'$ is properly contained in a square upper-triangular submatrix of $A$ whose additional diagonal element is $A_{i,j}$.
    By the maximality of $A'$ we obtain $|A_{i,j}| \le 1$, which yields the claim.
\end{proof}

As the set $J$ is bounded in terms of $\Delta$, using $x^*$ and Theorem~\ref{thm:Cook_et_al} we can efficiently enumerate the relevant values for all variables indexed by $J$.
Thus, by deleting the respective columns and adapting the right-hand side, we may assume that $J = \emptyset$.

Since we can bound the number of rows of $A_{I,*}$ in terms of $\Delta$ and as all entries of $A$ are within $\{-\Delta,\dots,\Delta\}$, we see that the number of distinct columns in $A_{I,*}$ can also be bounded in terms of $\Delta$.
Thus, we may partition the columns into a constant number of sets $J_1,\dots,J_k \subseteq [n]$ such that $A_{I,J_t}$ consists of identical columns for each $t \in [k]$.
If one of these submatrices is all-zero, we remove the respective set $J_t$.
Note that, given any point $z \in \Z^n$ we only need to know the sum $\sum_{j \in J_t} z_j$ for every $t \in [k]$ in order to determine whether $A_I z \le b_I$.
Again by using $x^*$ and Theorem~\ref{thm:Cook_et_al} we can efficiently enumerate the relevant values $s_t$, $t \in [k]$ for these sums.
Thus, we may assume that we are given $s_1,\dots,s_k \in \Z$ such that any point $x \in \Z^n$ maximizing $w^\intercal x$ subject to
\begin{equation} \label{eqhsj8a1}
    \sum_{j \in J_t} x_j = s_t \qquad \text{for } t = 1,\dots,k
\end{equation}
and
\begin{equation} \label{eqhsj8a2}
    A_{[m] \setminus I} x \le b_{[m] \setminus I}
\end{equation}
is an optimal solution for our original problem, and that every optimal solution for our original problem can be cast in that way.

We claim that every variable appears in at most two of the above constraints, in which case we are done since we can apply the algorithm in Theorem~\ref{thmIPcolumns}.
To this end, note first that every variable appears in at most one of the constraints in~\eqref{eqhsj8a1}, and in at most two of the constraints in~\eqref{eqhsj8a2}, since $A$ contains at most two nonzero entries in each column.
If a variable appears in one constraint in~\eqref{eqhsj8a1}, it also appears with a nonzero coefficient in $A_I$, and hence it appears in at most one constraint in~\eqref{eqhsj8a2}.

\section*{Acknowledgements}
Samuel Fiorini and Yelena Yuditsky are supported by the Belgian National Fund for Scientific Research (FNRS), through PDR grant BD-OCP.
Gwena\"el Joret is supported by a CDR grant and a PDR grant from FNRS.
Stefan Weltge is supported by the Deutsche Forschungsgemeinschaft (DFG, German Research Foundation), project number 451026932. The authors thank Michele Conforti and Tony Huynh for discussions in the early stages of this project and the anonymous reviewers of the FOCS 2021 submission for their helpful comments. 
The authors are also much grateful to the anonymous reviewers of the JACM submission for their careful reading and their numerous comments, which helped improve the exposition of the paper. 

\bibliographystyle{abbrv}
\bibliography{bibliography}

\begin{thebibliography}{10}

\bibitem{Armstrong83}
M.~Armstrong.
\newblock {\em Basic Topology}.
\newblock Undergraduate Texts in Mathematics. Springer New York, NY, 1983.

\bibitem{Artmann}
S.~Artmann, R.~Weismantel, and R.~Zenklusen.
\newblock A strongly polynomial algorithm for bimodular integer linear
  programming.
\newblock In {\em Proceedings of the 49th Annual ACM SIGACT Symposium on Theory
  of Computing (STOC)}, pages 1206--1219, 2017.

\bibitem{BFMR14}
A.~Bock, Y.~Faenza, C.~Moldenhauer, and A.~J. Ruiz-Vargas.
\newblock Solving the stable set problem in terms of the odd cycle packing
  number.
\newblock In {\em 34th {I}nternational {C}onference on {F}oundation of
  {S}oftware {T}echnology and {T}heoretical {C}omputer {S}cience}, volume~29 of
  {\em LIPIcs. Leibniz Int. Proc. Inform.}, pages 187--198. Schloss Dagstuhl.
  Leibniz-Zent. Inform., Wadern, 2014.

\bibitem{BDEHN14}
N.~Bonifas, M.~Di~Summa, F.~Eisenbrand, N.~H{\"a}hnle, and M.~Niemeier.
\newblock On sub-determinants and the diameter of polyhedra.
\newblock {\em Discrete \& Computational Geometry}, 52(1):102--115, 2014.

\bibitem{ocpgenus}
M.~Conforti, S.~Fiorini, T.~Huynh, G.~Joret, and S.~Weltge.
\newblock The stable set problem in graphs with bounded genus and bounded odd
  cycle packing number.
\newblock In {\em Proceedings of the Fourteenth Annual ACM-SIAM Symposium on
  Discrete Algorithms}, pages 2896--2915. SIAM, 2020.

\bibitem{ocp1}
M.~Conforti, S.~Fiorini, T.~Huynh, and S.~Weltge.
\newblock Extended formulations for stable set polytopes of graphs without two
  disjoint odd cycles.
\newblock In {\em International Conference on Integer Programming and
  Combinatorial Optimization}, pages 104--116. Springer, 2020.

\bibitem{CGST86}
W.~Cook, A.~M.~H. Gerards, A.~Schrijver, and {\'E}.~Tardos.
\newblock Sensitivity theorems in integer linear programming.
\newblock {\em Mathematical Programming}, 34(3):251--264, 1986.

\bibitem{cslovjecsek2021block}
J.~Cslovjecsek, F.~Eisenbrand, C.~Hunkenschr{\"o}der, L.~Rohwedder, and
  R.~Weismantel.
\newblock Block-structured integer and linear programming in strongly
  polynomial and near linear time.
\newblock In {\em Proceedings of the 2021 ACM-SIAM Symposium on Discrete
  Algorithms (SODA)}, pages 1666--1681. SIAM, 2021.

\bibitem{Dadush}
D.~Dadush.
\newblock {\em Integer programming, lattice algorithms, and deterministic
  volume estimation}.
\newblock PhD thesis, Georgia Institute of Technology, 2012.

\bibitem{DHK10}
E.~D. Demaine, M.~Hajiaghayi, and K.-i. Kawarabayashi.
\newblock Decomposition, approximation, and coloring of odd-minor-free graphs.
\newblock In {\em SODA 2010: Proceedings of the 21st Annual ACM-SIAM Symposium
  on Discrete Algorithms}, pages 329--344. SIAM, Philadelphia, 2010.

\bibitem{DiestelKMW12}
R.~Diestel, K.~Kawarabayashi, T.~M{\"{u}}ller, and P.~Wollan.
\newblock On the excluded minor structure theorem for graphs of large
  tree-width.
\newblock {\em J. Comb. Theory, Ser. {B}}, 102(6):1189--1210, 2012.

\bibitem{DF94}
M.~Dyer and A.~Frieze.
\newblock Random walks, totally unimodular matrices, and a randomised dual
  simplex algorithm.
\newblock {\em Math. Program.}, 64(1-3):1--16, 1994.

\bibitem{EJ70}
J.~Edmonds and E.~L. Johnson.
\newblock Matching: a well-solved class of linear programs.
\newblock In {\em Combinatorial Structures and their Applications: Proceedings
  of the Calgary Symposium, June 1969}, New York, 1907. Gordon and Breach.

\bibitem{EV17}
F.~Eisenbrand and S.~Vempala.
\newblock Geometric random edge.
\newblock {\em Math. Program.}, 164(1-2):325--339, 2017.

\bibitem{EisenbrandWeismantel}
F.~Eisenbrand and R.~Weismantel.
\newblock Proximity results and faster algorithms for integer programming using
  the steinitz lemma.
\newblock {\em ACM Transactions on Algorithms (TALG)}, 16(1):1--14, 2019.

\bibitem{GGRSV09}
J.~Geelen, B.~Gerards, B.~Reed, P.~Seymour, and A.~Vetta.
\newblock On the odd-minor variant of {H}adwiger's conjecture.
\newblock {\em J. Combin. Theory Ser. B}, 99(1):20--29, 2009.

\bibitem{GoemansRothvoss}
M.~X. Goemans and T.~Rothvoss.
\newblock Polynomiality for bin packing with a constant number of item types.
\newblock {\em Journal of the ACM (JACM)}, 67(6):1--21, 2020.

\bibitem{GKS95}
J.~W. Grossman, D.~M. Kulkarni, and I.~E. Schochetman.
\newblock On the minors of an incidence matrix and its smith normal form.
\newblock {\em Linear Algebra and its Applications}, 218:213--224, 1995.

\bibitem{hemmecke2013n}
R.~Hemmecke, S.~Onn, and L.~Romanchuk.
\newblock N-fold integer programming in cubic time.
\newblock {\em Mathematical Programming}, 137(1):325--341, 2013.

\bibitem{HMNT93}
D.~S. Hochbaum, N.~Megiddo, J.~S. Naor, and A.~Tamir.
\newblock Tight bounds and 2-approximation algorithms for integer programs with
  two variables per inequality.
\newblock {\em Mathematical Programming}, 62:69--83, 1993.

\bibitem{JansenKMR19}
K.~Jansen, K.~Klein, M.~Maack, and M.~Rau.
\newblock Empowering the configuration-ip - new {PTAS} results for scheduling
  with setups times.
\newblock In A.~Blum, editor, {\em 10th Innovations in Theoretical Computer
  Science Conference, {ITCS} 2019, January 10-12, 2019, San Diego, California,
  {USA}}, volume 124 of {\em LIPIcs}, pages 44:1--44:19. Schloss Dagstuhl -
  Leibniz-Zentrum f{\"{u}}r Informatik, 2019.

\bibitem{jansen2020near}
K.~Jansen, A.~Lassota, and L.~Rohwedder.
\newblock Near-linear time algorithm for n-fold ilps via color coding.
\newblock {\em SIAM Journal on Discrete Mathematics}, 34(4):2282--2299, 2020.

\bibitem{GareyJohnson}
D.~S. Johnson and M.~R. Garey.
\newblock {\em Computers and intractability: A guide to the theory of
  NP-completeness}.
\newblock WH Freeman, 1979.

\bibitem{Kannan}
R.~Kannan.
\newblock Minkowski's convex body theorem and integer programming.
\newblock {\em Mathematics of operations research}, 12(3):415--440, 1987.

\bibitem{KR_SODA10}
K.-i. Kawarabayashi and B.~Reed.
\newblock An (almost) linear time algorithm for odd cycles transversal.
\newblock In {\em Proceedings of the Twenty-First Annual ACM-SIAM Symposium on
  Discrete Algorithms}, SODA '10, pages 365--378, Philadelphia, PA, USA, 2010.
  Society for Industrial and Applied Mathematics.

\bibitem{KR_STOC10}
K.-i. Kawarabayashi and B.~Reed.
\newblock Odd cycle packing.
\newblock In {\em Proceedings of the 42nd ACM symposium on Theory of
  computing}, STOC '10, pages 695--704, New York, NY, USA, 2010. ACM.

\bibitem{KTW20}
K.-i. Kawarabayashi, R.~Thomas, and P.~Wollan.
\newblock Quickly excluding a non-planar graph.
\newblock \href{https://arxiv.org/abs/2010.12397}{arXiv:2010.12397}, 2020.

\bibitem{Lenstra}
H.~W. Lenstra~Jr.
\newblock Integer programming with a fixed number of variables.
\newblock {\em Mathematics of operations research}, 8(4):538--548, 1983.

\bibitem{Mohar97}
B.~Mohar.
\newblock Face-width of embedded graphs.
\newblock {\em Mathematica Slovaca}, 47(1):35--63, 1997.

\bibitem{MoharThom}
B.~Mohar and C.~Thomassen.
\newblock {\em Graphs on surfaces}.
\newblock Johns Hopkins University Press, Baltimore, U.S.A., 2001.

\bibitem{MSW21}
S.~Morell, I.~Seidel, and S.~Weltge.
\newblock Minimum-cost integer circulations in given homology classes.
\newblock In {\em Proceedings of the Fifteenth Annual ACM-SIAM Symposium on
  Discrete Algorithms}, pages 2725--2738. SIAM, 2021.

\bibitem{NT74}
G.~L. Nemhauser and J.~L.~E. Trotter.
\newblock Properties of vertex packing and independence system polyhedra.
\newblock {\em Math. Program.}, 6:48---61, 1974.

\bibitem{nemhauser1975vertex}
G.~L. Nemhauser and L.~E. Trotter.
\newblock Vertex packings: structural properties and algorithms.
\newblock {\em Mathematical Programming}, 8(1):232--248, 1975.

\bibitem{Paat}
J.~Paat, M.~Schl{\"o}ter, and R.~Weismantel.
\newblock The integrality number of an integer program.
\newblock In {\em International Conference on Integer Programming and
  Combinatorial Optimization}, pages 338--350. Springer, 2020.

\bibitem{Papadimitriou}
C.~H. Papadimitriou.
\newblock On the complexity of integer programming.
\newblock {\em Journal of the ACM (JACM)}, 28(4):765--768, 1981.

\bibitem{Reed99}
B.~Reed.
\newblock Mangoes and blueberries.
\newblock {\em Combinatorica}, 19(2):267--296, 1999.

\bibitem{RS-GraphMinorsXVI-JCTB03}
N.~Robertson and P.~D. Seymour.
\newblock Graph minors. {XVI}. {E}xcluding a non-planar graph.
\newblock {\em J. Combin. Theory Ser. B}, 89(1):43--76, 2003.

\bibitem{RV90}
N.~Robertson and R.~P. Vitray.
\newblock Representativity of surface embeddings.
\newblock In {\em Paths, Flows, and VLSI-Layout}, pages 293--298.
  Springer-Verlag Berlin, 1990.

\bibitem{SchrijverLPIP}
A.~Schrijver.
\newblock {\em Theory of linear and integer programming}.
\newblock John Wiley \& Sons, 1998.

\bibitem{SchrijverCombOpt}
A.~Schrijver.
\newblock {\em Combinatorial optimization: polyhedra and efficiency},
  volume~24.
\newblock Springer Science \& Business Media, 2003.

\bibitem{seymour95}
P.~D. Seymour.
\newblock Matroid minors.
\newblock In {\em Handbook of combinatorics, {V}ol. 1, 2}, pages 527--550.
  Elsevier Sci. B. V., Amsterdam, 1995.

\bibitem{Tardos86}
E.~Tardos.
\newblock A strongly polynomial algorithm to solve combinatorial linear
  programs.
\newblock {\em Operations Research}, 34(2):250--256, 1986.

\bibitem{tazari12}
S.~Tazari.
\newblock Faster approximation schemes and parameterized algorithms on
  (odd-){$H$}-minor-free graphs.
\newblock {\em Theoret. Comput. Sci.}, 417:95--107, 2012.

\bibitem{VC09}
S.~I. Veselov and A.~J. Chirkov.
\newblock Integer program with bimodular matrix.
\newblock {\em Discrete Optimization}, 6(2):220--222, 2009.

\bibitem{witt2018polyhedral}
J.~T. Witt, M.~E. L{\"u}bbecke, and B.~Reed.
\newblock Polyhedral results on the stable set problem in graphs containing
  even or odd pairs.
\newblock {\em Mathematical Programming}, 171(1):519--522, 2018.

\end{thebibliography}

\end{document}